\documentclass[11pt, reqno]{amsart}
\usepackage[T1]{fontenc}
\usepackage[american]{babel}
\usepackage[toc,page]{appendix}
\usepackage{amssymb,amsthm,amsmath} 
\usepackage{indentfirst}
\usepackage{color}%per scrivere con diversi colori
\usepackage{graphicx}%per inserire grafici
\usepackage{pdfpages}
\usepackage{mathtools}
\usepackage{amsthm}
\usepackage{tikz}
\usetikzlibrary{matrix}
\usepackage{enumerate}
\usepackage{tikz-cd}
\usepackage[framemethod=default]{mdframed}
\usepackage[colorlinks,citecolor=blue,urlcolor=blue]{hyperref}
\usepackage{verbatim}

\usepackage[numbers, sort]{natbib} 
\usepackage{graphicx,color}
\usepackage{epstopdf}
\usepackage[left=1.1in, right=1.1in, top=1.1in, bottom=1.1in]{geometry}
\usepackage{mathrsfs}  
\usepackage{subfig}
\usepackage{booktabs} % for much better looking tables
\usepackage{array} % for better arrays (eg matrices) in maths
\usepackage{paralist} % very flexible & customisable lists (eg. enumerate/itemize, etc.)
\usepackage{verbatim} % adds environment for commenting out blocks of text & for 
\usepackage{fancyvrb}
\usepackage{float}
\usepackage{caption}
\usepackage{bbm} %For blackboard bold 1

\usepackage{url}
\usepackage[toc,page]{appendix}
\usepackage{color}
\allowdisplaybreaks

\usepackage[shortlabels]{enumitem}

\usepackage[colorinlistoftodos, textsize=small]{todonotes}

\newtheorem{thm}{Theorem}[section]
\newtheorem{prop}[thm]{Proposition}

\newtheorem{lemma}[thm]{Lemma}
\newtheorem{cor}[thm]{Corollary}
\newtheorem{assumption}[thm]{Assumption}

\theoremstyle{definition}
\newtheorem{note}[thm]{Note}
\newtheorem{defn}[thm]{Definition}

\AtEndEnvironment{note}{\hfill\qedsymbol}
\counterwithin{equation}{section}

\newcommand{\E}{\mathbb{E}}
\newcommand{\N}{\mathbb{N}}

\newcommand{\R}{\mathbb{R}}

\newcommand{\T}{\mathbb{T}}
\newcommand{\Z}{\mathbb{Z}}

\newcommand{\qq}{\mathcal{Q}}

\newcommand{\cB}{\mathcal{B}}
\newcommand{\cC}{\mathcal{C}}
\newcommand{\cF} {\mathcal F}
\newcommand{\cG}{\mathcal G}
\newcommand{\cH}{\mathcal H}
\newcommand{\cL}{\mathcal{L}}
\newcommand{\cM}{\mathcal{M}}
\newcommand{\cN} {\mathcal N}

\newcommand{\cP}{\mathcal{P}}

\newcommand{\cV}{\mathcal{V}}
\newcommand{\cW}{\mathcal{W}}
\newcommand{\cX}{\mathcal{X}}
\newcommand{\mP}{\mathbb{P}}
\newcommand{\mE}{\mathbb{E}}
\newcommand{\cU}{\mathcal{U}}
\newcommand{\cE}{\mathcal{E}}
\newcommand{\cZ}{\mathcal{Z}}
\newcommand{\cI}{\mathcal{I}}
\newcommand{\ka}{\kappa}

\newcommand{\e}{\varepsilon}
\newcommand{\ep}{\varepsilon}

\newcommand{\pa}{\partial}

\newcommand{\be}{\begin{equation}}
\newcommand{\ee}{\end{equation}}

\newcommand{\iin}{^{i,N}}
\newcommand{\iine}{^{i,N,\varepsilon}}
\newcommand{\jn}{^{j,N}}
\newcommand{\jne}{^{j,N,\varepsilon}}

\newcommand{\itr}{\int_{\T \times \R}}
\newcommand{\itt}{\int_{\T}}
\newcommand{\1}{\mathbbm{1}}
\newcommand{\mY}{Y}
\newcommand{\cY}{\mathcal Y}

\newcommand{\cD}{\mathcal{D}}
\newcommand{\mX}{X}

\newcommand{\Wt}{\cW_{2,t}}
\newcommand{\Ws}{\cW_{2,s}}
\newcommand{\WT}{\cW_{2,T}}
\newcommand{\paxx}{\partial_{xx}}
\newcommand{\noise}{\beta}
\newcommand{\solu}{u}
\newcommand{\solv}{v}
\newcommand{\rrangle}{\right \rangle \right \rangle}
\newcommand{\llangle}{\left \langle \left \langle}

\newcommand{\gs}{\nu}

   \DeclareFontFamily{U}{wncy}{}
   \DeclareFontShape{U}{wncy}{m}{n}{<->wncyr10}{}
    \DeclareSymbolFont{mcy}{U}{wncy}{m}{n}
    \DeclareMathSymbol{\sh}{\mathord}{mcy}{"58}

\usepackage{enumitem,amssymb}
\newlist{todolist}{itemize}{2}
\setlist[todolist]{label=$\square$}

\usepackage{pifont}

\definecolor{mygreen}{HTML}{006622}

\usepackage{etoolbox}

\makeatletter
\newcommand\@todonotes@owner{comment}
\define@key{todonotes}%
    {owner}{\def\@todonotes@owner{#1}}

\newtoggle{ownercomment}
\newtoggle{ownerminor}
\newtoggle{owneredit}

\renewcommand{\@todonotes@addElementToListOfTodos}{%
    \if@todonotes@colorinlistoftodos%
     \addtocontents{todo}
      {%
       \protect\iftoggle{owner\@todonotes@owner}
         {%
          \protect\contentsline {todo}
            {\protect\fcolorbox{\@todonotes@currentbordercolor}%
                {\@todonotes@currentbackgroundcolor}%
                {\protect\textcolor{\@todonotes@currentbackgroundcolor}{o}}%
            \ \@todonotes@caption
            }{\thepage}{\@currentHref}%
          }{}%
       }%
    \else%
      \addtocontents{todo}
      {%
       \protect\iftoggle{owner\@todonotes@owner}
         {%
          \protect\contentsline {todo}
            {\@todonotes@caption
            }{\thepage}{\@currentHref}%
          }{}%
       }%
    \fi}%

\makeatother
\usepackage{graphicx}

\begin{document}
\title[Nonlinear SPDEs via weighted interacting particle systems]{Approximation  of non-linear SPDEs with additive noise via weighted interacting particles systems: the stochastic Mckean-Vlasov equation}

\author{L. Angeli$^{(1)}$}
    \address{$^{(1)}$Mathematics Department, Heriot-Watt University and Maxwell Institute, Edinburgh, letizia.angeli92@gmail.com}

    \author{D. Crisan$^{(2)}$}
    \address{$^{(2)}$ Mathematics Department, Imperial College, London, d.crisan@imperial.ac.uk}

    \author{M. Kolodziejczyk$^{(3)}$}
    \address{$^{(3)}$Mathematics Department, Heriot-Watt University and Maxwell Institute, Edinburgh, mk2006@hw.ac.uk}

    \author{M. Ottobre$^{(4)}$}
    \address{$^{(4)}$ Corresponding Author. Mathematics Department, Heriot-Watt University and Maxwell Institute, Edinburgh, m.ottobre@hw.ac.uk}
    
%\author{L. Angeli, D. Crisan, M. Kolodziejczyk, M. Ottobre}
%\date{June 2023}

\begin{abstract}
This paper is devoted to the problem of approximating non-linear Stochastic Partial Differential Equations (SPDEs) via interacting particle systems. In particular,   we consider the {\em Stochastic McKean-Vlasov equation}, which is the McKean-Vlasov (MKV) PDE (also known in other contexts as the porous media equation), perturbed by {\em additive} trace class noise. As is well-known,  the MKV PDE can be obtained as mean field limit of the empirical measure of a stochastic system of interacting particles, where particles are subject to independent sources of noise.   If the particles are  subject  to the same noise (often referred to as  {\em common noise}), as opposed to independent sources of noise,   the limit of the associated empirical measure is no longer deterministic,  and it is  described by an SPDE. The SPDE emerging in this context can be viewed as a stochastic  perturbation of the MKV PDE, where the noise structure is   {\em multiplicative} and in gradient form.  There is now a natural question, which is the one we consider and answer in this paper: can we obtain the SMKV equation, i.e. additive perturbations of the MKV PDE,  as limit of interacting particle systems? It turns out that, in order to obtain  the SMKV equation,  one needs to study {\em weighted}  empirical measures of particles, where the particles evolve according to a system of SDEs with independent noise, while the weights are time evolving and subject to common noise. 

Besides achieving the immediate goal of constructing a system of interacting particles which converges to the SMKV equation, the approach used in this work is potentially suggestive of ways to construct interacting particle systems to approximate other non-linear SPDEs, such as (variations of) Dean-Kawasaki,  Keller-Segel-Dean-Kawasaki in particular. Furthermore, dealing with additive noise structures required new ideas of proof,  and these ideas have the potential of being  easily adapted to more general sources of noise, such as  jump processes.   
The work of this manuscript therefore complements and contributes to various streams of literature, in particular: i) much attention in the community is currently devoted to obtaining SPDEs as scaling limits of appropriate dynamics; this paper contributes to a complementary stream, which is devoted to obtaining representaions of SPDE through limits of empirical measures of  interacting particle systems;  ii) since the literature on limits of weighted empirical measures is often constrained to the case of static (random or deterministic) weights, this paper contributes to further expanding this line of research to the case of time-evolving weights. 

\smallskip
{\bf Keywords.} Stochastic Partial Differential Equations, Interacting Particle Systems, Stochastic McKean-Vlasov Equation, Particle Approximations to Stochastic Partial Differential Equations. 

\smallskip
{\bf AMS Subject Classification. } 35Q83, 35Q70, 60H15, 35R60, 65C35, 82M60.  
\end{abstract}

\maketitle

\section{Introduction}
This paper is concerned with the problem of obtaining non-linear Stochastic Partial Differential equations (SPDEs) with {\em additive} noise as limits of appropriate interacting particle systems. 
More precisely, we consider  nonlinear McKean-Vlasov Partial differential Equations (PDEs)   with stochastic or deterministic forcing; that is, evolutions of the following type
\begin{equation}\label{generic PDE+forcing}
\pa_t \gs_t(x) = \paxx \gs_t(x) + \pa_x \left[ (V'(x) + (F' \ast \gs_t)(x)) \gs_t(x)\right] +  \pa_t \,\mathcal U_t
\end{equation}
where the unknown $\gs=\gs_t(x)$ is a real-valued function, \footnote{Throughout this 
		paper, for any quantity, say $Y$, that depends on time, we use interchangeably the notation $Y_t$ or
  $Y(t)$ to denote time-dependence and we use the notation $Y_{\cdot}$ or $t \mapsto Y_t$ when we want to refer to the whole path. } $\gs: \R_+ \times \mathbb T \rightarrow \R$, with $\mathbb T$ the one dimensional torus of length $2 \pi$;   $V$ and $F$  are given smooth functions,  $V, F\in C^{\infty}(\T;\R)$,  usually referred to as the {\em environmental} and {\em inter-particle} potentials, respectively (below the reason for this nomenclature);  $'$ denotes derivative with respect to the spatial variable and $\mathcal U$ is a forcing term, which could depend on time only, or on both space and time. We will consider three cases: $\mathcal U_t = \mathcal U(t,x)= q(x) Y_t$  where $q$ is a smooth function on the torus and $Y_t$ is a   $\alpha$-H\"older continuous time-dependent  function,   $\mY_{\cdot} \in C^{\alpha}([0,T];\R)$, for some $\alpha \in (0,1)$;\footnote{We recall that $C^{\alpha}([0,T];\R)$ is the space of real valued paths $f$ on the interval $[0,T]$,  $[0,T]\ni t\mapsto f(t) \in \R$,  for which there exists a constant $C>0$ such that $|f(t)-f(s)|\leq C |t-s|^{\alpha}$, for every $t,s \in [0,T]$. } $\mathcal U_t = q(x) w_t$ where $w_t$ is a one-dimensional standard Brownian motion (and then, by easy extension,  a finite sum of terms of this form); $\mathcal U$ is a (weighted) infinite sum of Brownian motions, namely  $\mathcal U = W(t,x)$, where  $W$ is  trace class Wiener noise,
\begin{equation}\label{infinite dimensional noise}
      W(t,x) := \sum \limits_{z \in \Z} \lambda_z e_z(x) w^z_t \, ,\footnote{The sum is indexed by $z\in \mathbb Z$ rather than by natural index simply because of the way we choose the basis $\{e_z\}_{z \in \mathbb Z}$ of $L^2(\T;\R)$ later in the paper, so this is a purely aesthetic choice.}
\end{equation}
 where $\{e_z\}_{z \in \Z}$ is an orthonormal basis of $L^2(\T;\R)$, $\{\lambda_z\}_{z \in \Z} \subset \R_+$ is a  (non-identically zero) square-summable sequence of non-negative real numbers  and the $w_t^z$'s are standard one-dimensional independent Brownian motions. To be more explicit, we will consider the following three evolutions: 
 \begin{equation}
\begin{dcases}\label{rough PDE}
   & \pa_t \rho_t (x)= \pa_{xx} \rho_t(x) + \pa_x \left [ (V^{'}(x)+(F^{'}*\rho_t)(x)) \rho_t(x) \right ]+q(x) \, \pa_t\mY_t, \quad  t \in (0,T),\,x \in \T, \\ 
   & \rho|_{t=0}=\rho_0(x), \qquad x \in \mathbb T \, ,
\end{dcases}
\end{equation}
 where $Y_{\cdot}$  is a $\alpha$- H\"older path, 
\begin{equation}\label{pde+bm}
    \begin{dcases}
        & \pa_t\solu_t(x)=\pa_{xx}\solu_t(x) + \pa_x \left [ (V^{'}(x)+(F^{'}*\solu_t)(x)) \solu_t(x) \right ]+q(x) \,\pa_tw_t,\quad t \in (0,T), x \in \T, \\ 
   & \solu|_{t=0}=\solu_0(x), \quad x \in \T \, ,
    \end{dcases}
\end{equation}
where $w_t$ is a one-dimensional real-valued standard Brownian motion, and finally
	\begin{align}\label{SPDEintro}
        \begin{dcases}
		& \pa_t v_t(x) = \pa_{xx}v_t(x)+\pa_x \left[ (V^{'}(x)+(F^{'}*v_t)(x)) v_t(x) \right ]+  \,  \pa_t W_t,\quad  t \in(0,T), x \in  \T  \\
		%& v(t,0) =v(t,2\pi),\qquad t \in [0,T]  \\
		& v|_{t=0}=v_0(x), \qquad x \in \mathbb T \,, 
	  \end{dcases}
    \end{align} 
 where the trace class noise $W$ is as in \eqref{infinite dimensional noise}. In the first case,  when $\mathcal U = q(x)Y_t$, the resulting equation \eqref{rough PDE} is deterministic - it is simply a PDE with a deterministic forcing term, so the unknown $\rho=\rho(t,x): \mathbb R_+ \times \mathbb T \rightarrow \mathbb R$ is a deterministic function of time and space.   In the second and third case, \eqref{pde+bm} and \eqref{SPDEintro}, respectively,  the resulting evolutions are  SPDEs. The unknowns  $u=u(t,x): \mathbb R_+ \times \mathbb T \rightarrow \mathbb R$ and $v=v(t,x): \mathbb R_+ \times \mathbb T \rightarrow \mathbb R$ are random functions of space and time and, as customary, in the notation we omit the dependence of $u$ and $v$ on the realization $\omega$ in the underlying probability space. 

 Well-posedness and long-time behaviour of  \eqref{SPDEintro} have been investigated in \cite{angeli2023well}. The well-posedness of \eqref{SPDEintro} when $W$ is cylindrical Wiener noise is still an open problem, and this is the main reason why in this paper we consider trace class noise only. The well-posedness of  \eqref{rough PDE} and \eqref{pde+bm} can be obtained similarly to the well-posedness of \eqref{SPDEintro}, see Section \ref{section: main results}.

 The purpose of this paper is to find particle approximations to the above dynamics, i.e. to find appropriate particle systems which converge, in a sense that we will specify below,  to the solution of \eqref{generic PDE+forcing}. To explain what spurred our interest in evolutions of type \eqref{generic PDE+forcing} (which may a priori seem simpler than their counterparts with multiplicative noise), let us start by introducing the following system of $N$ interacting particles
	\be\label{initialPS}
	dX_t\iin= \left[ - V'(X_t\iin) + \frac{1}{N}\sum \limits_{j=1}^N F'(X_t\jn-X_t\iin)\right]  dt+ \sqrt{2} d\beta_t^i, 
	\ee
	where, for every $i=1\dots N$, $X_t\iin \in \T$ represents the position of the $i-$th particle at time $t>0$. The Interacting particle system (IPS) \eqref{initialPS} has been extensively studied in the literature, as a model for collective dynamics and consensus formation, see e.g. \cite{dawson1983critical, herrmann2010non, delgadino2021diffusive, carrillo2020long, bertini2014synchronization} and references therein.  The potential  $F$ describes the interaction between particles,  ({\em inter-particle potential}), while $V$, which is the same for each particle (it does not depend on $i$) is commonly used to describe environmental effects ({\em environmental potential});  the $\beta_t^i$'s are independent one-dimensional standard Brownian motions. When the number $N$ of particles in the system is large, it is natural to look for a macroscopic (coarse-grained) description of system \eqref{initialPS}. To do so, according to an established methodology in statistical mechanics, one considers the {\em empirical measure} $h_t^N$ associated to the particle system \eqref{initialPS},  $h^N_t:= \frac1 N \sum
	\delta_{X_t^{i,N}}$  , which is,  for each $t>0$,  a random probability measure. Then, in the limit  $N \rightarrow \infty$ (generally referred to as {\em mean field limit}), $h_t^N$ converges weakly in $C([0,T];\R)$ (but results on infinite time horizons do hold true in certain cases as well, see e.g. \cite{durmus2020elementary, heydecker2019pathwise, barre2021fast} and references therein) to the solution   of the following  non-linear (and  non-local) PDE  \begin{equation}\label{PDEsimpleparticlesystem}
		\pa_t h_t(x)= \sigma \pa_{xx} h_t(x) +\pa_x \left[\left(V'(x) + (F'\ast h_t)(x)\right) h_t(x)\right]\,,     
	\end{equation}
	for the unknown $h=h(t,x):\R_+ \times \T \rightarrow \R$,
	provided this is true for the corresponding initial data, i.e provided $h_0^N$ converges to 
	$h_0$, see \cite{graham1996asymptotic, lacker2018mean}. In short, we say that the PDE \eqref{PDEsimpleparticlesystem} is obtained as limit of the IPS \eqref{initialPS}, meaning that the empirical measure $h^N$ associated to \eqref{initialPS} converges to the solution $h$ of the PDE \eqref{PDEsimpleparticlesystem}.  The convergence  of $h^N$ to $h$  means that the particle system can be used as an approximation of the limiting PDE. Note that while the particle system \eqref{initialPS} is stochastic, the limit \eqref{PDEsimpleparticlesystem} is deterministic, fact  that can be viewed as a manifestation of the law of large numbers.   Another, more pathwise perspective on this is the following: instead of looking at the convergence of the empirical measure, one can show that,  as 
	$N\rightarrow \infty$, the particles  become independent ({\em propagation of chaos}) and, in the 
	limit, the motion of each of them is described by the following Stochastic Differential Equation (SDE)
	\begin{equation*}\label{non-linearSDE}
	dX_t  = - \left(V'(X_t) + \int_{\T} F'(y-X_t)h_t(dy)
	\right) dt+ \sqrt{2} d\beta_t \,,  \quad X_t \in \T\, ,
	\end{equation*}
	where $\beta_t$ is a one-dimensional standard Brownian motion and $h_t$ is the law of $X_t$ at time $t$, so that the above evolution is non-linear in the sense that the process depends on its own law, i.e.~it is {\em non-linear in the sense of McKean}. The law $h_t$ of $X_t$ can then be shown  to be  the solution to the PDE \eqref{PDEsimpleparticlesystem}.  For later purposes we note that the PDE \eqref{PDEsimpleparticlesystem} preserves mass, i.e. (the solution is positive if started from a positive initial datum) and the value of the integral $\int_{\T}h_t(x) \, dx$ is constant in time. This is compatible with the fact that during the evolution \eqref{initialPS} the number of particles is preserved.

  Equation \eqref{generic PDE+forcing} is obtained by adding a forcing term  to the PDE \eqref{PDEsimpleparticlesystem}; for this reason, when the forcing $\mathcal U$ is either of the form $q(x)w_t$ or a trace class Wiener noise,  we refer to the resulting equations, equations \eqref{pde+bm} and \eqref{SPDEintro} respectively,  (the latter being the one that initially motivated this work)    as  to {\em McKean-Vlasov SPDEs} or, more accurately, as  {\em Stochastic McKean-Vlasov equations} (SMKV).\footnote{The former name could be misleading so we clarify that the solution to \eqref{SPDEintro}, seen as a function-space-valued process, is not an infinite dimensional McKean-Vlasov SDE, as the process does not depend on its own law. The investigation of infinite-dimensional McKean-Vlasov SDEs has been recently tackled in \cite{hong2022strong, hong2021distribution}, see   \cite{angeli2023well} for a more thorough discussion on this.}   We mention in passing that  our interest in the evolution \eqref{SPDEintro} started when considering problems in infinite dimensional sampling (which are  unrelated to the work of this paper,  but we mention nonetheless for context). The flavour of this initial motivation is not dissimilar to the one in the very recent study \cite{delarue2024rearranged}, which treats related types of equations.  After the seminal paper \cite{dawson1995stochastic}, stochastically perturbed McKean-Vlasov dynamics (though not necessarily with additive noise) have also been considered in \cite{cardaliaguet2019master, carmona2018probabilistic, kolokoltsov2019mean, maillet2023note, debussche2020existence} (to mention just a few works on the matter and  without claim to completeness of references) as such equations naturally arise  in the context of mean field games or, potentially, as models in mathematical biology.

% The literature on interacting particle systems and their corresponding PDE limits is truly vast,  including both theoretical developments (with recent efforts to try and treat the setting on non-exchangeable particles [....]) and a pletora of application fields, where mean field theory has proved extremely successful [...].
 %The same can be said about the body of work on particle approximations to Stochastic PDEs. Let us review this literature shortly, selecting the works that are closest to the present paper or that best help us explain the nature of the problem at hand and the difficulties involved. 

  If one is interested, as we are,  in obtaining, in the limit, a stochastic PDE, a useful framework to consider is the one of IPSs subject to  {\em common noise}, which is a relatively simple variation of the mean field paradigm  we have just summerised.  A simple instance of the common noise setting can be seen by considering the following IPS 
 \begin{equation}\label{PSwithcommonnoise}
	dX_t^{i,N}= - V'(X_t^{i,N}) dt + \frac{1}{N}\sum \limits_{j=1}^N F'(X_t^{j,N}-X_t^{i,N})  dt+ \sqrt{2} d\beta_t^i+ \sqrt{2\tilde\sigma} dw_t,  
\end{equation}
	where, crucially, the one-dimensional Brownian noise $w_t$  is {\em the same} for each particle (i.e. it is the `common noise'), and  $\tilde\sigma>0$. 
	Because the  same noise $w_t$ acts on all the particles, the limit of the particle system \eqref{PSwithcommonnoise} is no longer deterministic, and it is stochastic instead \cite{lacker2018mean, carmona2018probabilistic}. In particular, the limit  of the empirical measure  associated with \eqref{PSwithcommonnoise} is given  by the following SPDE
 \begin{align}\label{SPDEPScomnoise}
		\pa_t h_t(x)  = \pa_x \left[\left(V'(x) + (F'\ast h_t)(x)\right) h_t(x) + (\sigma+\tilde\sigma) \pa_x h_t(x)\right] - \sqrt{2 \tilde\sigma} (\pa_x h_t)d w_t \,.
	\end{align}
The limiting equation \eqref{SPDEPScomnoise} is in this case interpreted as the conditional law of the following McKean-Vlasov evolution
\begin{equation*}\label{non-linearSDEcommon noise}
	dx_t  = - \left(V'(x_t) + \int_{\T} F'(y-x_t)h_t(dy)
	\right) dt+ \sqrt{2} d\beta_t +  \sqrt{2\tilde\sigma} dw_t\,,  \quad X_t \in \T\, ;
	\end{equation*}
that is,  in the above $h_t$ is the conditional law of $x_t$,  conditional on the path $\{w_t\}_{t\geq 0}$; in short $h_t = \mathcal L (x_t \vert w)$, see \cite{lacker2018mean, carmona2018probabilistic}. In the common noise setting the common noise is treated as a given forcing or, to put it differently, as an extra, given, time-dependent drift,  hence the conditioning. From  a modelling perspective, the common noise is often use to describe `shocks' to the system, which affect each particle, see e.g. \cite{ledger2021mercy, carmona2018probabilistic}. 
 
The common noise setting will be conceptually useful to us but the dynamics \eqref{SPDEPScomnoise} is not what we wanted to obtain. 
For the purposes of this discussion, the main  difference between \eqref{SPDEintro} (or more, generally, \eqref{generic PDE+forcing}) and \eqref{SPDEPScomnoise} is that in the former equation noise acts additively on the solution (the forcing $\mathcal U_t$  does not depend on the solution  itself), while in the latter noise is multiplicative; moreover, \eqref{SPDEPScomnoise} has transport (gradient) structure, while \eqref{generic PDE+forcing} does not - fact that is the source of many complications. Notice indeed that both the drift (the PDE part) and the diffusion term  of \eqref{SPDEPScomnoise} are in gradient form.  Hence, by  integration by parts, one can formally see that 
 the gradient structure of \eqref{SPDEPScomnoise} implies that mass is preserved along the solution of \eqref{SPDEPScomnoise}; that is, if $h$ is the solution of \eqref{SPDEPScomnoise}, then
 $
\int_{\T} h_t(x) \, dx = C,   
 $ almost surely, 
 where $C$ is a constant independent of time. Notice again that also  in the IPS \eqref{PSwithcommonnoise} the number of particles is conserved (particles are neither killed nor reweighted).

However equation \eqref{generic PDE+forcing}  {\em does not} in general preserve mass  and it does not preserve 
 positivity either, so one {\em cannot expect} to obtain \eqref{generic PDE+forcing} (or any of its incarnations \eqref{rough PDE}, \eqref{pde+bm}, \eqref{SPDEintro}) as limit of 
 empirical (probability) measures, similarly to what we have described above for 
 \eqref{initialPS}/\eqref{PSwithcommonnoise} and their corresponding limit \eqref{PDEsimpleparticlesystem}/\eqref{SPDEPScomnoise} (where no killing/reweighting of 
 particles takes place). Intuitively, one expects that, in order to obtain \eqref{generic 
 PDE+forcing}, particles will need to be reweighted in some way. 

 In this paper we will indeed show that equations of the form \eqref{generic PDE+forcing} can be obtained as 
 limits of a weighted interacting particle system.  The weights we will use are not ``static weights" (i.e. fixed, independent of time); on the contrary, the evolution of the weights is coupled to the evolution of the particles and it is described by a system of stochastic differential equations.  
 In the particle-weight system that we will design the particles are driven by independent sources of noise, while the evolution of the weights is subject to common forcing (common noise in the stochastic case of \eqref{pde+bm} and \eqref{SPDEintro}). 
In Subsection \ref{subsec:heuristics and informal statement} below we 
 give a heuristic argument to explain how the particle-weight system which converges to \eqref{SPDEintro} is constructed, and provide a slightly simplified, informal version of the particle-weight system we will consider, see \eqref{informalpartsys1}-\eqref{informalpartsys2}. In Section \ref{section: main results} we present full results.

 The inspiration to construct our particle-weight system comes from \cite{kurtz1999particle, crisan2018particle} (see also references therein), which are devoted to finding particle representations of non-linear SPDEs (we will explain the difference between particle {\em representation} and particle {\em approximation} in Note \ref{Note1}).  The class of non-linear SPDEs considered in \cite{kurtz1999particle, crisan2018particle}  includes non-linearities of  Allen-Cahn or Fisher KPP  type and moreover \cite{kurtz1999particle} treats multiplicative equations only.   In particular, the class of SPDEs considered in  \cite{kurtz1999particle, crisan2018particle} does not include  equations of the type \eqref{generic PDE+forcing}. Coping with equations of type \eqref{generic PDE+forcing} requires new ideas, and fundamentally different approaches in the proofs.   We will make a more thorough 
 comparison with  \cite{kurtz1999particle, crisan2018particle}  in Note \ref{Note1} and  explain the full strategy of proof in Section \ref{strategy of proof}, here we  give  a flavour.

 Our interest was initially in studying the Stochastic McKean-Vlasov equation i.e. equation \eqref{SPDEintro}; one then quickly realises that, from the point of view of deriving particle approximations, the infinite dimensional nature of the noise in \eqref{SPDEintro} does not pose any particular difficulty, so that it is convenient ( to set up proofs, to simplify notation  and for understanding) to start by considering equation \eqref{pde+bm}, which is forced by a simple one dimensional Brownian Motion,  and deriving particle approximations to such an evolution. Our strategy 
 of proof strongly relies on fixing a realization  of the Brownian path $w_t$ driving \eqref{pde+bm} and on viewing \eqref{pde+bm} as a specific instance of \eqref{generic PDE+forcing}, when  the  forcing in \eqref{generic PDE+forcing} is given by such a realization. Because of the H\"older regularity of Brownian motion, with this approach  one is led to considering PDEs of the form \eqref{rough PDE}, where the forcing $Y_t$ has the same H\"older continuity of Brownian motion, but it is  fixed deterministically. Since $\alpha$-H\"older continuous functions can be approximated via piece-wise smooth functions,   all our results are obtained by considering a sequence $\{Y^{\kappa}_t\}_{\kappa \in \N}$ of 
 piecewise smooth,  time-dependent deterministic functions which converge (in appropriate sense, see Section 
 \ref{section: main results}) to $Y_t$.  We first prove all the statements of interest when  
 \eqref{generic PDE+forcing} is driven by a piecewise smooth forcing, i.e. $\mathcal U= q(x)Y^{\kappa}_t$ 
 and then, {\em at the appropriate point in the proof}, we take the limit in $\kappa$. Some of the technical advantags of doing all proofs by ``fixing the driving path"
 are explained as we move along the proof,  in Note \ref{note:cutoff}, Note \ref{note:commutativity limits} and Note \ref{note importante}. In few words, if the driving force is smooth then no It\^o second order correction terms arise from the chain rule. In our case these terms would be quite complicated to handle.     An alternative, intuitive way to explain why this approach can succeed comes from the observation that since the equations 
 we are studying are subject to additive noise,  their It\^o and Stratonovich interpretation coincide; hence,  the 
 reason why this approach works can then  be seen in the general framework  of Wong-Zakai approximation theory \cite{twardowska1996wong, karatzas2014brownian}.  Moreover, since the Brownian path $w_t$ in \eqref{pde+bm}  (similarly for $W$ in \eqref{SPDEintro})   will be precisely the common noise acting on the weights of our particle-weight system (see \eqref{informalpartsys2} and \eqref{weights_simple}),  this perspective is coherent with the fact that, when dealing with common noise (see comments after equation \eqref{SPDEPScomnoise}), it is customary to carry out all proofs by `conditioning on the forcing common noise'. Related to this, another, perhaps more pragmatic way of seeing the approach of fixing the noise, is the following: we will construct \eqref{rough PDE} as (weighted) marginal of a certain joint distribution; of course, there are many joint distributions which would give the same marginal. From this perspective, fixing the noise corresponds to choosing a convenient joint distribution - convenient for the purposes of simplifying the proof, see Note \ref{note:cutoff}.

 However 
 these are just  intuitive explanations. Indeed, 
 complications come from the fact that, while the equation in which we are interested is subject to additive noise, 
 other quantities of interest in the proof satisfy equations which are driven by multiplicative noise (see Section 
 \ref{section: main results}). So taking the limit of the regularization parameter at the appropriate 
 moment is crucial,  to avoid dealing with complicated second order Stratonovich correction terms.  
Finally, our proofs do not require any rough path theory; nonetheless the paper \cite{coghi2021rough} has been of great inspiration -  the methods of proof we use  here are very different, but the inspiration of trying to simplify proofs by ``fixing the driving path" has certainly come from \cite{coghi2021rough}.

 We do not claim that this is the only possible approach of proof, but we argue for its simplicity. 
 The approach of fixing the path has two perks: on the one hand it is extremely helpful in overcoming 
 a number of technical hurdles; on the other hand it allows in principle 
 to produce results for SPDEs with many types of noise - whichever noise has paths that can be 
 approximated with time-dependent paths of appropriate smoothness, as long as the corresponding 
 equation \eqref{generic PDE+forcing} is well posed. We indeed believe that the same approach would work if the forcing was fractional Brownian motion or a jump process, as e.g. in \cite{pei2020pathwise}.  

\noindent

 %%%%%%%%%%%%%%%%%aggiusta
 Our results contribute to broadening the literature on limits of  weighted empirical measures as well. The literature on this topic is large and a full review is out of the scope of this paper, so we mention only some works, with no claim to completeness of reference. Beyond the mean field regime (where all weights are fixed and equal to one) interest has been attracted by the case in which the weights are  fixed, deterministic or random, as e.g. in the case of Erd\"os graphs (see \cite{coppini2023central}, which contains a very thorough review).  Motivated by applications e.g. to mathematical biology, epidemiology or opinion formation \cite{gross2008adaptive, gross2006epidemic}, time-evolving weights (representing an evolving network of connections), possibly coupled to the particles' evolution,  have also been investigated, see \cite{nugent2023evolving, barre2021fast} and references therein.
However, aside from the  already mentioned works \cite{kurtz1999particle, crisan2018particle},  we are not aware of other references where the weights are coupled to the particles and subject to common noise.

A final note on connection with broader literature: as is well known, much  attention in the community is currently 
 devoted to  
 deriving SPDEs as {\em scaling limits} of appropriate processes -   we refer the reader to \cite{berglund2022introduction} and references therein for a nice and accessible review. In this case one seeks to obtain the limiting equation via appropriate rescaling of space and time. One can also ask the question of whether it is possible to obtain  SPDEs as limits of interacting particle systems, in the sense we have explained in this introduction. This is the question that the works \cite{kurtz1999particle, crisan2018particle}, see also  references therein,  and the 
 present paper are trying to address. 
 
  In this context we point out that the framework of this work could be used to inspire particle approximations to (versions of) the Dean-Kawasaki equation, see e.g. \cite{fehrman2021well, cornalba2019regularized, dean1996langevin}  or \cite{martini2022additive}, on the Keller-Segel-Dean-Kawasaki equation.  This will be object of future work. 
 
In Subsection \ref{subsec:heuristics and informal statement} below we give a heuristic argument for the choice of particle-weight system.  Precise statements of main results, strategy of proof and organization of the paper are described in Section \ref{section: main results}. 

 \subsection{Informal statement of main result. }\label{subsec:heuristics and informal statement} 
 We start by giving an ansatz for the particle-weight  system that converges to \eqref{SPDEintro}. We then state an informal version of the main result of this paper and  provide a heuristic calculation  to show  why the result is expected to  hold.  
 
 Consider the following system,  

\begin{align}
     dX\iine_t \, &=\, - V'(X_t\iine) dt - \frac{1}{N}\sum \limits_{j=1}^{N} A\jne_t\, F'(X_t \iine-X_t^{j,N, \ep})\,  dt+ \sqrt{2} d\beta_t^i \label{informalpartsys1}\\    
      dA\iine_t\, &=\, \frac{\sum \limits_{z \in \mathbb Z} \lambda_z e_z(X_t^{i,N,\ep}) dw_t^z}{\frac{1}{N}\sum \limits_{j=1}^N\Phi_{\e}(X\iine_t-X^{j, N, \varepsilon}_t)} \,, \label{informalpartsys2}
\end{align}
where  $i \in \{1, \dots, N\}$ and, for each $i$,  $(X\iine_t, A_t\iine) \in \T \times \R$. One should think of $X_t\iine$ as representing the position of the $i$-th particle at time $t$, and $A\iine_t$ as representing a weight. The $\beta_t^i$'s are one-dimensional Brownian motions,  independent of each other and of the one-dimensional Brownian motions $\{w_t^z\}_{z \in \mathbb Z}$; the $\lambda_z$'s and $e_z$'s are as in \eqref{infinite dimensional noise}.  The functions $\Phi_{\varepsilon}: \T \rightarrow \R$ are `mollifiers', i.e. they are a family  (indexed by $\varepsilon>0$) of smooth functions, converging to the Dirac delta at zero, strictly positive for each $\varepsilon>0$ fixed, and normalised so that $\int_{\T}\Phi_{\varepsilon}(x) dx =1$, the role of which is explained in few lines. The particle-weight system \eqref{informalpartsys1}-\eqref{informalpartsys2} will be used to obtain \eqref{SPDEintro}, see Theorem \ref{thm:mainthm} below. 

The reader who is unfamiliar with infinite dimentional noise may find it simpler to focus on equation \eqref{pde+bm} rather than on \eqref{SPDEintro}. This is not an oversimplification and in fact the core of the problem we will discuss is not related to using  infinite dimensional noise, but rather to the fact that the noise we consider 
is additive and to the specific nature of the nonlinearity we are trying to obtain.   Note indeed that equation \eqref{pde+bm}
 is substantially a specific instance of \eqref{SPDEintro} when $\lambda_0=\sqrt{2\pi}$, $e_0(x)=(2\pi)^{-1/2}$ and $\lambda_z=0$ for every $z \neq 0$. So, to obtain the SPDE \eqref{pde+bm} instead of \eqref{SPDEintro} just consider \eqref{informalpartsys1} coupled to the  weights $\tilde{A}\iine_t$, evolving according to 
\begin{equation}\label{weights_simple}
d \tilde{A}\iine_t\, =\, \frac{q(X_t^{i, N, \e})}{\frac{1}{N}\sum \limits_{j=1}^N \Phi_{\e}(X\iine_t-X^{j, N, \varepsilon}_t)} dw_t \,, 
\end{equation}
rather than to the weights \eqref{informalpartsys2}. In this simpler case  one can more directly see that, while the particles $X_t^i$ are subject to independent sources of noise (the $\beta_t^i$'s), the weights $A_t^i$ are subject to the {\em same} noise (the Brownian $w_t$ is the same for every $i$). 

From now on we go back to referring to the system \eqref{informalpartsys1}-\eqref{informalpartsys2}. Let us introduce the
weighted empirical measure $v^{N,\varepsilon}_t$ defined as 
\begin{equation}\label{eqn:weightedempiricalmeasure}
    v^{N,\varepsilon}_t(dx):=\frac{1}{N}\sum \limits_{i=1}^{N} A\iine_t \cdot \delta_{X\iine_t}(dx)\,, 
\end{equation}
where the particles $X\iine$ and the weights $A\iine$ are as in \eqref{informalpartsys1}-\eqref{informalpartsys2}.  Then the following holds. 
\begin{thm}[Informal Statement of main result] \label{thm:mainthm}
Let $v_t$ be the solution of \eqref{SPDEintro} and $v_t^{N,\ep}$ as in \eqref{eqn:weightedempiricalmeasure}. Then the weighted empirical measure   $v_t^{N,\ep}$
 converges weakly, as  $\ep\rightarrow 0$ and  $N\rightarrow \infty$, to the solution of \eqref{SPDEintro}; that is, 
\begin{equation*}
    \lim_{\varepsilon\rightarrow 0 }\lim_{N\rightarrow \infty}  \sup_{t \in [0,T]}\langle v_t^{N, \varepsilon} - v_t, f \rangle    =0\, , a.s.
\end{equation*}
for every $f \in C^{\infty}(\mathbb T)$ and where  in the above $\langle \cdot, \cdot \rangle$ denotes duality pairing between measures and functions. 
\end{thm}
The complete statement of this theorem and the strategy of proof are contained in Section \ref{section: main results}, see Theorem \ref{theorem infinite dimensional approximation}. To explain the presence of the mollifier $\Phi_{\ep}$ in \eqref{informalpartsys2}, note that the term at the denominator of \eqref{informalpartsys2} can be written as 
\be\label{eqn:denom}
\frac{1}{N}\sum \limits_{j=1}^N\Phi_{\e}(X\iine_t-X^{j, N, \varepsilon}_t) = (\Phi_{\varepsilon} \ast \zeta_t^{N, \varepsilon} )(X\iine)
\ee
where 
\be\label{eqn:empmeaspartsys}
\zeta_t^{N, \ep}(dx):= \frac{1}{N}\sum \limits_{i=1}^{N} \delta_{X\iine_t}(dx)\,;
\ee
that is, $\zeta_t^{N, \ep}$ is the empirical measure associated with the particle system $\{X\iine_t\}_{i =1}^N$. From the heuristic calculation that we will do in the next section (see \eqref{additive} and comments afterwards) it will become clear that, in principle, one would like to consider \eqref{informalpartsys1} coupled to the following weights
$$
dA_t^{i,N} = \frac{\sum \limits_{z \in \mathbb Z} \lambda_z e_z(X_t^{i,N}) dw_t^z}{\zeta_t^{N}(X_t^{i,N})} \, ,
$$
instead of \eqref{informalpartsys2}. However, from a technical point of view, it is not clear how to make sense of the above object, because of the presence of the empirical measure at the denominator. This is where the mollification comes into play.

 We emphasize that the SPDE \eqref{SPDEintro} is not obtained as limit of the empirical measure $\zeta^{N, \ep}$ (which is a random probability measure) associated to the particle system, but rather as limit of the weighted empirical measure $v^{N,\ep}$  \eqref{eqn:weightedempiricalmeasure} (which is a random {\em signed} measure). 
 Let us now come to provide a heuristic argument to explain why $v^{N,\ep}$  converges to the solution of the SPDE \eqref{SPDEintro}. 
%%%%%%%%%%%%%%%%%%%
%%%%%%%%%%%%%%%%%%%%%%%%%%%%%%
%%%%%%%%%%%%%%%%%%%%%%%

\subsection{Heuristics}\label{subsec:heuristics}
In this section we take the particle-weight system \eqref{informalpartsys1}-\eqref{informalpartsys2} 
as an ansatz; given this ansatz, we provide an informal argument to show that, assuming  $v_t^{N,\e}$ converges, as $N\rightarrow \infty$ and $\e\rightarrow 0$,  to some limit $v_t$, such a limit is expected to be  the solution of the SPDE \eqref{SPDEintro}. One can also define a weighted 
empirical measure using the particle-weight system \eqref{informalpartsys1}-\eqref{weights_simple} and 
such a measure would converge to the solution of  \eqref{pde+bm} (to follow the calculation 
below in the latter simpler case just set  $\lambda_0=\sqrt{2\pi}$,  $\lambda_z=0$ for every $z\neq 0 $ and take $e_0=(2\pi)^{-1/2}$). 

By applying It\^o's formula to the process $\{A\iine_t \, f(X\iine_t)\}_{t\geq 0}$, we obtain 
\begin{align}
\label{eqn:ito1}
    d\left(A\iine_t \, f(X\iine_t)\right)\,& =\,  A\iine_t \cdot f'(X\iine_t)\,dX\iine_t\,+ \sigma\,A\iine_t\cdot f''(X\iine_t)\,dt\, \nonumber\\
    & \, + \,f(X\iine_t)\,dA\iine_t+\, f'(X\iine_t)\,dA\iine_t\,dX\iine_t \,.
    \end{align}
Formally, since the Brownian motions $\beta_t^i$'s are independent of the $w_t^z$'s, the last addend in the above vanishes. As for the first addend, note that the displacement $dX_t\iine$ in \eqref{informalpartsys1} can be rewritten, using \eqref{eqn:weightedempiricalmeasure}, as 
$$
dX\iine_t \, =\, - [V'(X_t\iine)  + (F'\ast v_t^{N, \ep})(X_t\iine)] dt+ \sqrt{2} d\beta_t^i \,.
$$   
Similarly, from \eqref{informalpartsys2} and \eqref{eqn:denom}, the displacement $dA_t\iine$ appearing in the third addend on the right hand side (RHS) of \eqref{eqn:ito1} can be rewritten as 
$$
dA_t\iine = \frac{\sum \limits_{z \in \mathbb Z} \lambda_z e_z(X_t^{i,N,\ep}) dw_t^z}{(\Phi^{\varepsilon}\ast \zeta_t^{N,\ep})(X\iine_t)}  \,.
$$
Using these observations,  we then get
    \begin{align}
    d\left(A\iine_t \, f(X\iine_t)\right)&=- A\iine_t\cdot\left[ V'(X_t\iine) + (F'\ast v^{N,\e}_t)(X\iin_t)\right]f'(X\iine_t) \,  dt \nonumber \\
    & +  \sqrt{2}\,A\iine_t f'(X\iine_t)\, d\beta_t^i
 +\sigma\,A\iine_t\cdot f''(X\iine_t)\,dt \nonumber \\
 & +\,f(X\iine_t)\cdot \frac{\sum \limits_{z \in \mathbb Z} \lambda_z e_z(X_t^{i,N,\ep}) dw_t^z}{(\Phi^{\varepsilon}\ast \zeta_t^{N,\ep})(X\iine_t)}\, . \label{brbrbr}
\end{align}
Since 
$$
\frac 1N \sum \limits_{i=1}^N A^{i, N, \e}_t f(X_t^{i, N, \e}) = \langle f, v_t^{N, \ep}\rangle ,  
$$
taking sums over $i=1,\dots,N$ on both sides of \eqref{brbrbr} and dividing by $N$, we obtain
\begin{align}
    d \langle f,\,v^{N,\e}_t\rangle &= \langle -V' \, f'-(F'\ast v^{N,\e}_t)\,f'+\sigma f'',\, v^{N,\e}_t\rangle\,dt + \frac{\sqrt{2}}{N}\sum \limits_{i=1}^{N} A\iine_t\, f'(X\iine_t)\, d\beta_t^i\nonumber\\
    &+ \left \langle f \, \frac{\sum \limits_{z \in \mathbb Z} \lambda_z e_z dw_t^z}{\Phi_\e\ast \zeta^{N,\e}_t},\, \zeta^{N,\e}_t\right\rangle\ ,\label{dvN}
\end{align}
where $v^{N,\varepsilon}$ is the weighted empirical measure \eqref{eqn:weightedempiricalmeasure}. Suppose now that, as $N\rightarrow \infty$, $\zeta_t^{N,\ep}$ and $v_t^{N,\ep}$ converge to some functions (or, more generally, measures) $\zeta_t^\ep$ and $v_t^{\ep}$, respectively.   Suppose also that there exist functions $\zeta_t$ and $v_t$ such that $\zeta_t^\ep \rightarrow \zeta_t $ and $v_t^{\ep} \rightarrow v_t$ (in an appropriate sense),   as $\ep \downarrow 0$. Then, formally passing to the limits $N\rightarrow \infty, \e \rightarrow 0$ in each of the terms in \eqref{dvN} one can see that $v_t$ will satisfy \eqref{SPDEintro}, in weak sense. Indeed, the martingale term (the second addend on the RHS of \eqref{dvN}) is a mean-zero term with quadratic variation of order $N^{-1}$ and will therefore disappear in the limit $N\rightarrow \infty$. As for the last term, which is the most delicate, we have 
\begin{align}
\left \langle f \, \frac{\sum \limits_{z \in \Z} \lambda_z e_z dw_t^z}{\Phi_\e\ast \zeta^{N,\e}_t},\, \zeta^{N,\e}_t \right \rangle & \stackrel{N\rightarrow \infty}{\longrightarrow} \left \langle f\, \frac{\sum \limits_{z \in \Z} \lambda_z e_z dw_t^z}{\Phi_\e\ast \zeta^{\ep}_t},\, \zeta^{\e}_t \right \rangle \nonumber\\
&\stackrel{\ep\rightarrow 0}{\longrightarrow} \left \langle f \, \frac{\sum \limits_{z \in \Z} \lambda_z e_z dw_t^z}{ \zeta_t},\, \zeta_t \right \rangle = \int_{\T} f(x) \pa_tW(t,x) \, dx \,,\label{additive}
\end{align}
 provided that $\zeta_t$ has strictly positive density.
 This shows that the weights $A\iine_t$ are designed in such a way that, in the limit, the dependence on the density $\zeta_t$ associated to the particle system `cancels' thanks to the simplification in the last step of the above, so that the resulting limiting noise is additive. 

Clearly, many of the technical difficulties in this paper arise from needing estimates from below, which are uniform in the appropriate parameter,  on the denominator in \eqref{informalpartsys2}. Comments on how these have been addressed can be found in  Note \ref{note:unifinepsilon} (and the associated Lemma \ref{stime}). Informally, one can see that having these lower bounds  allows one to rigorously perform the 'cancellation' in the above.  The part of the proof where this `cancellation' is made precise is in Section \ref{limit M to infty}, see Note \ref{note importante}.

Since the IPS \eqref{initialPS} converges to \eqref{PDEsimpleparticlesystem}, 
one might wonder why we are not using the particle-weight system \eqref{initialPS}-\eqref{informalpartsys2} in order to obtain \eqref{SPDEintro}, rather than the system \eqref{informalpartsys1}-\eqref{informalpartsys2} (note that in the former case the particles would be independent of the weights). As one can see by applying an analogous heuristic argument to the one shown above, if  we were to use the particle-weight system \eqref{initialPS}-\eqref{informalpartsys2}, then the corresponding weighted empirical measure $v_t^{N, \e}$ would converge to the following SPDE
\be\label{wronglimit}
\pa_t v_t = \pa_{xx}v_t+\pa_x \left[ V^{'}v_t+(F^{'}*\zeta_t) v_t \right ]+  \pa_t W_t,
\ee
where $\zeta_t$ again is the limit (assuming it exists) of the empirical measure of the particles. In this case $\zeta_t$ would solve the PDE \eqref{PDEsimpleparticlesystem}, the evolution of which is completely decoupled from the evolution \eqref{wronglimit}. 
Hence, the limiting SPDE \eqref{wronglimit} would be linear; in particular, we would not obtain the desired non-linearity. 

Finally, if the equation for the particles was given by \eqref{informalpartsys1} and the equation for the weights was not divided by $\Phi^{\ep} \ast \zeta_t^N$, i.e. if the weights were chosen to evolve according to 
$$
dA_t^{i,N} = {\sum \limits_{z \in \mathbb Z} \lambda_z e_z(X_t^{i,N}) dw_t^z}  \,, 
$$
from \eqref{dvN}-\eqref{additive} it is easy to see that the limiting equation would be 
\begin{equation}\label{eqn:starstar}
\pa_t v_t = \pa_{xx}v_t+\pa_x \left[ V^{'}v_t+(F^{'}*v_t) v_t \right ]+ \zeta_t \pa_t W_t \,. 
\end{equation}
This time however the evolution for  $\zeta_t$ and $v_t$ would be coupled,  so the noise in the above is multiplicative, which is not what we want.

 The statement in Theorem \ref{thm:mainthm} is informal. In particular,  both the statements of the theorem and  the heuristics that we have just presented contain only two limits, the many particle limit and the $\varepsilon \rightarrow 0$ limit. In reality we will need to take four limits, the two additional ones coming from the need to take a truncation  (the drift in \eqref{informalpartsys1} is unbounded in the weight variable) \footnote{The reason for taking this truncation, which a priori seems a bit excessive, considering that the drift is Lipshitz,  is discussed in Note \ref{note:cutoff}} and, more importantly, the `regularization limit', i.e. the limit coming from the fact that we will approximate Browninan paths with a sequence of piecewise smooth paths.   Precise statements of results and the organization of the paper are given in Section \ref{section: main results}.  

\begin{note}\label{Note1}\textup{ We make some comments to compare Theorem \ref{thm:mainthm} to the works \cite{kurtz1999particle, crisan2018particle}.\\
$\bullet$ In \cite{kurtz1999particle} the authors consider nonlinear SPDEs with multiplicative noise, and their representations in terms of weighted empirical measures of the form 
\begin{equation}\label{deFinetti}
\mathcal{V}_t = \lim_{n\rightarrow \infty} \frac 1n \sum \limits_{i=1}^n A^i_t \delta_{X^i_t} \,.
\end{equation}
To be more precise, while our particle-weight system \eqref{informalpartsys1}-\eqref{informalpartsys2} contains $N$ particles (and $N$ weights) and the equations for $(X^{i,N, \ep}_t, A^{i,N, \ep}_t)$ depend on $v^{N, \ep}$ (the relevant superscript here being $N$, not $\ep$), the particle-weight systems considered in \cite{kurtz1999particle} are infinite systems,  $\{X^i_t,A^i_t\}_{i \in \mathbb N}$ and the evolution of $(X^i_t,A^i_t )$ depends on the measure \eqref{deFinetti}. For this reason, the system $\{X^i_t,A^i_t\}_{i \in \mathbb N}$ considered in \cite{kurtz1999particle}
is a {\em particle representation} of the limiting SPDE, rather  than a {\em particle approximation}. Differently put, in \cite{kurtz1999particle} the measure \eqref{deFinetti} is assumed to converge, here we prove such a convergence. \\
$\bullet$ Neither in \cite{kurtz1999particle} nor in \cite{crisan2018particle} the equation for the weights is ``divided by" the empirical measure $\zeta^{N, \varepsilon}$ of the particles. This is new to this work and it is what allows us to obtain additive noise. The reweighting of the particles according to \eqref{informalpartsys2} could be intuitively interpreted as a form of importance sampling. \\
$\bullet$ In \cite{crisan2018particle} the authors 
study particle representations of SPDEs with boundary conditions. In that work the particles $\{X^i_t\}_{i \in \mathbb N}$ evolve independently of each other and of the weights, and start their evolution in 
stationarity so that, denoting by $\pi$ the stationary measure, $X^i_t$ is distributed according to 
$\pi$ for every time $t>0$ and for every $i \in \mathbb N$. The SPDE at hand in \cite{crisan2018particle} is well-posed in the 
space $L^2_{\pi}$ (the space $L^2$ weighted by the measure $\pi$, see \cite{crisan2018particle} for 
details) and the main statement, i.e. the  fact that the measure \eqref{deFinetti} is a representation 
of the solution of the SPDE at hand,  is then proved in the same space. This means that the noise 
appearing in the SPDE is not multiplied by the solution $v$ of the SPDE itself, but it is `multiplied 
by $\pi$' (as the weak formulation is in $L^2_{\pi}$). With the approach taken in 
\cite{crisan2018particle} (where, again, the equation for the weights is not `divided by $\zeta^{N,\e}$'), if 
the particles were not stationary and  independent of each other, the noise term in the limiting SPDE 
would be multiplied by the limit $\zeta$ (if it exists) of the empirical measure $\zeta_t^{N,\e}$ 
\eqref{eqn:empmeaspartsys} associated to the particles - this can be seen with a reasoning similar to the one done to obtain \eqref{eqn:starstar}. If one is 
interested in obtaining additive noise, this is not  much of a problem as long as the particles evolve 
independently of the weights. However, in our case we need the particles to depend on the weights (in 
order to obtain the correct nonlinear term in the limiting SPDE). Hence, if we retained the dependence on $\zeta$ in the noise term, we would  be obtaining 
an SPDE with multiplicative noise (as $\zeta$ would depend on $v$). For this reason we `divide by $\zeta^{N,\e}$' the equation for the weights,  as exemplified by the formal calculation  \eqref{additive}. 
}
\end{note}

\section{Main results and strategy of proof}\label{section: main results}

The main results of this paper are Theorem \ref{main_thm}, Theorem \ref{main_thm2} and Theorem \ref{theorem infinite dimensional approximation} below, providing particle approximations for the PDE \eqref{rough PDE}, and for the SPDEs \eqref{pde+bm} and \eqref{SPDEintro}, respectively. 

In this section we first comment on notion of solutions we will use for equations \eqref{rough PDE}, \eqref{pde+bm} and  \eqref{SPDEintro} and on the well-posedness of such equations. We then state Theorem \ref{main_thm}, Theorem \ref{main_thm2} and Theorem \ref{theorem infinite dimensional approximation}.  Since the proof of Theorem \ref{main_thm} is the core of this paper - and indeed, once Theorem \ref{main_thm} is proved, the other two follow easily (we explain why later in this section) - in Section \ref{strategy of proof} we present and motivate the strategy of proof of Theorem \ref{main_thm}, and make several comments on the main results of this paper and on their proof. 

We make the following standing assumption, which will hold throughout,  and we don't repeat it in every statement. 
\begin{assumption}[Standing assumption]\label{assumption 2.1}
The coefficients $V,F,q$ appearing in equations \eqref{rough PDE}-\eqref{SPDEintro} are smooth functions,  $V,F,q \in C^{\infty}(\T;\R)$.
\end{assumption}

We will use both the weak and the mild formulation of equations  \eqref{rough PDE}-\eqref{SPDEintro}.  More precisely, 
$\rho=\rho_t(x): \R_+ \times \T \rightarrow \R$ is  a  {\em mild solution}  (we should say `continuous-time $L^2(\T;\R)$-valued mild solution' but we say mild solution in short)  to \eqref{rough PDE} if for each $t\geq 0$, $\rho_t $ belongs to the space  $ L^2(\T;\R)$, the map $t \mapsto \rho_t \in L^2(\T;\R)$ is continuous and the following equality is satisfied (as an equality in $L^2$) for every $t \in [0,T]$: 
\begin{equation}\label{mild solution}
\rho_t(x)=e^{t\pa_{xx}}\rho_0(x)+\int_0^t e^{(t-s)\pa_{xx}}[(V^{'}+F^{'}*\rho_s)(x)\rho_s(x)]\,ds+\int_0^t e^{(t-s)\pa_{xx}}q(x)\,dY_s \, ,   
\end{equation}
where  $\{e^{t\pa_{xx}}\}_{t>0}$ denotes the heat semi-group on $\T$ (the definition of which is recalled in \eqref{heat semigroup}). The integral in the last addend of the above is to be intended as a Young integral (we give a recap on Young integration in Appendix \ref{young}).  

Similarly,  a $\hat{\mathcal F}_t$-adapted,  time-continuous $L^2(\T;\R)$-valued  stochastic process is a mild solution to \eqref{SPDEintro} if the equality 
\begin{equation*}
v_t(x)=e^{t\pa_{xx}}v_0(x)+\int_0^t e^{(t-s)\pa_{xx}}[(V^{'}+F^{'}*v_s)(x)v_s(x)]\,ds+\int_0^t e^{(t-s)\pa_{xx}}q(x)\,dW_s(x) \,,\,    
\end{equation*}
holds in $L^2(\T;\R)$, $\hat{\mathbb P}$-almost surely, for every $t \in [0,T]$, where  $ \left (\Omega, \hat \cF, \{\hat{\cF}_t\}_{t \geq 0}, \hat\mP \right )$ is the underlying filtered probability space, with $\hat{\cF}_t$ the filtration generated by the driving noise $W$ in \eqref{SPDEintro} (i.e. by the Brownian motions $\{w_t^i\}_{i \in \Z}$ appearing in \eqref{infinite dimensional noise}). In this instance, the integral appearing in the last addend is intended as a stochastic integral in $L^2(\T;\R)$ (see \cite[Chapter 2]{prevot2007concise} for an introduction to stochastic integration in infinite dimensions). 
The notion of mild solution for \eqref{pde+bm} is analogous so we don't repeat it. 

For any given $\rho_0 \in L^2(\T;\R)$ the mild solution  of  equations \eqref{rough PDE}-\eqref{SPDEintro} exists and it is unique, see Lemma \ref{lemma: well-posedness of rhoPDE}.  This mild solution coincides with a weak solution (see again Lemma \ref{lemma: well-posedness of rhoPDE}). We will exploit this equivalence (which is useful as a result per se' and in this paper it helps shortening some proofs), so we briefly recall that a \textit{weak solution} to \eqref{rough PDE} is a function $\rho:\R_+ \times \T \to \R$ such that $\rho_t \in L^2(\T;\R)$ for every $t \geq 0$, the map $t \to \rho_t \in L^2$ is continuous and, moreover, for any given $f \in C^{\infty}(\T;\R)$ the following equality holds for every $t \geq 0$
\begin{equation}\label{def: weak solution}
    \langle f,\rho_t \rangle=\langle f,\rho_0 \rangle + \int_0^t \langle \pa_{xx}f,\rho_s \rangle\,ds-\int_0^t \langle \pa_xf,(V^{}+F^{'}*\rho_s)\rho_s \rangle\,ds+ \langle f,q \rangle Y_t,
\end{equation} 
%\begin{equation*}
%    \langle f,\rho_t \rangle_{L^2(\T;\R)}=\langle f,\rho_0 \rangle_{L^2(\T;\R)} + \int_0^t \langle \pa_{xx}f,\rho_s \rangle_{L^2(\T;\R)}\,ds-\int_0^t \langle \pa_xf,(V^{}+F^{'}*\rho_s)\rho_s \rangle_{L^2(\T;\R)}\,ds+ \langle f,q \rangle_{L^2(\T;\R)} Y_t.
%\end{equation*} 
where here $\langle \,,\,\rangle$ denotes the scalar product in $L^2(\T;\R)$. \footnote{With abuse of notation $\langle\,,\,\rangle$ is used both for duality pairing and for $L^2$ scalar product; it should be clear from context which is the case, but we will clarify in places.}
Analogous definitions of weak solution can be stated for \eqref{pde+bm} and \eqref{SPDEintro} so we won't repeat them.

 A \textit{weak measure-valued} solution to \eqref{rough PDE} is a family $\{\rho_t\}_{t \geq 0}$ of measures satisfying \eqref{def: weak solution}, where this time in \eqref{def: weak solution} the brackets $\langle\,,\,\rangle $ are to be intended as duality pairing, and such that $\rho_{\cdot} \in C([0,T];\mathcal M(\T))$, where $\mathcal{M}(\T)$ is the space of finite signed measures on the torus endowed with the total variation norm.

We emphasize that the solutions we deal with are weak in PDE sense, but will always be strong solutions in probabilistic sense.

In what follows we will also consider three additional probability spaces $ \left (\Omega, \cF, \{\cF_t\}_{t \geq 0}, \mP \right )$, $ \left ( \Omega, \bar{\cF}, \{\bar{\cF_t}\}_{t \geq 0}, \bar{\mP} \right )$ and $ \left ({\Omega}, \tilde{\cF}, \{\tilde{\cF_t}\}_{t \geq 0}, \tilde{\mP} \right )$ where  $\{\cF_t\}_{t \geq 0}$ is the filtration generated by the Brownian motions  $\{\beta_t^i\}_{i \in \mathbb \N}$ driving the equation of the particles (see e.g. \eqref{particelle} below), $\{\bar{\cF}_t\}_{t \geq 0}$ is the filtration generated by the single Brownian motion $\{w_t\}_{t \geq 0}$ driving SPDE \eqref{pde+bm} while $\{\tilde{\cF}\}_{t \geq 0}$ is the filtration generated by $\{\cF_t\}_{t \geq 0}$, $\{\bar{\cF}_t\}_{t \geq 0}$ and by $\{\hat{\cF}_t\}_{t \geq 0}$. 
Corresponding expectations will be denoted by $\mathbb E, \bar{\mathbb E}, \tilde{\mathbb E}, \hat{\mathbb E}$, respectively.

 Let us now come to present the main results of this paper, starting from Theorem \ref{main_thm}, which contains a particle approximation result for the PDE \eqref{rough PDE}. 
As we have mentioned in the introduction, the proof of this result
is obtained by considering a sequence of smooth paths which converge to the forcing $\mY_t$ in \eqref{rough PDE}. So let $\{\mY^{\kappa}_{\cdot}\}_{\ka \in \N}$ be a  family of paths which are piecewise $C^1$ and  such that $\mY^{\kappa}_{\cdot} \xrightarrow[]{\kappa \to \infty} \mY_{\cdot}$ in $C^{\gamma}([0,T];\R)$ for some $\gamma < \alpha$.\footnote{We endow $C^{\alpha}$ with the metric $|f|_{C^{\alpha}}:=\sup_{t \in [0,T]}|f_t|+\sup \limits_{t,s \in [0,T], t \neq s} \frac{|f_t-f_s|}{|t-s|^{\alpha}}$.Without loss of generality we can assume that such an approximating family exists. Indeed, we recall that the space of  piece-wise $C^1$ functions is not dense in $C^{\alpha}(\T;\R)$. However,  if $\mY \in C^{\alpha}(\T;\R)$ then for any given $\gamma< \alpha$ there exists an approximating family of $C^1$ piece-wise functions converging to $\mY$ in $C^{\gamma}(\T;\R)$ (see e.g. \cite[Theorem 2.52 and Theorem 8.27]{weaver2018lipschitz}).
Hence, by `reducing' the H\"older exponent of the forcing term we can always assume the existence of such an approximating sequence.} 

Let us consider the following system
\begin{equation}\label{particelle}
    \begin{dcases}
        & \!\!\!\!\!\! dX_t^{i,N,\e, M,\kappa}  = - \!  \!\left( V^{'}(X_t^{i,N,\e, M,\kappa}) 
       + \frac{1}{N} \sum \limits_{j=1}^{N} \chi_M(A_t^{i,N,\e, M,\kappa}) F^{'}(X_t^{i,N,\e, M,\kappa}-X_t^{j,N,\e, M,\kappa}) \right )\!dt\! +\! \sqrt{2}d\noise_t^i\\
        & \!\!\! \!\!\!dA_t^{i,N,\e, M,\kappa}=\frac{q(X_t^{i,N,\e, M,\kappa})}{\Phi_{\e} *\zeta_t^{N,\e, M,\kappa}(X_t^{i,N,\e, M,\kappa})}d\mY_t^{\kappa}, \\
        &\\
        & \!\!\! \!\!\! X^{i,N,\e, M,\kappa}|_{t=0}=X_0^i,\,\,\, A^{i,N,\e, M,\kappa}|_{t=0}=A_0^i, 
    \end{dcases}
\end{equation}
where in the above $i \in \{1, \dots, N\}$ and 
\begin{itemize}
    \item $\{\chi_M\}_{M \in \N}$ is a family of cutoff functions, $\chi_M:\R \rightarrow \R$, defined as  
\begin{equation}\label{chiM}
\chi_M(a)=
    \begin{dcases}
         & a,\,\,\, |a|\leq M \\ \,\,
        &  0,\,\,\,\, |a| \geq 2M.
    \end{dcases}
\end{equation}
Note that for each $M \in \N$, $\chi_M$ is Lipschitz continuous, with Lipschitz constant 1 (independent of $M \in \N$).  
\item $\{ \Phi_{\e} \}_{\e >0}$ is a family of mollifiers, i.e. it is a family of functions $\Phi_{\e}:\T \rightarrow \R$ with the following properties:  $\Phi_{\e}(x) \geq c_{\e}>0$ for every $x \in \T$;   $\int_{\T}\Phi_{\e}(x)\,dx=1$ for every $\e >0$ and $\Phi_{\e}$ converges weakly, as $\e$ tends to zero,  to  the Dirac delta at zero, $\delta_0$. To fix ideas, we will consider Von-Mises distributions on the torus, with mean zero and concentration parameter $\e^{-1}$, i.e. 
\begin{equation}\label{von mises mollifier}
    \Phi_\e(x)\,:=\, \frac{1}{2\pi I_0(\e^{-1})}  e^{\e^{-1}\,\cos x},\, \qquad x \in \T,
\end{equation}
where  $I_0$ denotes the modified Bessel function of the first kind of order 0 (see also \cite{Amo74} for further details),
\begin{equation*}
    I_0(\e^{-1})\,=\, \sum \limits_{k=0}^\infty \frac{(2\e)^{-2k}}{(k!)^2}, \,\,\, \e >0.
\end{equation*}
\item $\zeta_t^{N,\e, M, \ka}$ is the empirical measure associated to the particles $X^{i, N, \e, M, \kappa}$, that is
\begin{equation}\label{empirical_zeta}
    \zeta_t^{N,\e, M, \ka}(dx):=\frac{1}{N}\sum \limits_{i=1}^N \delta_{X_t^{i,N,\e, M,\ka}}(dx),
\end{equation}
so that  $\zeta_t^{N,\e, M, \ka}$ is a random probability measure, i.e. $\zeta_t^{N,\e, M, \ka} \in \cP(\T)$ for each $N, \e, M, \ka, t$ fixed, where  $\cP(\T)$ denotes the space of probability measures on the torus $\T$.

\end{itemize}
%We also recall that $\ast$ denotes convolution so, to be clear, 
%$$\Phi_{\e}\ast \zeta_t^{N, \e, M. \kappa} (X_t^{i, N, \e, M, \kappa}) = \frac 1N %\sum \limits_{i=1}^N \Phi_{\e} (X_t^{i, N, \e, M, \kappa}- X_t^{j, N, \e, M, \kappa}) \,.$$

 Let us now introduce the following weighted empirical distribution : 
\begin{equation}\label{weighted_empirical}
    \rho_t^{N,\e,M,\kappa}(dx):=\frac{1}{N}\sum \limits_{i=1}^{N} A^{i,N,\e, M,\kappa}_t\cdot  \delta_{X^{i,N,\e, M,\kappa}_t}(dx) \,.
\end{equation} 
For each $N, \e, M, \ka, t$ fixed, $\rho_t^{N,\e,M,\kappa}$ is a random signed measure, i.e. $\rho_t^{N,\e,M,\kappa} \in \cM(\T)$.  In what follows, we make the following assumptions on the initial data.
\begin{assumption} \label{assunzioni sui dati iniziali}
    \begin{enumerate}[(a)]
        \item  The initial data $(X_0^i,A_0^i)$ of the particles system \eqref{particelle} are $\cF_0$-measurable i.i.d random variables with distribution $\mu_0 \in \cP_2(\T \times \R)$ where $\cP_2(\T \times \R)$ is the space of probability measures on $\T \times \R$ with finite second moment \label{def: mu0}.
        \item  The joint distribution $\mu_0$ has a smooth density, which, with an abuse of notation, we keep denoting by $\mu_0$. \footnote{Here and throughout, with abuse of notation, we use the same letter to denote a measure and its density, when it exists.}\label{def: mu0smooth}
        \item The $\T$-marginal $\zeta_0$ of $\mu_0$, namely $\zeta_0(x):=\int \limits_{\R} \mu_0(x,a)da$,  is strictly positive i.e. $\eta:=\min \limits_{x \in \T} \zeta_0(x)>0.$ \label{def:zeta0}
        \item The function $\rho_0(x):=\int \limits_{\R}a\mu_0(x,a)da$ belongs to $L^2(\T;\R)$. \label{def:rho0}
    \end{enumerate}
\end{assumption}
Note that \ref{def:rho0} is not an extra assumption as it follows from \ref{def: mu0} and \ref{def: mu0smooth}, but we spell it out for later ease of exposition.
We are interested in proving that $\rho_t^{N,\e,M,\kappa}$ in \eqref{weighted_empirical} converges to $\rho_t$, solution to \eqref{rough PDE}, in weak sense; this is the content of Theorem \ref{main_thm} below. 

\begin{thm}\label{main_thm}
Under Assumption \ref{assunzioni sui dati iniziali} the weighted empirical measure $\rho_t^{N,\e, M,\ka}$ (defined in \eqref{weighted_empirical}) converges weakly to $\rho_t$ solution of \eqref{rough PDE}.\footnote{This result still holds if we had replaced \ref{def: mu0} in Assumption \ref{assunzioni sui dati iniziali} with the condition that for every $N \in \N$ and $i=1,\cdots,N$ the pairs $(X_0^{i},A_0^{i})$ have distribution $\mu_0^i \in \cP_2(\T \times \R)$ such that the empirical distribution $\frac{1}{N}\sum \limits_{i=1}^N \mu_0^i \xrightarrow{N \to \infty} \mu_0$ in $\cP_2(\T \times \R)$.} Namely, for any given $f\in C^{\infty}(\T;\R)$ the following holds: 
\begin{equation*}
    \lim_{\ka \to \infty} \lim_{M\to\infty} \lim_{\e\downarrow 0} \lim_{N\to\infty}\,\,\,\sup \limits_{t \in [0,T]}\,\,\, \left |\langle f,\rho_t^{N,\e,M,\kappa} \rangle \,-\, \langle f, \rho_t \rangle \right |=0,\,\,\,\,\mP-a.s.
\end{equation*}
where $\rho_t$ is the weak solution of the PDE \eqref{rough PDE} with initial datum $\rho_0$ defined by \ref{def:rho0} in Assumption \ref{assunzioni sui dati iniziali}. 
\end{thm}
%The brackets $\langle\,\,,\,\,\rangle$ denote duality pairing between measures and smooth functions. If the measure were to admit an $L^2(\T;\R)$ density with respect to the Lebesgue measure,  i.e. $\mu(dx)=\mu(x)dx$ with $\mu \in L^2(\T;\R)$, \footnote{Here and throughout, with abuse of notation, we use the same letter to denote a measure and its density, when it exists.} then the dual product coincides with the scalar product of $L^2(\T;\R)$ i.e. $\langle f, \mu \rangle = \langle f, u \rangle_{L^2(\T;\R)}$, for every $f \in C^{\infty}(\T;\R)$.  

The proof of Theorem \ref{main_thm} is in four steps, corresponding to taking the four limits in the statement of the theorem.  We explain the strategy of proof in Subsection \ref{strategy of proof} below.  More explicitly, the proof of Theorem \ref{main_thm} is a consequence of Proposition \ref{prop limit rho N} (limit $N\rightarrow \infty$), Proposition \ref{prop limit rho eps} (limit $\e \rightarrow 0$), Proposition \ref{prop limit rho M} (limit $M\rightarrow \infty$) and Proposition \ref{prop limit rho kappa} (limit $\kappa \rightarrow \infty$), which are stated in Subsection \ref{strategy of proof} and proved in Section \ref{section_limitN} to Section \ref{limit kappa}, respectively. 
%%%%%%%%%%%%%%%%%%%%%%%%%%%%%%%%%%%%%%%%%%%%%%%%%%%%%%%%%%%%%%%%%%%%%
%%%%%%%%%%%%%%%%%%%%%%%%%%%%%%%%%%%%%%%%%%%%%%%%%%%%%%%%%%%%%%%%%%%%HERE

Since solutions to \eqref{pde+bm} are built in probabilistically strong sense, Theorem \ref{main_thm2} below is just a straightforward consequence of Theorem \ref{main_thm}. Indeed, solutions to \eqref{pde+bm} are built by simply fixing a realization $w_t$ of the driving Brownian motion; since the path of such a realization is ($\bar{\mathbb P}$-almost surely) $\alpha$- H\"older (for any $\alpha <1/2$), once the driving path is fixed, equation \eqref{pde+bm} is just a specific instance of equation \eqref{rough PDE}. 
 With this in mind, let $\bar{\omega} \in \left ( \Omega,\bar{\cF}, \{\bar {\cF_t} \}_{t \geq 0},\bar {\mP} \right )$ (by using this shorthand notation we mean that $\bar{\omega}$ belongs to $\Omega$, with $\Omega$ endowed with the $\sigma$-algebra $\bar{\cF}$ and the probability measure $\bar{\mP}$) be such that $w_t{(\bar{\omega})}$ is  a H\"older continuous realization of the noise driving \eqref{pde+bm} and then let  $\solu_t(\bar{\omega})$ be the corresponding solution. 
Let $\bar{Y}^{\kappa}_{\cdot}$ be a sequence of paths converging to $w_t(\bar{\omega})$ in $C^{\gamma}$, as $\ka \to \infty$, for some $\gamma <1/2$. 
\begin{thm}\label{main_thm2}
 %let us fix a realization $\bar{\omega} \in \Omega$ of it i.e. we consider the time-continuous $L^2(\T;\R)$-valued process $\{u_t(\bar{\omega})\}_{t \in [0,T]} \in C([0,T];L^2(\T;\R))$. 
Under Assumption \ref{assunzioni sui dati iniziali} and with the notation introduced so far, let $\bar{\rho_t}^{N,\e, M,\ka}$ be the weighted empirical measure constructed as in  \eqref{weighted_empirical} and associated to the particle-weight system \eqref{particelle} where the  family of trajectories 
$\{{Y}^{\ka}_{\cdot}\}_{\ka \in \N}$ is replaced by the family 
 $\{\bar{Y}^{\ka}_{\cdot}\}_{\ka \in \N}$ described above. Then the following limit holds: 
\begin{equation}\label{Them24limit}
 \lim_{\ka \to \infty} \lim_{M\to\infty} \lim_{\e\downarrow 0} \lim_{N\to\infty}\,\,\,\sup \limits_{t \in [0,T]}\,\,\, \left |\langle f,\bar{\rho}_t^{N,\e,M,\kappa} \rangle \,-\, \langle f, \solu_t(\bar{\omega}) \rangle \right |=0,\,\,\,\,\mP-a.s.
\end{equation}
where $u_t$ is the weak solution to the SPDE \eqref{pde+bm} with initial datum $\rho_0$ defined in \ref{def:rho0} of Assumption \ref{assunzioni sui dati iniziali}. 
\end{thm}
%{\color{red} check that you really meant $\mP$ in the above, once you have finalised the notation for the filtrations}
\begin{proof} with the comments made so far, the proof of Theorem \ref{main_thm2} is  a straightforward application of Theorem \ref{main_thm}.  
\end{proof}
We emphasize that the limit \eqref{Them24limit} holds $\mP$-almost surely. That is, Theorem \ref{main_thm2} is to be read as follows: we first fix a realization of the noise $w_t$; once this realization is fixed, there is still noise left in the system (more precisely, in $\bar{\rho}_t^{N, \ep, M, \ka}$), namely the noise associated to the Brownian motions $\beta_t^i$ in \eqref{particelle} and therefore the limit \eqref{Them24limit} holds almost surely with respect to this source of stochasticity. 

 %Let us emphasize that the above limit holds $\mP$-a.s. with respect to the Brownian motions $\{\beta^{i}\}_{i=1}^N,$ $N \in \N$, of the particle-weight system \eqref{particelle} for any fixed realization $\bar{\omega}$ of $\{u_t\}_{t \in [0,T]}$ and, therefore, of $\{\beta_t\}_{t \geq 0}$. 
 %That is, for any fixed  $\tilde{\omega} \in \tilde{\Omega}$ and for almost all $\omega \in \Omega$ the above limit holds {\color{red} $\Omega$ and $\tilde{\Omega}$ should be the same, what changes is not the sample space, is the set of sets that you can measure}
%\item We point out that in the above theorem and in what follows the weak convergence is meant in the PDE sense (i.e. one takes a test function, evaluates it against a distribution and then takes the limit) while on the stochastic side the limit is taken pathwise (i.e. for almost all trajectories).  

We explain in Note \ref{note:cutoff} why, from a technical point of view,  it is convenient to study the deterministic evolution  \eqref{rough PDE} and associated system \eqref{particelle} in order to produce results on the stochastic evolution \eqref{pde+bm} (and \eqref{SPDEintro}) and in Note \ref{note:commutativity limits} we make comments on the order in which the limits are taken. 

In  equation \eqref{rough PDE} the forcing term   $q(x)\pa_tY_t$ is the product of a smooth function of $x$, the function $q$,  and of a  $\alpha$-H\"older continuous path; it will be clear from the proof of Theorem \ref{main_thm} that the procedure used to construct the particle system converging to \eqref{rough PDE}  can be extended without any additional effort to the case in which \eqref{rough PDE} is forced by  a finite sum of such terms, see  equation \eqref{rough approx} below.   That is, let  $\{Y_t^{z}\}_{z=-m}^m $,   $m \in \N$,  be a family of $2m+1$  time-dependent paths, which are  $\alpha$-H\"older continuous and consider the equation obtained by replacing the term  $q(x)\pa_tY_t$ in \eqref{rough PDE} with the finite sum  $\sum \limits_{|z| \leq m}\lambda_ze_z(x)\pa_tY_t^z$ (where we recall the $e_z$'s are the basis of $L^2$ in \eqref{fourier}) - namely, equation \eqref{rough approx} below.  Fix $\gamma < \alpha$ and,  for every $z=-m,\cdots,m$,   let $\{Y_t^{z,\ka}\}_{\ka \in \N}$ be a  piece-wise $C^1$ approximation of $Y_t^z$ ; that is,  $Y_{\cdot}^{z,\ka} \xrightarrow[]{\ka \to \infty} Y_{\cdot}^z$ in $C^{\gamma}$.  With this premise, consider the following system of interacting particles
 \begin{equation}\label{infinite dimensional IPS}
    \begin{dcases}
      & \!\!\!\!\!\!\!\!\! \resizebox{\textwidth}{!}
     {
        $dX_t^{i,N,\e, M,\kappa,m}  = - \! \left( \!V^{'}(X_t^{i,N,\e, M,\kappa,m}) 
      \!+ \! \frac{1}{N} \!\! \sum \limits_{j=1}^{N} \chi_M(A_t^{i,N,\e, M,\kappa,m}) F^{'}(X_t^{i,N,\e, M,\kappa,m}-X_t^{j,N,\e, M,\kappa,m}) \!\!\right )\!dt \!+\!\sqrt{2}d\noise_t^i$
     } \\
        & \!\!\!\!\!\!\! dA_t^{i,N,\e, M,\kappa,m}= \sum \limits_{|z| \leq m} \lambda_z \frac{ e_z(X_t^{i,N,M,\e,m})}{\Phi_{\e} *\zeta_t^{N,M,\e,m}(X_t^{i,N,\e, M,\kappa,m})} dY_t^{z,\kappa}, \\
        & \!\!\!\!\!\!\!\! X^{i,N,\e, M,\kappa,m}|_{t=0}=X_0^i,\,\,\,\, A^{i,N,\e, M,\kappa,m}|_{t=0}=A_0^i.
    \end{dcases}  
\end{equation}
and  the corresponding weighted empirical measure:
\begin{equation}\label{weighted_empirical_infinite_dimensional}
    \rho_t^{N,\e,M,\ka,m}(dx):=\frac{1}{N}  \sum \limits_{i=1}^{N}A_t^{i,N,\e, M,\kappa,m}\delta_{X_t^{i,N,\e, M,\kappa,m}}(dx),\qquad t \in [0,T]. 
\end{equation}
Then, with the same proof of Theorem \ref{main_thm} the following proposition holds.
\begin{prop}\label{prop linear combination Intro}
%If we let $X_t^{N,\e,M,\ka,m}$ and $A_t^{N,\e,M,\ka,m}$ be as in \eqref{infinite dimensional IPS}.
Let Assumption \ref{assunzioni sui dati iniziali} hold;  let $\rho_t^{N,\e,M,\ka,m}$ be the weighted empirical measure defined in \eqref{weighted_empirical_infinite_dimensional} and  $\rho_t^{m}$ be the solution to the following equation
\begin{equation}\label{rough approx}
    \begin{dcases}
           & \pa_t\rho_t^{m}=\pa_{xx}\rho_t^{m}+\pa_x[(V^{'}+F^{'}*\rho_t^{m})\rho_t^{m}]+\sum \limits_{|z| \leq m} \lambda_z e_z \pa_t\mY_t^z,\qquad t \in (0,T]\\
            &\rho^{m}|_{t=0}=\rho_0,
    \end{dcases}
\end{equation}
where $\rho_0$ is the function defined by \ref{def:rho0} in Assumption \ref{assunzioni sui dati iniziali}. Then,  for each $m \in \N$ fixed and any given $f \in C^{\infty}(\T;\R)$,  the following limit holds 
\begin{equation*}
    \lim_{\ka \to \infty} \lim_{M\to\infty} \lim_{\e\downarrow 0} \lim_{N\to\infty} \,\,\sup \limits_{t \in [0,T]}\,\,\left | \langle f,\rho_t^{N, \e, M , \kappa, m} \rangle \,-\, \langle f, \rho_t^m \rangle \right |\,=\,0,\qquad \mP-a.s.
\end{equation*}
\end{prop}
The proof of the above proposition can be done by following the exact same steps of the proof of Theorem \ref{main_thm}, so we don't repeat the whole proof - the main purpose of stating Proposition \ref{prop linear combination Intro} is to more clearly explain how Theorem \ref{theorem infinite dimensional approximation} below is obtained. Indeed, at this point, since the trace class noise $W$ is given by \eqref{infinite dimensional noise}, i.e. it is just a (weighted) infinite sum of Brownian motions (i.e. an infinite sum of $\alpha$-H\"older paths), to obtain an approximation result for the SPDE \eqref{SPDEintro},  as a consequence of Proposition \ref{prop linear combination Intro}, it substantially suffices to let the truncation parameter $m$ in \eqref{infinite dimensional IPS} to infinity.

% by following the rationale that led us to formulate Theorem \ref{main_thm2} the particle approximation result for the SPDE \eqref{SPDEintro} can be shown.

More preciely,  by following the same rationale that led us to formulate Theorem \ref{main_thm2},  let $\{w_t^z\}_{z \in \mathbb Z}$ be as in \eqref{infinite dimensional noise}, so that the paths $w_t^z$'s drive the dynamics \eqref{SPDEintro}. Let 
 $\hat{\omega} \in \left ( \Omega,\hat{\cF}, \{\hat {\cF_t} \}_{t \geq 0},\hat {\mP} \right )$ 
be such that for every $z \in \Z$ the path $w_t^z  \left (\hat{\omega}\right)$ is a $\alpha$-H\"older  continuous realization of $w_t^z$ (for some $\alpha < \frac{1}{2}$). Then let $v_t(\hat{\omega})$ be the corresponding  solution to the SPDE \eqref{SPDEintro},  for a given initial datum $v_0 \in L^2(\T;\R)$.
%(which from \cite[Section 4]{angeli2023well} we know it exists and is unique as a time-continuous $L^2(\T;\R)$-valued stochastic process). 
For every given $z \in \Z$, let $\hat{Y}^{z,\kappa}_{\cdot}$ be a sequence of  piece-wise $C^1$ paths converging, as $\ka \to \infty$,  to $w_{\cdot}^z(\hat{\omega})$ (in $C^{\gamma}$for some $\gamma <\alpha$). Then the following holds.
\begin{thm}\label{theorem infinite dimensional approximation} %let us fix a realization $\bar{\omega} \in \Omega$ of it i.e. we consider the time-continuous $L^2(\T;\R)$-valued process $\{u_t(\bar{\omega})\}_{t \in [0,T]} \in C([0,T];L^2(\T;\R))$. 
Let Assumption \ref{assunzioni sui dati iniziali} hold. Suppose there exists $\delta \in \left (0,\frac{1}{2} \right )$ such that $\sum \limits_{z \in \Z} |z|^{4\delta} \lambda_z^2 < +\infty$ and let $\{\hat{\rho}_t^{N,\e,M,\ka,m}\}_{t \in [0,T]}$ denote the weighted empirical measure \eqref{weighted_empirical_infinite_dimensional} associated to system \eqref{infinite dimensional IPS} with $\hat{Y}_t^{z,\ka}$ in place of $Y_t^{z,\ka}$, for every $z \in \Z$ and $\ka \in \mathbb N$. %with the additional condition that for each fixed $z \in \Z$ the family of trajectories $\{Y^{z,\ka}\}_{\ka \in \N}$ approximate $Y_t^{z}$ given in \eqref{forcing= infinite dimensional bm}. 
Then the following limit holds: 
\begin{equation*}
\lim \limits_{m \to \infty} \lim_{\ka \to \infty} \lim_{M\to\infty} \lim_{\e\downarrow 0} \lim_{N\to\infty}\,\,\,\sup \limits_{t \in [0,T]}\,\,\, \left |\langle f,\hat{\rho_t}^{N,\e,M,\kappa,m} \rangle \,-\, \langle f, v_t(\hat{\omega}) \rangle \right |=0,\,\,\,\,\mP-a.s., 
\end{equation*}
where $v_t$ is the solution of the SPDE \eqref{SPDEintro} with initial datum $v_0$ equal to $\rho_0$ (where $\rho_0$ is as in \ref{def:rho0} of Assumption \ref{assunzioni sui dati iniziali}).  
\end{thm}
The proof of Theorem \ref{theorem infinite dimensional approximation} can be found in Section \ref{infinite noise bm}. The assumption on the decay of the eigenvalues $\lambda_z$ appears in the above theorem to guarantee well-posedness of \eqref{SPDEintro} (and, as a consequence, it is used when taking the limit in $m$, detail in Section \ref{infinite noise bm}).
%and it is not hard to be carried out once one proves Theorem \ref{main_thm}. Hence, we focus on Theorem \ref{main_thm} and we come to describe the programme to achieve this result.
\subsection{Strategy of proof of Theorem \ref{main_thm}}\label{strategy of proof}
From this moment onward, unless otherwise stated we make Assumption \ref{assunzioni sui dati iniziali} a standing assumption so we don't repeat it in every statement.

As already mentioned, the proof of Theorem \ref{main_thm} is divided into four steps, corresponding to taking the limits $\lim \limits_{N \to \infty}$, $\lim \limits_{\e \downarrow 0}$, $\lim \limits_{M \to \infty}$ and $\lim \limits_{\ka \to \infty}$ (in this order). Let us explain how each of the steps is dealt with.
\\
$\bullet $ {\bf Letting the number of particles $N$ to infinity.} This is the most standard step. 
 We start by studying the convergence of the pair $\left (X_t^{i,N,\e,M,\ka}, A_t^{i,N,\e,M,\ka} \right )$ solving 
\eqref{particelle},  as $N \to \infty$. More precisely,  in Section \ref{section_limitN} (see Proposition \ref{prop_limitN}),  we prove that  the empirical measure  $\mu_t^{N,\e, M,\ka}$ of the pair $\left ( X_t^{i,N,\e,M,\ka}, A_t^{i,N,\e,M,\ka} \right )$, namely,   
\begin{equation}\label{empirical_m}
    \mu_t^{N,\e,M,\kappa}(dxda):=\frac{1}{N}\sum \limits_{i=1}^{N} \delta_{(X^{i,N,\e, M,\kappa}_t,A^{i,N,\e, M,\kappa}_t)}(dxda) \, ,
\end{equation}
converges to the probability measure $\mu_t^{\e,M,\ka}$ on $\T \times \R$ which is the law  of the process $\left (X_t^{\e,M, \ka}, A_t^{\e,M, \ka} \right )$, solution of the following McKean-Vlasov SDE :
\begin{equation}\label{sistemaMKE}
    \begin{dcases}
    & X_t^{\e,M,\ka}\,=\, X_0\,-\,\int_0^t \left(V'(X_s^{\e,M,\ka})+\Gamma_M(X_s^{\e,M,\ka},\mu_s^{\e,M,\ka})\right)\,ds\,+\,\sqrt{2}\,\beta_t,  \\
    & A_t^{\e,M,\ka}\,=\, A_0\,+\,\int_0^t \frac{q(X_s^{\e,M,\ka})}{\Phi_{\e}*\zeta_s^{\e,M,\ka}(X_s^{\e,M,\ka})}\,d\mY^{\kappa}_s,
    \end{dcases}
\end{equation}
where in the above $\beta_t$ is a standard Brownian motion on the torus, $\mu_t^{\e,M, \ka}:=\cL\left (X_t^{\e,M, \ka}, A_t^{\e,M, \ka} \right )$  (here and throughout $\cL$ denotes the law of a stochastic process),  $\Gamma_M$ %and $\Xi_{\e}$ 
is defined as 
\begin{align}\label{def_GammaM}
    \Gamma_M(x,\mu)\,:=\, \itr \chi_M(a)\cdot F'(x-y)\,\mu(dyda),\quad \mbox{for every }\,x \in \T,\,\mu \in \cP_2(\T \times \R)
    %\Xi_{\e}(x,\mu)\,:=\, \frac{q(x)}{\itr \Phi_\e(x-y)\mu(dyda) },\,\, x \in \T,\, \mu \in \cP_2(\T \times \R), \label{def_betaMe}
\end{align}
and $\zeta_t^{\e,M,\ka}$ is the $\mathbb{T}$-marginal of the law $\mu_t^{\e,M,\ka}$ of the pair $(X_t^{\e,M, \ka}, A_t^{\e,M, \ka})$, namely,
\begin{equation}\label{zetaMKE}
    \zeta_t^{\e,M,\ka}(dx)=\int_{\R} \mu_t^{\e,M,\ka}(dxda),\qquad t \in [0,T].
    %\,\,\,\,f \in C^{\infty}(\T;\R) \,. %\mE\left( f(X_t^{\e,M,\ka}\right),\,\,\,\,
\end{equation}
In view of this result, the marginal $\zeta_t^{\e,M,\ka}$ can be interpreted as the limit (in  weak sense) as $N \to \infty$ of the empirical measure $\{ \zeta_t^{N,\e,M,\ka} \}_{N \in \N}$ in \eqref{empirical_zeta} associated  with the particles $X_t^{i, N, \e, M. \ka}$ (we will not use this fact, but it is still useful to point it out for general understanding).
Moreover, in the limit $N\rightarrow \infty$, for each $i \in \N$, the pair $\left (X_t^{i,N,\e,M,\ka}, A_t^{i,N,\e,M,\ka} \right )$ converges to the solution $(X_t^{\e,M, \ka}, A_t^{\e,M, \ka})$ of \eqref{sistemaMKE}, see again Proposition \ref{prop_limitN} . 

Using e.g. the methods of \cite{coghi2019stochastic}, one can show (see Proposition \ref{prop:law satisfies PDE} at the end of this section) that the limit measure $\mu_t^{\e,M, \ka}$  is the unique weak measure-valued solution of the following  PDE
\begin{equation}\label{PDEmuMKE}
    \begin{dcases}
&\pa_t\mu_t^{\e,M,\ka}=\paxx\mu_t^{\e,M,\ka}+\pa_x \left[ (V^{'}+\Gamma_M(x,\mu_t^{\e,M,\ka})\mu_t^{\e,M,\ka}\right ]-\frac{q(x)}{\Phi_{\e}*\zeta_t^{\e,M,\ka}}\pa_a\mu_t^{\e,M,\ka}\pa_tY_t^{\ka}\\
        & \mu^{\e,M,\ka}|_{t=0}=\mu_0.
    \end{dcases}
\end{equation}
Note that the above equation is a closed equation for $\mu_t^{\ep,M,\ka}$: while $\zeta^{\ep,M,\ka}_t$  appears in the equation, $\zeta_t^{\ep,M,\ka}$ depends on $\mu_t^{\ep,M,\ka}$ only (see \eqref{zetaMKE}). We also note that, in establishing well-posedness of \eqref{PDEmuMKE},  it is clearly crucial that $\Phi_{\ep}$ is bounded below for each $\ep$ fixed, see proof of Proposition \ref{prop:law satisfies PDE}.
%%%%%%%%%%%%%%%%%%%%%%%%%%%
%%%%%%%%%%%%%%%%%%%%%%%%%%%%%%%%%55
  
%Lastly, the measure $\{\mu_t^{\e,M,\ka}\}_{t \in [0,T]}$ solves the PDE below (which can be formally obtained as in \cite{coghi2019stochastic}),
Since $\mu_t^{\e, M, \ka}$ solves \eqref{PDEmuMKE},  the marginal $\zeta_t^{\e,M,\ka}$ solves the following (linear) PDE 
\begin{equation}\label{PDEzetaMKE}
    \begin{dcases}
        & \pa_t\zeta_t^{\e,M,\ka}=\pa_{xx}\zeta_t^{\e,M,\ka}+\pa_x \left [ \left (V^{'}+\Gamma_M(x,\mu_t^{\e,M,\ka})\right) \zeta_t^{\e,M,\ka}\right ] \\
        & \zeta^{\e,M,\ka}|_{t=0}=\zeta_0.
    \end{dcases}
\end{equation}
This can be formally seen either by integrating  \eqref{PDEmuMKE} with respect to the variable $a \in \R$ or using It\^o's formula; indeed, since
$\zeta_t^{\e,M,\ka}$ is the marginal on $\T$ of $\mu_t^{\e, M,\ka}=\cL(X_t^{\e, M,\ka},A_t^{\e, M,\ka})$,   $\zeta_t^{\e,M,\ka}$ satisfies 
\begin{equation*}\label{zetaMKexpectation}
 \langle f,\zeta_t^{\e, M,\ka} \rangle=\mE \left( f(X_t^{\e, M,\ka})\right),\qquad\text{for all }f \in C^{\infty}(\T;\R)\,.
\end{equation*}
Hence, using It\^o's formula, we have 
\begin{align*}
    & df(X_t^{\e,M,\ka})=f^{'}(X_t^{\e,M,\ka})dX_t^{\e,M,\ka}+f^{''}(X_t^{\e,M,\ka})dt\\
    & =-f^{'}(X_t^{\e,M,\ka})\left [ V^{'}(X_t^{\e,M,\ka})+\Gamma_M(X_t^{\e,M,\ka},\mu_t^{\e,M,\ka})\right ]dt+\sqrt{2}f^{'}(X_t^{\e,M,\ka})d\noise_t+f^{''}(X_t^{\e,M,\ka})dt \, 
\end{align*}
which, after taking  expectation,  formally gives the  PDE \eqref{PDEzetaMKE}.  Let us also note that, since \eqref{PDEzetaMKE} is linear, from classical results in PDE theory, see e.g. \cite[Chapter 7]{evans2022partial} or \cite{friedman2008partial}, the solution to \eqref{PDEzetaMKE} is always a smooth function (since all the coefficients are smooth and the initial datum $\zeta_0$ is chosen to be smooth as well).  In the proofs of Section \ref{section_limitN} we will not use explicitly the fact that $\zeta_t^{\e,M, \ka}$ solves \eqref{PDEzetaMKE}. However it is important to note it here as a motivation for the way we proceed in the next steps of the proof. 

%At first glance, since the dynamics of the weight $A_t^{i,N,\e, M,\ka}$ with $i=1,\cdots,N$ depends on the empirical measure $\zeta_t^{N,\e, M,\ka}$ in \eqref{empirical_zeta} one should expect the dynamics of the weight $A_t^{\e,M,\ka}$ to the depend on a probability measure $\zeta_t^{\e,M,\ka}$ which depends only on the spatial variable $x \in \T$ and not on the joint distribution $\mu_t^{\e,M,\ka}$. In a subtle manner that is what the operator $\Xi_{\e}$ entails. Indeed, we can recast the dynamics of the weight $ A_t^{\e,M,\ka}$ in the following way. Namely, we have
%\begin{align}\label{weight beta}
 %   dA_t^{\e,M,\ka}=\frac{q(X_t^{\e,M,\ka})}{\Phi_{\e}*\zeta_t^{\e,M,\ka}(X_t^{\e,M,\ka})}d\mY_t^{\ka}.
%\end{align}
%Indeed, if we define the probability measure $\{\zeta_t^{\e,M,\ka}\}_{t \in [0,T]}$ as 
%\begin{equation}\label{zetaMKE}
%    \langle f,\zeta_t^{\e,M,\kappa} \rE:=\mE \left( f(X_t^{\e,M,\kappa}) \right )=\itr f(x)\,\mu_t^{\e,M,\ka}(dxda),\,\,t \in [0,T],\,\,\,f \in C^{\infty}(\T;\R)
%\end{equation}
%and since the mollifier $\Phi_{\e}$ depends only the spatial variable $x \in \T$ we thus obtain
%\begin{align*}
%& \Phi_{\e}*\zeta_t^{\e,M,\ka}(x)=\int_{\T} \Phi_{\e}(x-y)\,\zeta_t^{\e,M,\ka}(dy)\\
%& =\mE \left( \Phi_{\e}(x-X_t^{\e,M,\ka})\right)=\mE \left ( \1_{\R}(A_t^{\e,M,\ka} ) \Phi_{\e}(x-X_t^{\e,M,\ka}) \right )\\
%&=\itr \Phi_{\e}(x-y)\,\mu_t^{\e,M,\ka}(dyda)=\Xi_{\e}(x,\mu_t^{\e,M,\ka}),\,\,\,\,x \in \T,\,\, t \in [0,T],
%\end{align*}
%which gives \eqref{weight beta}.\\
While the above facts are instrumental to the proof, we are not interested in the limit of $\mu_t^{N,\e, M,\ka}$ but in the limit of $\rho_t^{N,\e, M,\ka}$, defined in \eqref{weighted_empirical}. Since  $\mu_t^{N,\e, M,\ka}$ converges to $\mu_t^{\e, M,\ka}$, upon defining the measure $\rho_t^{\e,M,\ka}$ as follows 
\begin{equation}\label{rhoMKE}
\rho_t^{\e,M,\kappa}(dx) := \int_{\mathbb R} a \mu_t^{\e,M,\kappa}(dxda),
\end{equation}
(note that for each $t>0$, $\rho_t^{\e,M,\ka}$ is a signed measure on the torus $\T$), we expect $\{\rho_t^{N,\e, M,\ka}\}_{N \in \N}$ to converge to $\rho_t^{\e,M,\kappa}$. One can at least formally obtain the equation solved by $\rho_t^{\e,M,\kappa}$ by multiplying equation \eqref{PDEmuMKE} by $a \in \R$ and integrating with respect to $a \in \R$. This is the intuition behind Proposition \ref{prop limit rho N} below,  the proof of which can be found in Section \ref{section_limitN}.
%Note that if $\{\rho_t^{\e,M,\kappa}\}_{t \in [0,T]}$ is defined as in \eqref{rhoMKE}, then the following holds 
%$$
%\langle f,\rho_t^{\e,M,\kappa} \rangle=\mE \left( A_t^{\e,M,\kappa}f(X_t^{\e,M,\kappa}) \right ),\,\,t \in [0,T],\,\,\,f \in C^{\infty}(\T;\R).
%$$ 
%The second approach \cite[Subection 3.4, pp.33-35]{lacker2018mean} consists in studying the PDE formulation of the 
%\eqref{particelle} and SDE \eqref{sistemaMKE}. In our case of interest, one looks at the empirical measure 
%\eqref{empirical_m} and studies its convergence as $N \to \infty$ and aims to prove that $\mu_t^{N,\e, M,\ka} \xrightarrow[N \to \infty]{} \mu_t^{\e,M,\ka}$ in $C\left( [0,T];\cP_2(\T \times \R)\right )$. These two approaches are equivalent and one can use one or the other depending on which one is more convenient. We implement the first approach in Section \ref{section_limitN} which in turn gives the following result:
%Once this is in place, we show that for any given $\kappa, \varepsilon, M >0$ the family of weighted measures $\{\rho^{N,\e,M,\kappa}_t\}_{t \in [0,T],N \in \N} \subset \cM(\T)$ converges to $\{\rho^{\e,M,\kappa}_t\}_{t \in [0,T]}\subset \cM(\T)$ where $\{\rho^{\e,M,\kappa}_t\}_{t \in [0,T]}$ is given by \eqref{rhoMKE}. 
\begin{prop}\label{prop limit rho N}
Fix $M,\kappa \in \N$, $\e >0$ and let $\{\rho_t^{N,\e,M,\kappa} \}_{N \in \N}$ and $\rho_t^{\e,M,\kappa}$ be as in \eqref{weighted_empirical} and \eqref{rhoMKE}, respectively. 
%defined as 
%\begin{equation}\label{rhoMKE}
%      \langle f,\rho_t^{\e,M,\kappa} \rE:=\mE \left( A_t^{\e,M,\kappa}f(X_t^{\e,M,\kappa}) \right ),\,\,t \in [0,T],\,\,\,f \in C^{\infty}(\T;\R).
%\end{equation}
%where the expected value is taken with respect to the Brownian motion $\{W_t\}_{t \geq 0}$ and let the weighted empirical measure $\{\rho_t^{N,M\e,\ka}\}_{t \in [0,T]}$ be as in \eqref{weighted_empirical}. 
Then the following limit holds:
\begin{equation*}
    \lim_{N\to\infty} \,\,\,\sup \limits_{t \in [0,T]}\,\,\, \left |\langle f,\rho_t^{N,\e,M,\kappa} \rangle \,-\, \langle f, \rho_t^{\e,M,\ka} \rangle \right |=0 \\,\,\,\,\,\mP-a.s.
    \end{equation*}
for any given $f\in C^\infty(\T;\R)$. Furthermore, $\rho_t^{\e,M,\ka}$ is a weak measure-valued (and, a posteriori, classical solution) to the following PDE
\begin{equation}\label{PDErhoMKE}
    \begin{dcases}
        &\pa_t\rho_t^{\e,M,\ka}=\paxx\rho_t^{\e,M,\ka}+\pa_x \left[ (V^{'}+\Gamma_M(x,\mu_t^{\e,M,\ka})\rho_t^{\e,M,\ka}\right ]+\frac{q(x)\zeta_t^{\e,M,\ka}(x)}{\Phi_{\e}*\zeta_t^{\e,M,\ka}(x)}\pa_tY_t^{\ka}\\
        & \rho^{\e,M,\ka}|_{t=0}=\rho_0,
    \end{dcases}
\end{equation}
where $\zeta_t^{\e,M,\ka}$ is defined in \eqref{zetaMKE}.
%and $\rho_0 \in \cM(\T)$ is defined in terms of the law of initial condition $\cL(X_0,A_0)$ of \eqref{sistemaMKE} as 
%\begin{equation*}
%    \rho_0(dx):= \int_{\R}a\mu_0(dxda).\qquad 
%\end{equation*}
\end{prop} 

\begin{note}\label{note:cutoff} Let us make some comments to motivate our strategy of proof. 

\noindent
$\bullet$ Introducing the cutoff $\chi_{M}$ is strictly speaking not needed for the proof of the limit $N\rightarrow \infty$ (but it does make the proof of this step easier, see the calculations around \eqref{eqn:cutoff}, which would be otherwise slightly more complicated). The place where we need it is in the proof of the limit $\ep\rightarrow 0$: by looking at \eqref{PDErhoMKE} it is clear that, when letting $\ep\rightarrow 0$,  one way to proceed is to obtain lower bounds which are uniform in $\ep$ for the quantity $\Phi_{\ep} \ast\zeta_t^{\ep,M,\kappa}$. The presence of the cutoff makes it easier to obtain such lower bounds. For more detail see Note \ref{note uniform bound in eps and M}.

\noindent
$\bullet$ If the  piece-wise $C^1$ forcing $\{Y_t^{\ka}\}_{t \in [0,T]}$ in \eqref{particelle} was replaced 
by a standard Brownian motion $\{w_t\}_{t \geq 0}$ then in the equation for $\mu_t^{\e,M,\ka}$ an additional second order term would arise; that is, the empirical measure \eqref{empirical_m} would converge to a measure, say $\mu_t^{\e,M}$, the evolution of which would not be described by \eqref{PDEmuMKE}, but rather by the following SPDE 
\begin{align}
        \pa_t\mu_t^{\e,M}  & = \paxx \mu_t^{\e,M} +\frac{1}{2} \frac{q^2(x)}{(\Phi_{\e}*\zeta_t^{\e,M})^2}\pa_{aa}\mu_t^{\e,M} \label{PDE1with additional second order term}\\
         & +\pa_x \left[ (V^{'}+
        \Gamma_M(x,\mu_t^{\e,M})\mu_t^{\e,M}\right ]-\frac{q(x)}{\Phi_{\e}*\zeta_t^{\e,M}}\pa_a\mu_t^{\e,M}\pa_t w_t  \, ,\nonumber 
\end{align}
see \cite{coghi2019stochastic}. Throughout the paper we will often need to estimate moments of the form $\itr |a| \mu(dx, da)$, with $\mu$ being either $\mu^{\ep, M, \kappa}$ or $\mu^{M, \kappa}$ (the latter defined by \eqref{PDEmuMK}). If the second order term (the second addend on the right hand side of the above) appears, estimating this quantity becomes more involved. This problem does not seem to be ameliorated by working directly with the McKean-Vlasov processes \eqref{sistemaMKE}/\eqref{sistemaMK} rather than with the equation for the corresponding joint densities. Another option would be to work directly with the equation for $\rho_t^{\e,M,\ka}$ or $\rho_t^{M,\ka}$, see \eqref{PDErhoMKE} and \eqref{PDErhoMK}, as these are additive equations (as opposed to the corresponding equations for the joint density, which are multiplicative). However, it seems to us that since $\Gamma_M$ depends on $\mu$, we would still end up needing estimates on moments of the joint density $\mu_t^{\e,M,\ka}$ or $\mu_t^{M,\ka}$. 

\noindent
$\bullet$ It is crucial to notice that, whether we fix the noise (and its sequence of approximants $Y^{\ka}$) and hence end up considering the joint distribution $\mu_t^{\ep, M, \ka}$ in \eqref{PDEmuMKE} or we keep the original Brownian motion and hence study the joint distribution $\mu^{\ep, M}$ in \eqref{PDE1with additional second order term}, the corresponding weighted marginal that we would obtain is the same. That is, if we set $\rho_t^{\ep, M} = \int_{\R} a \, \mu_t^{\ep, M}(da dx)$, then the equation satisfied by  $\rho_t^{\ep, M}$ (obtained by multiplying equation \eqref{PDE1with additional second order term} by $a \in \R$ and then integrating with respect to  $a \in \R$), is the same as equation \eqref{PDErhoMKE} satisfied by $\rho_t^{\ep, M, \ka}$. This is also true for the equations satisfied by all the joint distributions that will appear when we consider subsequent limits (namely when we consider  $\mu_t^{M, \ka}$, $\mu_t^{\ka}$ satisfying \eqref{PDEmuMK}, \eqref{PDEmuK}, respectively) and their weighted marginals ($\rho_t^{M, \ka}$, $\rho_t^{\ka}$ satisfying \eqref{PDErhoMK}, \eqref{PDErhoK}, respectively).  Therefore we can interpret the idea of fixing the noise as follows: we are interested in the limiting behaviour of a measure, $\rho^{N, \ep, M, \ka}$, which we can be constructed  as the (weighted) marginal of many potential probability distributions; what we do by fixing the noise is to realise $\rho^{N, \ep, M, \ka}_t$ and its limit $\rho_t$  (as well as its `intermediate limits' $\rho_t^{M, \ka}$ and $\rho_t^{\ka}$) as (weighted) marginals of a specific joint distribution, which we choose in a way to simplify the proof.  
\end{note}
%\begin{note}
%Although at first glance may be unclear to see; there is a strong link between \eqref{PDEmuMKE} and \eqref{PDErhoMKE}. Indeed, if $\{\mu_t^{\e,M,\ka}\}_{t \in [0,T]}$ is a weak solution to \eqref{PDEmuMKE} then $\{\rho_t^{\e,M,\ka}\}_{t \in [0,T]}$ is a weak solution to \eqref{PDErhoMKE}. Indeed, if we define 
%\begin{equation*}
%    \langle f,\rho_t^{\e,M,\ka}\rangle:=\itr af(x)\mu_t^{\e,M,\ka}(dxda),\,\,\,f \in C^{\infty}(\T;\R),\,\,\, t \in [0,T].
%\end{equation*}
%Then, by considering the test function $\Psi(x,a)=af(x)$, $(x,a) \in \T \times \R$, with $f \in C^{\infty}(\T;\R)$
%\end{note}
$\bullet$ {\bf Limit $\e \downarrow 0$.}
 Let us begin with noting that since the mollifiers $\{\Phi_{\e}\}_{\e>0}$ converge to the Dirac delta as $\e \downarrow 0$, if we assume  that $\mu_t^{\e,M,\ka}$ converges to some limit $\mu_t^{M,\ka}$ and hence 
 $\zeta_t^{\e,M,\ka}$ converges to a function, say $\zeta_t^{M,\ka}$, then we would obtain $\Phi_{\e}*\zeta_t^{\e,M,\ka} \stackrel{\e \rightarrow 0}{\longrightarrow} \zeta_t^{M,\ka}$. So, intuitively,  in the limit $\e \downarrow 0$, we expect  the stochastic process $(X_t^{\e,M,\ka},A_t^{\e,M,\ka})$ in \eqref{sistemaMKE} to converge to $(X_t^{M,\ka},A_t^{M,\ka})$,  solution to the following system:  
\begin{equation}\label{sistemaMK}
    \begin{dcases}    
    & X_t^{M,\ka}\,=\, X_0\,-\,\int_0^t \left(V'(X_s^{M,\ka})+\Gamma_M(X_s^{M,\ka},\mu_s^{M,\ka})\right)\,ds\,+\, 
         \sqrt{2}\,\beta_t,  \\
    & A_t^{M,\ka}\,=\, A_0\,+\,\int_0^t \frac{q(X_s^{M,\ka})}{\zeta_s^{M,\ka}(X_s^{M,\ka})}\,d\mY^{\kappa}_s \, , 
    \end{dcases}
\end{equation}
where  $\mu^{M, \ka}_t= \cL(X_t^{M, \ka}, A_t^{M, \ka})$. This is the correct `guess' but, for the purposes of the proof, we need to be careful how we define $\zeta_t^{M, \ka}$. Indeed, if we were to proceed as in the previous step we would now define  $\zeta_t^{M,\ka}$ to be the $\T$-marginal of  $\mu_t^{M,\ka}$,  namely,
\begin{equation*}\label{zetaMK}
    \zeta_t^{M,\ka}(dx) :=\int_{\R} \mu_t^{M,\ka}(dxda),\,\,\,\,\,t \in [0,T] \,.
    %\mE\left( f(X_t^{\e,M,\ka}\right),\,\,\,\,
\end{equation*}
We would then observe that 
 $\mu_t^{M, \ka}$ solves the following PDE
\begin{equation}\label{PDEmuMK}
    \begin{dcases}
        & \pa_t\mu_t^{M,\ka}=\paxx\mu_t^{M,\ka}+\pa_x \left[ (V^{'}+\Gamma_M(x,\mu_t^{M,\ka})\mu_t^{M,\ka}\right ]-\frac{q(x)}{\zeta_t^{M,\ka}}\pa_a\mu_t^{M,\ka}\pa_tY_t^{\ka}\\
        & \mu^{M,\ka}|_{t=0}=\mu_0 , 
    \end{dcases}
\end{equation}
and hence  $\zeta_t^{M,\ka}$ solves the Cauchy problem 
\begin{equation}\label{zetaPDEMK}
    \begin{dcases}
        & \pa_t\zeta_t^{M,\ka}=\pa_{xx}\zeta_t^{M,\ka}+\pa_x \left [ \left (V^{'}+\Gamma_M(x,\mu_t^{M,\ka})\right) \zeta_t^{M,\ka}\right ] \\
        & \zeta^{M,\ka}|_{t=0}=\zeta_0 \, .  
    \end{dcases}
\end{equation}
Again,  the PDE satisfied by $\zeta^{M, \ka}_t$ can be formally deduced by integrating \eqref{PDEmuMK} with respect to the variable $a \in \R$. If we were to proceed in this way then equation \eqref{PDEmuMK} would have the same structure as the PDE \eqref{PDEmuMKE}; in particular, this would be a closed equation in $\mu^{M, \ka}_t$ (as before, $\zeta^{M, \ka}_t$ would depend on  $\mu_t^{M, \ka}$ only). However, with these definitions, we would not know how to study the well-posedness of equation \eqref{PDEmuMK}, because now in the denominator of the last addend of \eqref{PDEmuMK} there is no mollifier, so  we don't know whether the denominator is strictly positive. In particular, the assumptions of \cite{coghi2019stochastic}, which we use to prove well-posedness of \eqref{PDEmuMKE}, are not satisifed in this case.  To overcome this issue, and guided by the intuition that if $\zeta_t^{M, \ka}$ was defined as the marginal of the joint law $\mu_t^{M, \ka}$ then $\zeta_t^{M, \ka}$ would solve the PDE \eqref{zetaPDEMK}, what we will prove is that $\mu_t^{\e,M, \ka}$ converges, as $\e\rightarrow 0$, to $\mu_t^{M, \ka}$, which is the law of the process $(X_t^{M, \ka}, A_t^{M, \ka})$, where in equation \eqref{sistemaMK} we define $\zeta_t^{M, \ka}$ as the solution of the PDE \eqref{zetaPDEMK}, see Definition \ref{DefMK}. We prove this fact in Section \ref{section_limit_e}, see Proposition  \ref{prop_e_before_M} in particular. The advantage of defining $\zeta_t^{M,\ka}$ as solution of the PDE  \eqref{zetaPDEMK} comes from the fact that, from classical PDE theory, we easily know that the solution to \eqref{zetaPDEMK} is positive (see Lemma \ref{wellpossistemaMK}), fact which makes the study of the well-posedness of \eqref{PDEmuMK}-\eqref{zetaPDEMK}  (see Proposition \ref{prop:law satisfies PDE}) easily accessible.  

%%%%%%%%%%%%%%%torna 55

%Lastly, the measure $\{\mu_t^{\e,M,\ka}\}_{t \in [0,T]}$ solves the PDE below (which can be formally obtained as in \cite{coghi2019stochastic}),

Analogously to what we did when  establishing the limit $N\rightarrow \infty$, we now define the following measure 
\begin{equation}\label{rhoMK}
\rho_t^{M,\ka}(dx):=\int_{\R} a\mu_t^{M,\ka}(dxda),\,\,\,\,t \in [0,T],    
\end{equation}
%and observe that
%\begin{equation*}
%    \langle f,\rho_t^{M,\ka}\rangle=\mE \left (A_t^{M,\ka}f(X_t^{M,\ka}) \right ),\,\,\,\,f \in C^{\infty}(\T;\R),\,\,t \in [0,T].
%\end{equation*}
which is a signed measure on the torus and it is the object we are actually interested in.  
In   Section \ref{section_limit_e} we prove that the following holds.  
%the convergence of $(X_t^{\e,M,\ka},A_t^{\e,M,\ka})$ to $(X_t^{M,\ka},A_t^{M,\ka})$ implies the following Proposition \ref{prop limit rho eps} the proof of which is contained in Section \ref{section_limit_e}.
\begin{prop}\label{prop limit rho eps}
Fix  $ M , \ka \in \N $  and let $\{ \rho_t^{\e,M,\kappa} \}_{\e>0} \subset \cM(\T)$ and $\rho_t^{M,\kappa} \in \cM(\T)$ be the measures defined  in \eqref{rhoMKE}  and \eqref{rhoMK}, respectively. Then, for each  $ M , \ka \in \N $ fixed,  the following limit holds 
\begin{equation}\label{limiteeps}
    \lim_{\e\downarrow 0}\, \,\,\,\sup \limits_{t \in [0,T]}\,\,\, \left |\langle f,\rho_t^{\e,M,\kappa} \rangle\,-\, \langle f, \rho_t^{M,\ka} \rangle \right |=0,
\end{equation}
for any given $f\in C^\infty(\T;\R)$. Furthermore $\rho_t^{M,\ka}$ is a weak measure-valued (and, a posteriori, classical) solution to the PDE
\begin{equation}\label{PDErhoMK}
    \begin{dcases}
        &\pa_t\rho_t^{M,\ka}=\paxx\rho_t^{M,\ka}+\pa_x \left[ (V^{'}+\Gamma_M(x,\mu_t^{M,\ka})\rho_t^{M,\ka}\right ]+q(x)\pa_tY_t^{\ka}\\
        & \rho^{M,\ka}|_{t=0}=\rho_0 \,. 
    \end{dcases}
\end{equation}
%where $\rho_0 \in \cM(\T)$ is defined in terms of the initial condition $(X_0,A_0)$ of \eqref{sistemaMKE} as 
%\begin{equation*}
%    \rho_0(dx)=\int_{\R} a\mu_t^{M,\ka}(dxda) \, .
%\end{equation*}
 As a consequence,  the solution  $\rho_t^{\e, M, \ka}$  of the  PDE \eqref{PDErhoMKE} converges (in the sense \eqref{limiteeps}) to  $\rho_t^{M, \ka}$, solution of the PDE \eqref{PDErhoMK}.  
\end{prop}
 
\begin{note}
    In comparing the PDE \eqref{PDErhoMKE} solved by $\rho_t^{\e, M,\ka}$ and the PDE \eqref{PDErhoMK} solved by $\rho_t^{M,\ka}$, it should be noted that   the forcing term of the former equation is not purely additive,   while the one of equation  \eqref{PDErhoMK}  is. So,  as expected from the heuristics of Section \ref{subsec:heuristics}, see \eqref{additive}, it is in taking the limit $\e \downarrow 0$ that we 'simplify' the multiplicative term in front of the forcing term and obtain a pure additive forcing.

\end{note}
$\bullet$ {\bf The limit $M\to \infty$.}
In this third step we let the cut-off parameter $M$ to infinity. Informally, since  $\Gamma_M$ (defined in \eqref{def_GammaM}) converges, as $M \to \infty$, to the function $\Gamma$ given by 
%Let us note that to get the picture of the limit we need to study the limit of $\Gamma_M(x,\mu)$,\,\,$x \in \T$, $\mu \in \cP_2(\T \times \R)$, as $M \to \infty$. As one can clearly see, by introducing the function $\Gamma$ defined as 
\begin{equation}\label{gamma senza M}
\Gamma(x,\mu):=\itr aF^{'}(x-y)\mu(dyda),\, \qquad \quad x \in \T,\,\,\,\mu \in \cP_2(\T \times \R),
\end{equation} 
%namely, $\Gamma_M \xrightarrow[M \to \infty]{} \Gamma$. Hence, 
assuming that the pair $(\mu_t^{M,\ka},\zeta_t^{M,\ka})$ converges to some limit $(\mu_t^{\ka},\zeta_t^{\ka})$,  we would intuitively expect $\Gamma_M(x,\mu_t^{M,\ka}) \xrightarrow[]{M \to \infty} \Gamma(x,\mu_t^{\ka})$. Hence, we also expect  that the  stochastic process $(X_t^{M,\ka},A_t^{M,\ka})$, solution to \eqref{sistemaMK},  converges to a stochastic process,  $(X_t^{\ka},A_t^{\ka})$,  solution to the following SDE:
\begin{equation}\label{sistemaK}
\begin{dcases}
   & X_t^{\ka}\,=\, X_0\,-\,\int_0^t \left(V'(X_s^{\ka})+\Gamma(X_s^{\ka}, \mu_s^{\ka})\right)\,ds\,+\, 
         \sqrt{2}\,\beta_t  \\
     & A_t^{\ka}\,=\, A_0\,+\,\int_0^t \frac{q(X_s^{\ka})}{\zeta_s^{\ka}(X_s^{\ka})}\,d\mY_s^{\kappa} \, ,     
   \end{dcases}
\end{equation}
where $\mu_t^{\ka}=\mathcal L (X_t^{\ka}, A_t^{\ka})$ and, similarly to \eqref{sistemaMK}, the function $\zeta_t^{\ka}$ is defined as the (unique) classical solution to the (linear) PDE
%\begin{equation}\label{zetaK}
%    \zeta_t^{\ka}(dx) = \int_{\R} \mu_t^{\ka}(dxda) \,.
%\end{equation}
\begin{equation}\label{PDEzetaK}
    \begin{dcases}
        & \pa_t\zeta_t^{\ka}=\pa_{xx}\zeta_t^{\ka}+\pa_x \left [ (V^{'}+\Gamma(x,\mu_t^{\ka})\zeta_t^{\ka}\right ] \\
        & \zeta^{\ka}|_{t=0}=\zeta_0.
    \end{dcases}
\end{equation}
This turns out to be the correct guess and we prove this fact in Proposition \ref{thm limit M}. 
Analogously to previous steps, it can be shown (see Proposition \ref{prop:law satisfies PDE}) that if $\mu_t^{\ka}$ is the law of the  McKean-Vlasov SDE \eqref{sistemaK},  then $\mu_t^{\ka}$ is a measure-valued solution of the following equation: 
\begin{equation}\label{PDEmuK}
    \begin{dcases}
        & \pa_t\mu_t^{\ka}=\paxx\mu_t^{\ka}+\pa_x \left[ (V^{'}+\Gamma(x,\mu_t^{\ka})\mu_t^{\ka}\right ]-\frac{q(x)}{\zeta_t^{\ka}}\pa_a\mu_t^{\ka}\pa_tY_t^{\ka}\\
        & \mu^{M,\ka}|_{t=0}=\mu_0,
    \end{dcases}
\end{equation}
where, for clarity, we iterate that  $\zeta_t^{\ka}$ is the  solution to \eqref{PDEzetaK}.
As in previous sections, let us set 
\begin{equation}\label{rhoK}
    \rho_t^{\ka}(dx):=\int_{\R} a\mu_t^{\ka}(dxda) \,  
\end{equation}
and  notice that, since the drift $\Gamma_M$ has been `replaced', in the limit,  by the drift $\Gamma$, we have 
\begin{align}\label{gamma-rho}
   \Gamma(x,\mu_t^{\ka})=\int_{\mathbb T \times \R} aF^{'}(x-y)\mu_t^{\ka}(da dy)  = \int_{\T} F^{'}(x-y) \rho^{\ka}_t(dy)=F^{'}*\rho_t^{\ka}(x),\qquad x \in \T \,. 
\end{align}
%Multiplying by $a \in \R$ the equation \eqref{PDEmuK} and then integrating over $a\in \R$ and
Using  \eqref{gamma-rho}, one expects  $\rho_t^{\ka}$ to be the solution of equation \eqref{PDErhoK} below.  In Section \ref{limit M to infty} we then prove that the following result holds. 
\begin{prop}\label{prop limit rho M}
Let $\ka \in \N$ be fixed.  Let  $\{\rho_t^{M,\ka}\}_{M \in \N}$ be the family of measures defined in \eqref{rhoMK} and $\rho_t^{\ka}$ be as in \eqref{rhoK}. Then the following limit holds: 
\begin{equation*}\label{limit rho M}
    \lim_{M\to \infty}\, \,\,\,\sup \limits_{t \in [0,T]}\,\,\, \left |\langle f,\rho_t^{M,\kappa} \rangle \,-\, \langle f, \rho_t^{\ka} \rangle \right |=0, \,
\end{equation*}
for any given $f \in C^{\infty}(\T;\R)$. 
Furthermore, $\rho_t^{\ka}$ is a weak measure-valued (and, a posteriori, classical) solution to the following PDE:  
\begin{equation}\label{PDErhoK}
    \begin{dcases}
       &\pa_t\rho_t^{\ka}=\pa_{xx}\rho_t^{\ka}+\pa_x \left[  (V^{'}+F^{'}*\rho_t^{\ka})\rho_t^{\ka} \right ]+q(x)\pa_t\mY_t^{\ka}\\
       & \rho_t^{\ka}|_{t=0}=\rho_0,
    \end{dcases}
\end{equation} 
As a consequence, the solution $\rho_t^{M,\ka}$ of the PDE \eqref{PDErhoMK} converges, as $M\rightarrow \infty$,  to $\rho_t^{\ka}$ solution to the PDE \eqref{PDErhoK}.
\end{prop}
Equation \eqref{PDErhoK} is a closed equation for $\rho_t^{\ka}$, as opposed to the evolutions for  $\rho_t^{\ep, M, \ka}$ and $\rho_t^{M, \ka}$, equations \eqref{PDErhoMKE} and \eqref{PDEmuMK}, respectively, which are not closed in their respective unknowns,  $\rho_t^{\ep, M, \ka}$ and $\rho_t^{M, \ka}$.

%%%come back here 
To prove Proposition \ref{prop limit rho M},  we will use  the fact (due  to \eqref{gamma-rho}) that the system \eqref{sistemaK}-\eqref{PDEzetaK} can be recast into an equivalent form, where the joint distribution $\mu_t^{\ka}$ does not directly appear. That is, the evolution \eqref{sistemaK}-\eqref{PDEzetaK} is equivalent to the following \begin{equation}\label{sistemaK-rho}
    \begin{dcases}
        & X_t^{\ka}\,=\, X_0\,-\,\int_0^t \left(V'(X_s^{\ka})+F^{'}*\rho_s^{\ka}(X_s^{\ka})\right)\,ds\,+\, 
         \sqrt{2}\,\beta_t,  \\
     & A_t^{\ka}\,=\, A_0\,+\,\int_0^t \frac{q(X_s^{\ka})}{\zeta_s^{\ka}(X_s^{\ka})}\,d\mY_s^{\kappa},\\
        & \pa_t\zeta_t^{\ka}=\pa_{xx}\zeta_t^{\ka}+\pa_x \left [ (V^{'}+F^{'}*\rho_t^{\ka})\zeta_t^{\ka}\right ], \\
        & \pa_t\rho_t^{\ka}=\pa_{xx}\rho_t^{\ka}+\pa_x \left[  (V^{'}+F^{'}*\rho_t^{\ka})\rho_t^{\ka} \right ]+q(x)\pa_t\mY_t^{\ka}, \\
        & \zeta^{\ka}|_{t=0}=\zeta_0,\,\,\rho^{\ka}|_{t=0}=   \rho_0. 
    \end{dcases}
\end{equation}
As one can see, the perk of system \eqref{sistemaK-rho} is that the evolution of the pair $\left ( X_t^{\ka}, A_t^{\ka} \right )$ is now  linear, as in the above $\rho_t^{\ka}$ is just the solution of a PDE that can be solved independently of all the other equations in the system.  Of course, this `linearity' comes at the price of having to increase dimensionality,  as system \eqref{sistemaK-rho} comes with the additional equation for $\rho_t^{\ka}$ to the set of SDEs and PDEs governing the system. The equivalence between the system \eqref{sistemaK}- \eqref{PDEzetaK} and system \eqref{sistemaK-rho} is rather straightforward, and explained in detail  in Section \ref{limit M to infty}, see Note \ref{equivalenza tra i due sistemi rhok}.

$\bullet$ {\bf Limit $\ka \to \infty$.}
We are left with showing the convergence of $\{\rho_t^{\ka}\}_{\ka \in \N}$ to $\rho_t$ solution to \eqref{rough PDE}. This is the content of the proposition below, the proof of which is contained in Section \ref{limit kappa}.
\begin{prop}\label{prop limit rho kappa}
Let  $\rho_t^{\ka}$ be the  solution to the PDE \eqref{PDErhoK}; then $\rho_t^{\ka}$  converges,  as $\ka \to \infty$,  to $\rho_t$ solution to \eqref{rough PDE}, in the sense that the following limit holds:
\begin{equation*}\label{limit rho K}
    \lim_{\ka \to \infty}\,\,\,\,\sup \limits_{t \in [0,T]}\,\,\, \left |\langle f,\rho_t^{\kappa} \rangle \,-\, \langle f , \rho_t \rangle \right |=0 \ ,
\end{equation*}
for any given $f \in C^{\infty}(\T;\R)$.    
\end{prop}
\begin{note}\label{note:commutativity limits}
Let us  make further comments on our results and on the strategy of proof. 
    \begin{itemize}
    \item We expect all four limits i.e. $\lim \limits_{\ka \to \infty}$, $\lim \limits_{M \to \infty}$, $\lim \limits_{\e \downarrow 0}$ and $\lim \limits_{N \to \infty}$ to commute; the order in which we take them is the one we found easiest, and crucial to our proof is to take the limit $\ka \rightarrow \infty$ last. If we took it sooner then   the Stratonovich second order  correction term  in \eqref{PDE1with additional second order term} would appear, presenting the difficulties we have discussed in Note  \ref{note:cutoff}. On this point, it is important to note that while, when taking the limits in $N, \ep$ and $M$, the limit for $\rho_t^{N, \ep, M, \ka}$, $\rho_t^{\ep, M, \ka}$ and $\rho_t^{M, \ka}$ was deduced from the limit for the joint distribution $\mu_t^{N, \ep, M, \ka}$, $\mu_t^{\ep, M, \ka}$ and $\mu_t^{M, \ka}$, respectively, in the last step, when we let $\ka$ to infinity, we never study the limit for $\mu_t^{\ka}$, and instead we study directly the limit of $\rho_t^{\ka}$. This is due to two things: firstly, the equation \eqref{PDErhoK} for $\rho_t^{\ka}$ and the equation \eqref{rough PDE} for $\rho_t$ have the same structure - in particular, they are both closed equations in the respective unknowns, so that the joint distributions $\mu_t^{\ka}$ and $\mu_t$ do not appear in such equations, hence there is no need to study their convergence. Moreover, if we were to study the limit for $\mu_t^{\kappa}$, then one can see that $\mu_t^{\ka}$ would converge to $\mu_t$, solution of the following equation: 

    \begin{equation}\label{rough PDE mu}
     \pa_t\mu_t=\pa_{xx}\mu_t+\frac{1}{2}\frac{q^2(x)}{\zeta_t^2(x)}\pa_{aa}\mu_t+\pa_x[(V^{'}+\Gamma(x,\mu_t))\mu_t]-\frac{q(x)}{\zeta_t(x)}\pa_a \mu_t \pa_t\mY 
         \end{equation}
     where $\zeta_t$ solves the following PDE
     \begin{equation*}
\pa_t\zeta_t=\pa_{xx}\zeta_t+\pa_x[(V^{'}+\Gamma(x,\mu_t))\zeta_t] \,.       
\end{equation*}  
In particular, and as expected, since the equation for $\mu^{\ka}$ is multiplicative, the  second order correction term does appear in \eqref{rough PDE mu}. Studying the evolution \eqref{rough PDE mu} would also require the use of rough path theory.  
    \item Still on the order in which we take limits, as we said we take here limits in a specific order and, especially, we take them one at a time. We refer the reader to Note \ref{note uniform bound in eps and M} to explain the detail of why we do so, but, in essence, this is because we were only able to obtain estimates which are uniform in one parameter at a time. There is a question about whether these limits can be taken together. If one was to try and do so  then we believe that $\varepsilon$ should be chosen to scale with $N$. 
    %and, moreover, for those   
    %limits, as shown in Section \ref{section_limitN}, Section \ref{section_limit_e} and Section 
%\ref{limit M to infty} the result holds for the joint distribution $\mu^{\e,M,\ka}$. As for the parameter $\ka \downarrow 
%    0$ it is crucial that $\lim \limits_{\ka \to \infty}$ is taken as the last limit. In this way, since in the 
 %   PDE of interest to us (see \eqref{PDErhoK}) the driving forcing  is purely additive, no correction term arises 
  %  due to the H\"older regularity of the forcing $Y$ in the limit $\ka \to \infty$. On the contrary, as shown 
   % in $(ii)$ below we expect the correction term to appear in the limiting PDE of the joint distribution $\{\mu_t^{\ka}\}_{t \in [0,T]}$. For this reason, we restrict the treatment of the last limit ($\lim \ka \to \infty$) to the weighted measure $\{\rho_t^{\ka}\}_{t \in [0,T]}$.      
\end{itemize}
\end{note}

We conclude this section with the  statement of Proposition  \ref{prop:law satisfies PDE}, the proof of which is postponed to Appendix \ref{well-posedness}. 

\begin{prop}\label{prop:law satisfies PDE}
    Let $\mu_t^{\ep, M, \ka}$ ($\mu_t^{M, \ka}$, $\mu_t^{\ka}$, respectively) be defined as the law of $(X_t^{M,\ka},A_t^{M,\ka})$ $\left ((X_t^{\ka},A_t^{\ka}),\,(X_t^{\ka},A_t^{\ka}),\, \text{respectively}\right )$ solution to the system \eqref{sistemaMKE} (\eqref{sistemaMK}, \eqref{sistemaK}, respectively). Then $\mu_t^{\ep, M, \ka}$ is a weak measure valued solution to  \eqref{PDEmuMKE} (equation \eqref{PDEmuMK}, \eqref{PDEmuK}, respectively). Furthermore $\mu_t^{\ep, M, \ka}$, $\mu_t^{M, \ka}$ and $\mu_t^{\ka}$ belong to  the space $C([0,T];\cP_2(\T\times \R))$. \footnote{For clarity we recall that \eqref{sistemaMKE}  is intended as a closed equation for $\mu_t^{\ep, M, \ka}$, while \eqref{PDEmuMK} and  \eqref{PDEmuK} are not closed equations; so, when we say that $\mu_t^{M, \ka}$ is a weak measure-valued solution to  \eqref{PDEmuMK} we really mean that the pair $(\mu_t^{M, \ka}, \zeta_t^{M, \ka})$ is a solution to  \eqref{PDEmuMK}-\eqref{zetaPDEMK}, with \eqref{PDEmuMK} to be intended in weak measure-valued sense and \eqref{zetaPDEMK} in classical sense. Similarly for the pair \eqref{PDEmuK}-\eqref{PDEzetaK}. Moreover we point out that while the notion of weak measure valued solution for $\mu_t^{\ep, M, \ka}$ and $\mu_t^{M, \ka}$ can be given as at the beginning of Section \ref{section: main results} i.e. in particular in the space  $\mathcal M(\T)$, the notion of weak measure-valued solution for $\mu_t^{\ka}$ is intended in $\mathcal P_1(\T \times \R)$.}
\end{prop}

%%%%%torna 4

\section{Notation and Preliminaries}\label{Notation}
In this section we set the notation and state some background results needed later in the paper. Let us consider a complete and separable metric space $(\cX,d_{\cX})$ and let $\cP(\cX)$ be the space of probability measures on $\cX$. Unless otherwise stated the $\sigma$-algebra provided on $\cX$ is the Borel $\sigma$-algebra $\cB(\cX)$. Following \cite[Section 2 and Theorem 3.3]{lacker2018mean} we denote by $\cP_1(\cX)$ and $\cP_2(\cX)$ the space of probability measures $\mu$ on $\cX$ such that
\begin{equation*}
    \int_{\cX} d_{\cX}(x,x_0)\,\mu(dx)\,<\,\infty\ ,
\end{equation*}
and 
\begin{equation}\label{finite second moment}
    \int_{\cX} d_{\cX}(x,x_0)^2\,\mu(dx)\,<\,\infty\ ,
\end{equation}
respectively, for any arbitrary reference point $x_0\in\cX$.\footnote{A straightforward calculation shows that \eqref{finite second moment} does not depend on the choice of the point.} 
Given two probability measures $\mu,\nu\in \cP_1(\cX)$ (or $\cP_2(\cX)$), we introduce the space
\begin{multline*}
    \Pi_{\cX}(\mu,\nu):=\Big\{ \pi\in\cP\left( \cX \times \cX \right)\big| \pi(S\times \cX)= \mu(S),\; \pi(\cX \times S)=\nu(S)\,,\,\mathrm{for\; every\; Borel\;set\;} S\subseteq \cX \Big\}\ ,
\end{multline*}
and define the 2-Wasserstein metric on $\cP_2(\cX)$ by
\begin{equation}\label{2-Wasserstein}
    \cW_{\cX,2}^2(\mu,\nu)\,:=\,\inf_{\pi\in\Pi_{\cX}(\mu,\nu)}\int_{\cX \times \cX} d_{\cX}(x,y)^2\,\pi(dxdy).
\end{equation}
Equivalently, we can write
\begin{equation*}
    \cW_{\cX,2}^2(\mu,\nu)\,=\, \inf_{X\sim \mu,\,\overline{X}\sim \nu}\E_{\mu,\nu}\left[d_{\cX}(X,\overline{X})^2\right],
\end{equation*}
where $X$ and $\overline{X}$ are $\cX$-valued random variables with given 
distribution $\mu$ and $\nu$, respectively, and $\mE_{\mu,\nu}$ denotes the expectation with respect to the joint law of $\mu$ and $\nu$. An analogous definition for the 1-Wasserstein distance $\cW_1$ on $\cP_1(\cX)$ holds, see \cite{villani2009optimal}.
We will often use a dual formulation of the Wasserstein distance due to Kantorovich (cfr. \cite[Theorem 2.15]{lacker2018mean} and \cite[Theorem 5.10]{villani2009optimal}), namely,  \begin{align}
\cW_{\cX,1}(\mu,\nu)= \sup \left \{ \langle f,\mu\rangle-\langle f,\nu \rangle | f,g \in \text{Lip}_1(\cX;\R) \right \} \label{kantorovich}    
\end{align}
where Lip$_1(\cX;\R)$ denotes the class of real-valued Lipschitz continuous functions on $\cX$ with Lipschitz constant less than or equal to 1 while for $\cW_{\cX,2}$ we have
\begin{align}\label{dual characterization}
\cW_{\cX,2}(\mu,\nu)= \sup \left \{ \int_{\cX} f(x)\mu(dx)+\int_{\cX}  g(x) \nu(dx)  \bigg | f \in L^1(\mu),g \in L^1(\nu), f(x)+g(y) \leq d_{\cX}^2(x,y)  \right \},    \notag
\end{align}
where $L^1(\mu)$, $L^1(\nu)$ denote the space of functions on $\cX$, integrable w.r.t $\mu$ and $\nu$, respectively. Moreover, from Jensen's inequality we obtain 
\begin{equation}\label{W1 leq W2}
    \cW_{\cX,1}(\mu,\nu) \leq \cW_{\cX,2}(\mu,\nu),\qquad\text{for any $\mu,\nu \in \cP_2(\cX)$}.
\end{equation}
We will often use the following characterization for the convergence in the 2-Wasserstein distance $\cW_{\cX,2}$. In what follows $C\left( \cX;\R\right)$ denotes the space of real-valued continuous functions on $\cX$.
\begin{prop}\label{weak formulation on TR}\cite[Theorem 2.13]{lacker2018mean}
%\footnote{We state the following proposition in the space $\cP_2(\T \times \R)$ but this result holds for any separable metric space $X$ in place of $\T \times \R$.}
Let $(\cX,d_{\cX})$ be a complete and separable metric space and let $\{\mu^n\}_{n \in \N},\,\,\mu \subset \cP_2(\mathcal{X})$ then the following conditions are equivalent:
\begin{enumerate}[(i)]
    \item $\cW_{\cX,2}(\mu^n,\mu) \xrightarrow[n \to \infty]{}0$
    \item For any $f \in C(\cX;\R)$ for which there exists $x_0 \in \cX$ and $c>0$ such that $|f(x)| \leq c\,\,d_{\cX}(x,x_0)^2$, for every $x \in \cX$, the following holds:
    \begin{equation*}
        \left | \langle f,\mu^n \rangle-\langle f, \mu \rangle \right | \xrightarrow[]{n \to \infty} 0.
    \end{equation*}
\end{enumerate}
\end{prop}
%\begin{prop}\label{weak formulation on C_T}\cite[Theorem 2.13, pp. 14-16]{lacker2018mean} 
%\footnote{We state the following proposition in the space $\cP_2(\T \times \R)$ but this result holds for any separable metric space $X$ in place of $\T \times \R$.}
%Let $\{\mu^n\}_{n \in \N},\,\,\mu \subset \cP_2(C_T)$ then the following holds:
%\begin{enumerate}[(i)]
%    \item $\cW_{2,T}(\mu^n,\mu) \xrightarrow[n \to \infty]{}0$
%    \item For any smooth function $\Psi \in C^{\infty}(\T \times \R)$ such that $|\Psi(x,a)| \leq c(1+a^2)$, $x \in \T$, $a \in \R$ (where $c>0$ is a suitable constant) the following holds:
%    \begin{equation*}
%        \sup \limits_{t \in [0,T]} |\llangle \Psi,\mu_t^n \rrangle - \llangle \Psi,\mu_t \rrangle | \xrightarrow[n \to \infty]{} 0.
%    \end{equation*}
%\end{enumerate}
%\end{prop}
%From this moment onward we set a 4-tuple $(\Omega,\mP,\cF,\{\cF_t\}_{t \geq 0})$ where $\mP$ is a probability measure on $\Omega$ with $\sigma$-algebra $\cF$ and $\{\cF_t\}_{t \geq 0}$ a filtration. 
In the next sections we mainly apply the above results to $\cX=\T \times \R$ endowed with the metric
\begin{equation*}
    d_{\T \times \R}((x,a),(y,b)):=\sqrt{\min\left\{|x-y|,\,2\pi-|x-y| \right\}^2\,+\,|a-b|^2},\qquad x,y \in \T, a,b \in \R,
\end{equation*}
and to $\cX=C_t:=C \left ( [0,t];\T \times \R\right )$ the space of time-continuous $\T \times \R$-valued functions defined up to a given time $t > 0$ endowed with the supremum metric
\begin{equation*}
    |(x_{\cdot},y_{\cdot})|_t:=\sup \limits_{s \in [0,t]}d_{\T \times \R}(x_s,y_s),\qquad    
\end{equation*}
having used in the above the notation $x_{\cdot}=\{x_s\}_{s \in [0,t]}$ for $x_{\cdot} \in C_t$, and similarly for $y_{\cdot} \in C_t$. Unless otherwise stated, we denote the 2-Wasserstein metric on $\cP_2(\T \times \R)$ and $\cP_2(C_t)$ by $\cW_2$ and $\Wt$, respectively. In addition to that, the spaces $\Pi_{\T \times \R}$ and $\Pi_{C_t}$ will be denoted by $\Pi_2$ and $\Pi_t$, respectively.
%Let $C_t:=C([0,t];\cX)$ be the space of time continuous $\cX$-valued functions defined up to a given time $t \geq 0$ and endowed with the supremum norm 
%\begin{equation}
%    |(x,y)|_t:=\sup \limits_{s \in [0,t]}d(x_s,y_s),\,\,x,y \in C_t,
%\end{equation}
%where $d$ is the metric of $\cX$.
%Then define the Wasserstein distance $\Wt$ on $\cP_2(C_t)$ by: 
%\begin{equation}\label{distanza traiettorie}
%   \Wt^2(\mu,\nu)=\inf \limits_{\pi \in \Pi_t(\mu,\nu)} \int_{C_t \times C_t} |(x,y)|_t^2 \,\,\pi(dxdy)\, 
%\end{equation}
%for any given $\mu, \nu \in \cP_2(C_t)$ and for any fixed $t \geq 0$, where $\Pi_t$ is defined as
%\begin{multline}\label{kuos 10}
 %   \Pi_t(\mu,\nu):=\Big\{ \pi\in\cP\left(C_t \times C_t \right)\,\big|\, \pi(S \times C_t)= \mu(S),\; \pi(C_t \times S)=\nu(S),\\
 %   \mathrm{for\; every\; Borel\;set\;} S\subseteq C_t \Big\},\,\,\,t \geq 0.
%end{multline}
%From \eqref{distanza traiettorie} the following property readily follows for any given $t>0$:
We also recall the following inequality between $\cW_2$ and $\Wt$, which we will use multiple times:
\begin{equation}\label{W2 leq W2t}
    \sup \limits_{s \in [0,t]}\cW_2^2(\mu_s,\nu_s) \leq \Wt^2(\mu_{\cdot},\nu_{\cdot}), 
\end{equation}
where in the above we have used the notation $\mu_{\cdot}=\{\mu_s\}_{s \in [0,t]}$ and $\nu_{\cdot}=\{\nu_s\}_{s \in [0,t]}$ for $\mu_{\cdot},\nu_{\cdot} \in C([0,t];\cP_2(\T \times \R))$ with $t \geq 0$.
When using the dual pairing between measures and functions we will need to distinguish between the case when the measure is on $\T $ and when it is on $\T\times \R$; when  $\rho \in \cM(\T)$ and $f \in C^{\infty}(\T;\R)$ we denote dual pairing by $\langle f,\rho \rangle$,  when $\mu \in \cP(\T \times \R)$ and $\Psi \in C_c^{\infty}(\T \times \R;\R)$ we  denoted it by $\langle \langle \Psi, \mu \rangle \rangle$.\\
As customary, we let 
\begin{align*}
& H^1(\T;\R):=\left \{ f \in L^2(\T;\R)\,\,\,:\,\,\, \pa_x f \in L^2(\T;\R)\right \}
%& W^{1,\infty}(\T;\R):=\left \{ f \in L^{\infty}(\T;\R)\,\,\,:\,\,\,\pa_x f \in L^{\infty}(\T;\R) \right \},
\end{align*}
where the spatial derivative $\pa_xf$ is taken in the weak sense and, for every $n >1$, 
\begin{equation}\label{sobolev usual}
    H^n(\T;\R):=\left \{ f \in L^2(\T;\R)\,\,\,:\,\,\, \pa_x f \in H^{n-1}(\T;\R)\right \}.
\end{equation}
We endow such spaces with the following norms:
\begin{align*}
    & |f|_{H^{n}(\T;\R)}^2:=\sum \limits_{i=0}^n |\pa_x^if|_{L^2(\T;\R)}^2,\,\,\,\, f \in H^n(\T;\R),\,\,\,\,n \in \N,
%    & |f|_{W^{1,\infty}(\T;\R)}:=|f|_{L^{\infty}(\T;\R)}+|\pa_xf|_{L^{\infty}(\T;\R)},\,\,\,\, f \in W^{1,\infty}(\T;\R).
\end{align*}
where $\pa_x^i f$ denotes the $i$-th weak derivative of $f$.
In the proofs we use frequently the notation $x \lesssim y$ for $x, y \in \R$, which means $x \leq Cy$ for some constant $C \geq 0$, the dependence of which on certain parameters is specified if necessary. For instance, by $x \lesssim_t y$ we mean that the constant $C$ depends on $t$. \\
Lastly, let us point out the following fact. In what comes next two types of integrals appear. The first one, with respect to piece-wise $C^1$-functions $\mY^{\kappa}$ and the second one with respect to H\"older continuous paths, namely
\begin{align*}
        &\int_0^t  f_s\,d \mY_s^{\ka},\, f \in C([0,T];\R)\,\,\text{and}\,\, \int_0^t  g_s\,d \mY_s,\,\, g \in C^{\infty}([0,T];\R).
\end{align*}
The first type of integral is intended in the Riemann-Stieltjes sense, while the second type of integral is intended in the Young sense. We refer the reader to Appendix \ref{young} for a brief recap on Young integral, see \cite[Sec 1.1]{lyons2007differential} for a comprehensive introduction. 

\section{The limit $N\to \infty$}\label{section_limitN}
The main purpose of this section is to prove Proposition \ref{prop limit rho N}. The proof of Proposition \ref{prop limit rho N} will be a  straightforward consequence of Proposition \ref{prop_limitN} below.    Hence, we focus our attention on proving the latter proposition, which is concerned with studying convergence, as $N\rightarrow \infty$,  of the empirical measure 
$\mu_t^{N,\e, M, \kappa}$ defined in \eqref{empirical_m} to $\mu_t^{\e,M,\ka}$, law of \eqref{sistemaMKE}. 
%we need some preliminary notions and results: Lemma \ref{well-posedness MKE} and Lemma \ref{gamma beta}. With this direction in mind, we first need a notion of solution for the system of SDEs \eqref{sistemaMKE}.
\begin{defn}\label{def_MV}
Let $(X_0,A_0): \Omega \to \T \times \R $ be a $\cF_0$-measurable random variable with distribution $\mu_0$ (which we recall satisfies the standing Assumption \ref{assunzioni sui dati iniziali}).
We say that the triple $t \to (X_t^{\e,M,\ka},A_t^{\e,M,\ka},\mu_t^{\e,M,\ka})\in C([0,T];\T \times \R \times \cP_2(\T \times \R))$ is a solution to \eqref{sistemaMKE} if
\begin{itemize}
    \item[(i)] $t \to (X_t^{\e,M,\ka},A_t^{\e,M,\ka}) \in \T \times \R$ is an $\{\cF_t\}_{t \in [0,T]}$-adapted stochastic process;
    \item[(ii)] The stochastic process $(X_t^{\e,M,\ka},A_t^{\e,M,\ka})$ solves \eqref{sistemaMKE}, for every $t\in[0,T]$, $\mP$-a.s. where $\mu_t^{\e,M,\ka}=\cL\left(X_t^{\e,M,\ka},A_t^{\e,M,\ka}\right)$,  $\Gamma_M$ has been defined in \eqref{def_GammaM} and $\zeta_t^{\e,M,\ka}$ is the $\T$-marginal of $\mu_t^{\e,M,\ka}$, as defined in \eqref{zetaMKE}. 
\end{itemize}
\end{defn}
\noindent 
From Proposition \ref{prop:law satisfies PDE} we know that $\mu_t^{\e,M,\ka}$ is the unique measure-valued solution to the PDE \eqref{PDEmuMKE}.
%while $\zeta_t^{\e,M,\ka}$ (see e.g. \cite{hormander} and \cite[Chapter 7]{evans2022partial}) is the unique classical solution to the PDE \eqref{PDEzetaMKE}. 
To streamline calculations we introduce the following operator  
\begin{align}
    \Xi_{\e}(x,\mu)\,:=\, \frac{q(x)}{\itr \Phi_\e(x-y)\mu(dyda)},\qquad x \in \T,\, \mu \in \cP_2(\T \times \R). \label{def_betaMe}
\end{align}
With this notation, the dynamics of the weights $A_t^{i,N,\e,M,\ka}$ in \eqref{particelle} and of the process $A_t^{\e,M,\ka}$ in \eqref{sistemaMKE} can be rewritten as 
%At first glance, since the dynamics of the weight $A_t^{i,N,\e, M,\ka}$ with $i=1,\cdots,N$ depends on the empirical measure $\zeta_t^{N,\e, M,\ka}$ in \eqref{empirical_zeta} one should expect the dynamics of the weight $A_t^{\e,M,\ka}$ to the depend on a probability measure $\zeta_t^{\e,M,\ka}$ which depends only on the spatial variable $x \in \T$ and not on the joint distribution $\mu_t^{\e,M,\ka}$. In a subtle manner that is what the operator $\Xi_{\e}$ entails. Indeed, we can recast the dynamics of the weight $ A_t^{\e,M,\ka}$ in the following way. 
\begin{align}\label{weight beta NMKE}
   & dA_t^{i,N,\e,M,\ka}=\Xi_{\e}(X_t^{i,N,\e,M,\ka},\mu_t^{N,\e, M,\ka})d\mY_t^{\ka},\,\,\,i=1,\cdots,N 
   \end{align}
   and 
   \begin{align}
   & dA_t^{\e,M,\ka}=\Xi_{\e}(X_t^{\e,M,\ka},\mu_t^{\e,M,\ka})d\mY_t^{\ka},\label{weight beta MKE}
\end{align}
respectively. Indeed, from definition \eqref{empirical_zeta} of $\zeta_t^{N,\e,M,\ka}$ and since the mollifier $\Phi_{\e}$ depends only the spatial variable $x \in \T$ we obtain
\begin{align*}
 \Phi_{\e}*\zeta_t^{N,\e,M,\ka}(x) & =\int_{\T} \Phi_{\e}(x-y)\,\zeta_t^{N,\e,M,\ka}(dy)\\
& =\frac{1}{N}\sum \limits_{i=1}^N \Phi_{\e}(x-X_t^{i,N,\e,M,\ka})=\frac{1}{N}\sum \limits_{i=1}^N \1_{\R}(A_t^{i,N,\e,M,\ka})\Phi_{\e}(x-X_t^{i,N,\e,M,\ka})\\
&=\itr \Phi_{\e}(x-y)\mu_t^{N,\e,M,\ka}(dyda),
%& =\mE \left( \Phi_{\e}(x-X_t^{i,N\e,M,\ka})\right)=\mE \left ( \1_{\R}(A_t^{\e,M,\ka} ) \Phi_{\e}(x-X_t^{\e,M,\ka}) \right )\\
%&=\itr \Phi_{\e}(x-y)\,\mu_t^{\e,M,\ka}(dyda)=\Xi_{\e}(x,\mu_t^{\e,M,\ka}),\,\,\,\,x \in \T,\,\, t \in [0,T],
\end{align*}
which gives \eqref{weight beta NMKE}. Equation \eqref{weight beta MKE} can be obtained similarly. Using \eqref{weight beta MKE}, the evolution \eqref{sistemaMKE} can be rewritten as follows
\begin{equation}\label{sistemaMKE+beta}
    \begin{dcases}
       & X_t^{\e,M,\ka}= X_0-\int_0^t \left [V^{'}(X_s^{\e,M,\ka}), +\Gamma_M(X_s^{\e,M,\ka},\mu_s^{\e,M,\ka}) \right ]\,ds+\sqrt{2}\beta_t,\\
       & A_t^{\e,M,\ka}= A_0+\int_0^t\Xi_{\e}(X_s^{\e,M,\ka},\mu_s^{\e,M,\ka})d\mY_s^{\ka}, \\
       %& dA_t^{\e,M,\ka}= \Xi_{\e}(X_t^{\e,M,\ka},\mu_t^{\e,M,\ka})d\mY_t^{\ka} \\
       %& \mu_t^{\e,M,\ka}=\cL(X_t^{\e,M,\ka},A_t^{\e,M,\ka})\\
       %& X^{\e,M,\ka}|_{t=0}=X_0,\,\,A^{\e,M,\ka}|_{t=0}=A_0 \,.
    \end{dcases}
\end{equation}
where $\mu_t^{\e,M,\ka}=\cL(X_t^{\e,M,\ka},A_t^{\e,M,\ka})$. 
\begin{lemma}\label{well-posedness MKE}
System \eqref{sistemaMKE+beta} (or equivalently \eqref{sistemaMKE}) admits a unique solution in the sense of Definition \ref{def_MV}. Moreover, the pair $(X_t^{\e,M,\ka},A_t^{\e,M,\ka})$ belongs to the space $L^2\left (\Omega; C([0,T];\T \times \R) \right )$.
\end{lemma}
We refer the reader to Appendix \ref{well-posedness} for the proof of Lemma \ref{well-posedness MKE}.
\begin{prop}[Particle approximation]\label{prop_limitN}
Let $M,\kappa \in \N$ and $\e,T>0$ be fixed. Let $(X_0,A_0)$ and $\mu_0$ as in Definition \ref{def_MV}. %and let the initial configuration of the particle-weight system \eqref{particelle} consist of $N$ i.i.d. random variables $(X \iin_0, A\iin_0)$, $i=1, \dots, N$, distributed according to  $\mu_0=\cL(X_0,A_0)$. 
Then the empirical distribution $\mu^{N,\e,M,\ka}_t$  of the particle system \eqref{particelle} (defined in \eqref{empirical_m}) converges to $\mu_t^{\e,M,\ka}$ (law of the pair $ (X_t^{\e,M,\ka},A_t^{\e,M,\ka})$ solution to the McKean-Vlasov SDE \eqref{sistemaMKE} ) in the Wasserstein $\WT$ distance, i.e. 
\begin{equation}\label{limit enne}
    \lim \limits_{N \to +\infty} \WT(\mu_{\cdot}^{N,\e, M, \ka},\mu_{\cdot}^{\e,M,\ka})=0,\,\,\,\,\mP-a.s.
\end{equation}
As a consequence, 
$$
\sup_{t \in [0,T]} \left[\left\vert X_t^{i, N,\e, M, \ka} - X_t^{\e, M, \ka}\right\vert^2+ 
\left\vert A_t^{i, N,\e, M, \ka} - A_t^{\e, M, \ka}\right\vert^2\right] \stackrel{N\rightarrow \infty}{\longrightarrow} 0, \quad \mathbb P - a.s. 
$$
%\cite[Theorem 5.3, p. 17]{coghi2019stochastic}
%{ \color{red} aggiungi nota in cui dici che la muNMk converge alla legge mueMk ed è l'unica soluzione di  citando coghi-gess}
\end{prop}
We first prove Proposition \ref{prop limit rho N} and then come back to the proof of Proposition \ref{prop_limitN}. 
\begin{proof}[Proof of Proposition \ref{prop limit rho N}]
%Proposition \ref{prop_limitN} implies in particular that the family of weighted empirical measures $\{\rho_t^{N,\e, M,\ka}\}_{t \in [0,T]}$ \eqref{weighted_empirical} converges weakly to the weighted distribution $\{\rho_t^{M,\ka,\e} \}_{t \in [0,T]}$ defined in \eqref{rhoMKE}. 
From \eqref{limit enne} and Proposition \ref{weak formulation on TR} (the latter with $\cX = C_T$) %and $(ii)$ of Proposition \ref{weak formulation on C_T} 
we know that 
\begin{equation}\label{limite muN}
    \lim \limits_{N \to \infty} \,\, \sup \limits_{t \in [0,T]} \left | \llangle \Psi, \mu_t^{N,\e, M,\ka} \rrangle- \llangle \Psi, \mu_t^{\e,M, \ka} \rrangle \right |=0,\,\,\,\,\mP-a.s.
\end{equation}
for any given $\Psi \in C^{\infty}(\T \times \R;\R)$ which grows at most quadratically with respect to the variable $a \in \R$. Moreover, from the definition of $\mu_t^{N,\e, M,\ka}$ and $\mu_t^{\e,M,\ka}$ we have 
\begin{align*}
   & \llangle \Psi,\mu_t^{N,\e, M,\ka} \rrangle=\frac{1}{N}\sum \limits_{i=1}^N\Psi(X_t^{i,N,\e, M,\ka},A_t^{i,N,\e, M,\ka}),\\
   & \llangle \Psi,\mu_t^{\e,M,\ka} \rrangle =\itr \Psi(x,a) \mu_t^{\e,M,\ka}(dxda)=\mE(\Psi(X_t^{\e,M,\ka},A_t^{\e,M,\ka})).       
\end{align*}
In particular, if we set $\Psi(x,a)=af(x)$, $f \in C^{\infty}(\T;\R)$, from the above, \eqref{limite muN}, \eqref{weighted_empirical} and \eqref{rhoMKE} we obtain 
\begin{equation*}
    \lim_{N\to\infty} \,\,\sup \limits_{t \in [0,T]} |\langle f,\rho_t^{N,\e, M,\ka} \rangle - \langle f,\rho_t^{\e,M.\ka} \rangle |=0,\,\,\mP-a.s.
\end{equation*}
for any given $f\in C^\infty(\T;\R)$. This proves the first claim in Proposition \ref{prop limit rho N}. 
For the second claim,  by recalling that $\mu_t^{\e,M,\ka}$ is the weak solution to \eqref{PDEmuMKE} and that $\rho_t^{\e,M,\ka}$ is given by \eqref{rhoMKE} we obtain %(in the upcoming calculation we use a formal writing and keep the derivatives on the measure but they are intended to be on the test function $f \in C^{\infty}(\T; \R)$) 
\begin{align*}
     \pa_t \left \langle f,\rho_t^{\e,M,\ka} \right \rangle & =\pa_t \llangle af,\mu_t^{\e,M,\ka} \rrangle = \left \langle \left \langle a \paxx f,  \mu_t^{\e,M,\ka} \rrangle\\
    & -\llangle a\pa_xf,\left [\left (V^{'}+\Gamma_M\left (x,\mu_t^{\e,M,\ka}\right )\right )\mu_t^{\e,M,\ka}\right ] \rrangle + \llangle \pa_a(af),\frac{q}{\Phi_{\e} * \zeta_t^{\e,M,\ka}} \mu_t^{\e,M,\ka} \pa_t\mY_t^{\ka} \rrangle \\
    %%%%%%%%%%%%%%%%%%%%%%%%%%%%%%%%
    & = \left \langle \paxx f , \rho_t^{\e,M,\ka} \right \rangle - \left \langle \pa_xf, \left [\left(V^{'}+\Gamma_M\left (x,\mu_t^{\e,M,\ka}\right )\right )\rho_t^{\e,M,\ka}\right ] \right \rangle \\
    & + \llangle f, \frac{q}{\Phi_{\e} * \zeta_t^{\e,M,\ka}}\mu_t^{\e,M,\ka} \rrangle \pa_t Y_t^{\ka} \, .\\
    %%%%%%%%%%%%%%%%%%%%%%%%%%5 
        \end{align*}
    Now, since $f$ depends on the variable $x \in \T$ only, we have
    \begin{align*}
    %\llangle \pa_a(af), \frac{q}{\Phi_{\e} * \zeta_t^{\e,M,\ka}}\mu_t^{\e,M,\ka} \rrangle \pa_t Y_t^{\ka} = 
    \llangle  f, \frac{q}{\Phi_{\e} * \zeta_t^{\e,M,\ka}}\mu_t^{\e,M,\ka} \rrangle \pa_t Y_t^{\ka} = \left \langle f, \frac{q\zeta_t^{\e,M,\ka}}{\Phi_{\e}*\zeta_t^{\e,M,\ka}} \right \rangle \pa_t Y_t^{\ka},
    \end{align*}
which shows that $\rho_t^{\ep,M,\ka}$ is a weak measure-valued solution  to \eqref{PDErhoMKE}. The fact that such a solution admits a density and that such a density is a weak, and then classical, solution to \eqref{PDErhoMKE} is a consequence of  classical PDE theory, see \cite[Chapter 7]{evans2022partial} or also \cite{hormander}, as equation \eqref{PDErhoMKE} is a linear parabolic PDE with a forcing which is $C^{\infty}$ in space and piece-wise $C^1$ in time. 
\end{proof}
We now state a preliminary lemma, Lemma \ref{gamma beta}, which will be used to prove Proposition \ref{prop_limitN}. The proof of Lemma \ref{gamma beta} is postponed to the end of this section.
\begin{lemma}\label{gamma beta}
Let $\Gamma_M$ and $\Xi_{\e}$ be as in \eqref{def_GammaM} and \eqref{def_betaMe}, respectively; then the following two inequalities hold for any given $x,y \in \T$, $\mu,\nu \in \cP_2(\T \times \R)$:
\begin{align}\label{ineq gammma}
    & \left|\pa_x^n\Gamma_M(x,\mu)-\pa_x^n\Gamma_M(y,\nu)\right|^2\,\leq K(n,M)(\,|x-y|^2\,+\, \cW^2_2(\mu,\,\nu)) \, \\
    & \left|\Xi_{\e}(x,\mu)-\Xi_{\e}(y,\nu)\right|^2\,\leq \Tilde{K}_{\e}( \,|x-y|^2\,+\, \cW^2_2(\mu,\,\nu)), \label{ineq beta}
\end{align}
where  the constants $K(n,M)$ and $\tilde{K}_{\e}$ are given by  
\begin{align*}
K(n,M) & :=\max \left \{M , 1 \right \} |F^{(n+1)}|_{L^{\infty}(\T;\R)},\\
 \Tilde{K}_{\e}  & := \frac{|q|_{L^{\infty}( \T;\R)}D_{\e}+|q'|_{L^{\infty}( \T;\R)}M_{\e}}{m_{\e}^2},
\end{align*}
respectively, with  $m_{\e}$, $D_{\e}$ and $M_{\e}$ defined in \eqref{lower bound phieps}, \eqref{upper bound derivative phieps} and \eqref{upper bound phieps}, respectively. 
\end{lemma}
\begin{proof}[Proof of Proposition \ref{prop_limitN}]
In what follows, to avoid notation overload and since we are considering the limit  $N \to \infty$ only, we omit the dependence on $M,\e,\ka>0$ for the quantities used in the proof. So, in this proof only, $X^{i,N}_t$ and $A^{i,N}_t$ are short notations for $X_t^{i,N, \e, M, \ka}$ and $A_t^{i,N, \e, M, \ka}$, respectively; similarly for $\mu_t^N$, which will be short for $\mu_t^{N, \e,M, \ka}$. Moreover, we do not track explicitly the dependence of the constants on $M,\e, \ka$ as this is not relevant to this proof. We denote by $C(t)$ a generic constant, depending continuously on $t$ (and increasing w.r.t time $t$), which may change from line to line. We also point out that unless otherwise stated, the inequalities (and equalities) of this proof will hold $\mP$-a.s. With this clarification,  we start by considering $N$ independent copies of the McKean-Vlasov system \eqref{sistemaMKE+beta}. Namely, for each $i \in\{1, \dots, N\}$,  let $(X_t^i, A_t^i)$ be the solution of the following problem  
\begin{align}\label{mckean-vlasov 2}
    \begin{dcases}
     &dX^i_t \,=\, - \left(V'(X^i_t) + \Gamma_M(X^i_t,\,\mu_t^i)\right)\,  dt\,+\,\sqrt{2}\,d\beta^i_t,  \\
     & dA^i_t\,=\, \Xi_{\e}(X^i_t,\,\mu_t)\,d\mY_t^{\ka} \ , \\
     & \mu_t^i=\cL(X_t^i,A_t^i), \\
     & X^i|_{t=0}=X_0^{i},\,\,A^i|_{t=0}=A_0^{i}.
\end{dcases}
\end{align}
For any given $i=1,\cdots,N$, the initial condition $(X_0^{i},A_0^{i})$ and the Wiener process $\beta_t^i$ coincide, respectively,  with the initial datum and the Wiener process, respectively,  appearing in the dynamics  \eqref{particelle}. Hence, as a consequence of Lemma \ref{well-posedness MKE} we obtain that
%that $\mu_t^i=\mu_t$, $t \in [0,T]$, where $\mu_t$ denotes the law of the solution process \eqref{sistemaMKE}. Consequently, 
the pairs $(X_t^i,A_t^i)$ are i.i.d. stochastic processes with same law $\mu_t = \mu_t^i$,  $i=1,\cdots,N$. \\
%Lastly, let us denote by $\{ \mu_t^N \}_{N \in \N}$ the empirical distribution of the particle system \eqref{particelle} defined in \eqref{empirical_m}. 
Let us now begin with estimating the difference $\left|X\iin_s - X^i_s \right |^2$. 
By recalling that $X_t^{i,N}$ solves the first SDE in \eqref{particelle}, from Jensen's inequality and \eqref{ineq gammma} we obtain 
\begin{align*}
    &\sup \limits_{s \in [0,t]}\left|X\iin_s-X^i_s\right|^2 \, \\
    &\leq C(t) \left [ \,\int_0^t \sup \limits_{r \in [0,s]}|X\iin_r-X^i_r|^2 \,dr\,+\, \int_0^t\sup \limits_{r \in [0,s]} \left|\Gamma_M(X\iin_r,\mu^N_r)-\Gamma_M(X^i_r,\mu_r)\right|^2 \,ds \right ]\\
    & \leq C(t) \left [ \int_0^t \sup \limits_{r \in [0,s]}|X\iin_r-X^i_r|^2 \,dr\,+\, \int_0^t \sup \limits_{r \in [0,s]} \cW_2^2(\mu_r^N,\mu_r) \,ds \right ].
\end{align*} 
By now applying Gronwall's inequality we get 
\begin{equation}\label{kuos1}
     \sup \limits_{s \in [0,t]}\left|X\iin_s-X^i_s\right|^2 \,\leq C(t) \, \int_0^t \sup \limits_{r \in [0,s]} \cW_2^2(\mu_r^N,\mu_r) \,ds.  
\end{equation}
Similarly, by recalling that $A_t^{i,N}$ solves the second SDE in \eqref{particelle}, using \eqref{weight beta NMKE}, \eqref{ineq beta} and Jensen's inequality we get
\begin{align*}
     \sup \limits_{s \in [0,t]} \left|A\iin_s-A^i_s\right|^2 \, &\leq C(t)\, \int_0^t \sup \limits_{r \in [0,s]} \left| \Xi_{\e}(X\iin_r,\mu^N_r)-\Xi._{\e}(X^i_r,\mu_r) \right|^2\,|\pa_s\mY_s^{\ka}|^2\,ds, \\
    & \leq C(t) \left [\int_0^t \sup \limits_{r \in [0,s]}|X\iin_r-X^i_r|^2\,dr\,+\, \int_0^t \sup \limits_{r \in [0,s]} \cW_2^2(\mu_r^N,\mu_r) \,ds \right ].
\end{align*}
Hence, from the above and \eqref{kuos1}, we have 
\begin{align}\label{XNAN}
     \sup \limits_{s \in [0,t]} \left[ |X\iin_s-X^i_s|^2 +|A\iin_s-A^i_s|^2 \right]\,\leq C(t) \,\int_0^t \sup \limits_{r \in [0,s]} \cW^2_2(\mu^N_r,\mu_r)\,ds. 
\end{align}
Let us now introduce the empirical measure associated to the i.i.d. particles in \eqref{mckean-vlasov 2}, namely,
\begin{equation*}
    \nu^N_t:=\frac{1}{N}\sum \limits_{i=1}^N \delta_{(X^i_t,A^i_t)}\ ,\qquad N \in \N,\,\, t \in [0,T] \, ,
\end{equation*}
and  recall that, from the definitions of Section \ref{Notation} (specifically from the definition of 2-Wasserstein distance \eqref{2-Wasserstein} with $\cX=C_t$), we have  
\begin{equation}\label{star i dont know}
\Wt(\mu_{\cdot}^{N}, \nu_{\cdot}^{N})=\inf \limits_{\pi \in \Pi_t \left (\mu^N,\nu^N \right )} \,\,\, \int_{C_t \times C_t} |(x,a)-(y,b)|_t^2 \,\,\pi(dxdadydb) \,, 
\end{equation}
where in the above we have used the notation $\mu_{\cdot}^N:= \{\mu^N_t\}_{t \in [0,T]}$, and similarly for $\nu_{\cdot}^N$.
%where $\mu|_{[0,t]}^N:= \{\mu^N_s\}_{s \in [0,t]}$, and similarly for $\nu|_{[0,t]}^N$. 
Hence, from \eqref{XNAN} we have 
\begin{equation}\label{XNAX2}
    \,\, \Wt^2(\mu_{\cdot}^N,\nu_{\cdot}^N)\,\leq \, \frac{1}{N}\sum \limits_{i=1}^N \left( \sup \limits_{s \in [0,t]} |X\iin_s-X^i_s|^2+|A\iin_s-A^i_s|^2\right).
\end{equation}
The above inequality can be seen to hold true by using \eqref{star i dont know} and considering the empirical measure $\pi$ on $\cP_2(C_t)$ defined as
\begin{equation*}
\pi:=\frac{1}{N}\sum \limits_{i=1}^N \delta_{\left((X\iin_s,A\iin_s)_{s \in [0,t]},(Y^i_s,A^i_s)_{s \in [0,t]}\right)}\,\in\, \Pi_t(\mu^N,\nu^N).    
\end{equation*}
From \eqref{XNAN} and \eqref{XNAX2} it then follows 
\begin{equation}\label{W2_bound}
    \,\,\Wt^2(\mu_{\cdot}^N,\nu_{\cdot}^N) \leq C(t) \int_0^t \,\,\sup \limits_{r \in [0,s]} \cW_2^2(\mu_r^N,\mu_r).
\end{equation}
Moreover, by triangular inequality and \eqref{W2_bound} we obtain 
\begin{align*}
     \,\, \Wt^2(\mu_{\cdot}^N,\mu_{\cdot}) & \leq  \left ( \Wt^2(\mu_{\cdot}^N,\nu_{\cdot}^N)+\Wt^2(\nu_{\cdot}^N,\mu_{\cdot}) \right )\\ 
     & \leq C(t) \left [ \int_0^t \,\, \sup \limits_{r \in [0,s]} \cW_2^2(\mu_r^N,\mu_r)\,ds  +\Wt^2(\nu_{\cdot}^N,\mu_{\cdot}) \right ]\\
    &\leq^{\eqref{W2 leq W2t}} C(t) \left [ \int_0^t \,\, \Ws^2(\mu_{\cdot}^N,\mu_{\cdot})\,ds+ \Wt^2(\nu_{\cdot}^N,\mu_{\cdot})\right ].
\end{align*}
Hence,  by Gronwall's inequality we obtain
\begin{equation*}
    \,\,\Wt^2(\mu_{\cdot}^N,\mu_{\cdot}) \leq C(t) \Wt^2(\nu_{\cdot}^N,\mu_{\cdot}),\qquad t \in [0,T].
\end{equation*}
Finally, since the particles-weights $(X^i_t,A^i_t)$, are i.i.d. random variables with same distribution $\mu_t=\cL(X^i_t,A^i_t)$, from the definition of empirical distribution $\nu^N$ and the strong law of large numbers (\cite[Corollary 2.14]{lacker2018mean}) for any given $t \in [0,T]$ we get 
\begin{equation*}
\,\, \Wt^2(\nu_{\cdot}^N,\mu_{\cdot}) \xrightarrow[N \to \infty]{} 0. 
\end{equation*}
This concludes the proof. 
\end{proof}

\begin{proof}[Proof of Lemma \ref{gamma beta}]
 Let us begin with proving \eqref{ineq gammma}. We first note that we can move the derivative inside the integral, so that
\begin{equation*}
    \pa_x^n \Gamma_M(x,\mu)=\itr \chi_M(a)F^{(n+1)}(x-y)\mu(dyda),\qquad \mu \in \cP_2(\T \times \R).
\end{equation*}
From the triangle inequality and the Lipschitz continuity of $\chi_M$ and $F^{(n+1)}$ (we recall that $\chi_M$ is Lipschitz continuous with Lipschitz constant 1) we obtain, for $\mu,\nu \in \cP_2(\T \times \R)$,   
 \begin{align}
      \left | \pa_x^n\Gamma_M(x,\mu) - \pa_x^n\Gamma_M(y,\nu) \right | & \leq \left | \pa_x^n\Gamma_M(x,\mu) -\pa_x^n \Gamma_M(x,\nu) \right |+\left | \pa_x^n\Gamma_M(x,\nu) - \pa_x^n\Gamma_M(y,\nu) \right | \nonumber \\
     & \leq^{\eqref{kantorovich}} |F^{(n+1)}|_{L^{\infty}(\T;\R)}\cW_1(\mu,\nu)\nonumber\\
     &+|F^{(n+1)}|_{L^{\infty}(\T;\R)}|x-y|\itr \chi_M(a)\,\mu(dzda)\nonumber\\
     & \leq^{\eqref{chiM}} |F^{(n+1)}|_{L^{\infty}(\T;\R)}\cW_1(\mu,\nu)+|F^{(n+1)}|_{L^{\infty}(\T;\R)}M|x-y|  \label{eqn:cutoff} \\
     & \leq \max\{M,1\}|F^{(n+1)}|_{L^{\infty}(\T;\R)}(|x-y|+\cW_2(\mu,\nu)).\nonumber 
 \end{align}
 As for inequality \eqref{ineq beta} we proceed similarly. Indeed,
  \begin{align*}
     & \left | \Xi_{\e}(x,\mu) - \Xi_{\e}(y,\nu) \right | \leq \left | \Xi_{\e}(x,\mu) - \Xi_{\e}(x,\nu) \right |+\left | \Xi_{\e}(x,\nu) - \Xi_{\e}(y,\nu) \right |
   \end{align*}
For the first addend, by first noting from \eqref{von mises mollifier} that 
\begin{align}\label{lower bound phieps}
& \Phi_{\e}(x) \geq \frac{1}{2\pi I_0\left(\frac{1}{\e}\right )}e^{-\frac{1}{\e}}=:m_{\e},\qquad\e>0 \\
& \left |\frac{d}{dx} \Phi_{\e}(x) \right | \leq \frac{e^{\frac{1}{\e}}}{2\pi \e I_0\left(\frac{1}{\e}\right )} =:D_{\e},\qquad\e>0 \label{upper bound derivative phieps}
\end{align}
we have
 \begin{align}\notag    
 \left | \Xi_{\e}(x,\mu) - \Xi_{\e}(x,\nu) \right | & = \left | \frac{q(x)\left ( \itr \Phi_{\e}(x-y)\mu(dyda)- \itr \Phi_{\e}(x-y)\nu(dyda)\right )}{\itr \Phi_{\e}(x-y)\mu(dyda)\itr \Phi_{\e}(x-y)\nu(dyda)} \right | \\ \label{stima primo addendo beta}
 &\leq \frac{|q|_{L^{\infty}(\T;\R)}D_{\e}}{m_{\e}^2}\cW_1(\mu,\nu) \,.
 \end{align}
 %    & \leq K_{\e} \left ( \cW_1(\mu,\nu)+|x-y| \right ) \leq C_{\e} \left ( \cW_2(\mu,\nu)+|x-y| \right ), \left| \itr \Phi_{\e}(x-y)\mu(dyda)- \itr \Phi_{\e}(x-y)\nu(dyda) \right| where $C_{\e} \to +\infty$ as $\e \downarrow 0$.
 As for the second addend, from
 \begin{equation}\label{upper bound phieps}
     \left |\Phi_{\e}(x)\right| \leq \frac{e^{\frac{1}{\e}}}{2\pi I_0\left(\frac{1}{\e}\right )}=:M_{\e},\qquad \e>0,
 \end{equation}
 the lower bound \eqref{lower bound phieps} and the upper bound \eqref{upper bound derivative phieps} we have
 \begin{align} \notag
    \left | \Xi_{\e}(x,\nu) - \Xi_{\e}(y,\nu) \right | & = \left | \frac{q(x)\itr \Phi_{\e}(y-z)\mu(dzda)-q(y)\itr \Phi_{\e}(x-z)\mu(dzda)}{\itr \Phi_{\e}(x-z)\mu(dzda)\itr \Phi_{\e}(y-z)\mu(dzda)} \right | \\
    & \leq \frac{|q|_{L^{\infty}(\T;\R)} D_{\e}+ |q^{'}|_{L^{\infty}(\T;\R)} M_{\e}}{m_{\e}^2}|x-y|. \label{stima secondo addendo beta}
 \end{align}
 Combining \eqref{stima primo addendo beta} and \eqref{stima secondo addendo beta} we obtain \eqref{ineq beta}, which concludes the proof.
\end{proof}

\section{The limit $\e \downarrow 0$}\label{section_limit_e}

In this section we prove Proposition \ref{prop limit rho eps}. This section is organised similarly to the previous one, so Proposition \ref{prop limit rho eps}  will be proved as a consequence of Proposition \ref{prop_e_before_M} below. So, after some preliminary definitions, we first state Proposition \ref{prop_e_before_M} then prove Proposition \ref{prop limit rho eps}, and finally prove Proposition \ref{prop_e_before_M} itself.    %To be precise, for any fixed $M,\ka \in \N$ we prove that the law of the stochastic process $t \to \left ( X_t^{\e,M,\ka},A_t^{\e,M,\ka} \right )$ (solution to McKean-Vlasov SDE \eqref{sistemaMKE}) converges as $\e \downarrow 0$ in the 2-Wasserstein distance of $C_T$ to the law of the pair $t \to \left ( X_t^{M,\ka},A_t^{M,\ka} \right )$ (solution to the McKean-Vlasov system \eqref{sistemaMKE}). 
\begin{defn}\label{DefMK}
Let $(X_0,A_0,\mu_0,\zeta_0)$ be as in Definition \ref{def_MV} and Assumption \ref{assunzioni sui dati iniziali}. A strong solution to \eqref{sistemaMK} is a 4-tuple 
$t \to (X_t^{M,\ka},A_t^{M,\ka},\mu_t^{M,\ka},\zeta_t^{M,\ka})\in C \left([0,T] ; \T \!\times \!\R \!\times \!\cP_2(\T \times \R) \!\times \! C^{\infty}(\T;\R) \right )$
satisfying the conditions given below
\begin{itemize}
    \item[(i)] $(X_t^{M,\ka},A_t^{M,\ka})$ is $\{\cF_t\}_{t \in [0,T]}$-adapted;
    \item[(ii)] The stochastic process $(X_t^{M,\ka},A_t^{M,\ka})$ solves \eqref{sistemaMK}, for every $t\in[0,T]$, $\mP$-a.s, where in \eqref{sistemaMK} $\mu_t^{M,\ka}=\cL(X_t^{M,\ka},A_t^{M,\ka})$ and $\zeta_t^{M,\ka}$ is the solution to the PDE \eqref{zetaPDEMK}.  
\end{itemize}
\end{defn}

We recall (see Proposition \ref{prop:law satisfies PDE}) that $\mu_t^{M,\ka}$ is the (unique) weak measure-valued solution to the PDE \eqref{PDEmuMK}.
%and $\zeta_t^{M, \ka}$ are solutions, respectively, of the parabolic PDEs  \eqref{PDEmuMK} and \eqref{zetaPDEMK}, respectively.

\begin{lemma}\label{wellpossistemaMK}
System \eqref{sistemaMK} admits a unique solution in the sense of Definition \ref{DefMK}. Furthermore, for any given $M,\ka \in \N$ the function $\zeta_t^{M, \ka}$ is strictly positive, i.e. $\zeta_t^{M,\ka}(x)>0$ for every $x \in \T$, $t \in [0,T]$ and the pair $(X_t^{M,\ka},A_t^{M,\ka})$ belongs to the space $L^2\left (\Omega; C([0,T];\T \times \R) \right )$.    
\end{lemma}

We refer the reader to Appendix \ref{well-posedness} for the proof of the above lemma. Let us emphasize that the above well-posedness result only gives us the strict positivity of the function $\zeta_t^{M, \ka}$ for each $M, \ka \in \N$ fixed (such a positivity comes as a biproduct of the scheme of proof), but the lower bound obtained through the proof of Lemma \ref{wellpossistemaMK} is not uniform in $M,\ka \in \N$. Bounds which are uniform in $M \in \N$ will be obtained in Lemma \ref{cor zeta}.

\begin{prop}\label{prop_e_before_M}
Let $M,\ka \in \N$ be fixed and let $X_0,A_0,\mu_0$ and $\zeta_0$ be as in Definition \ref{def_MV} and Assumption \ref{assunzioni sui dati iniziali}. 
Let $ \left(X^{\e, M,\ka}_t,\,A^{\e, M,\ka}_t,\,\mu^{\e, M,\ka}_t, \zeta^{\e, M,\ka}_t\right )$ be the unique solution to the McKean-Vlasov equation \eqref{sistemaMKE} with initial condition $\left(X_0,A_0, \mu_0,\zeta_0\right)$. Let $\left(X_t^{M,\ka},\,A_t^{M,\ka},\,\mu_t^{M,\ka},\zeta_t^{M,\ka}\right)$ be the unique solution to the McKean-Vlasov equation \eqref{sistemaMK} with the same initial conditions;  then 
\begin{equation}\label{wasserstein limit eps}
    \lim \limits_{\e \downarrow 0}\,\,\WT(\mu_{\cdot}^{\e, M,\ka},\mu_{\cdot}^{M,\ka})\,=\, 0.
\end{equation}
As a consequence, 
$$
\sup_{t \in [0,T]} \left[\left\vert X_t^{\e, M, \ka} - X_t^{M, \ka}\right\vert^2+ 
\left\vert A_t^{\e, M, \ka} - A_t^{M, \ka}\right\vert^2\right] \stackrel{\e \rightarrow 0}{\longrightarrow} 0, \quad \mathbb P - a.s. 
$$
\end{prop}
As a consequence of the above result Proposition \ref{prop limit rho eps} readily follows.
\begin{proof}[Proof of Proposition \ref{prop limit rho eps}]  This is completely analogous to the proof of Proposition \ref{prop limit rho N}, so we only sketch it. 
From \eqref{wasserstein limit eps} and Proposition \ref{weak formulation on TR},  we obtain 
\begin{equation*}
    \lim \limits_{\e \downarrow 0} \,\, \sup \limits_{t \in [0,T]} \left | \llangle \Psi, \mu_t^{\e, M,\ka} \rrangle - \llangle \Psi, \mu_t^{M, \ka} \rrangle \right |=0,\,\,\,\,
\end{equation*}
for any given $\Psi \in C^{\infty}(\T \times \R;\R)$ which grows at most quadratically in $a \in \R$. In particular, if we set $\Psi(x,a)=af(x)$, $f \in C^{\infty}(\T;\R)$ and recalling \eqref{rhoMKE} and \eqref{rhoMK} we have 
\begin{align*}
    & \llangle af,\mu_t^{\e, M,\ka} \rrangle= \itr af(x)\,\mu_t^{\e, M,\ka} (dxda)=%\mE(A_t^{\e, M,\ka}f(X_t^{\e, M,\ka}))=
    \langle f,\rho_t^{\e, M,\ka} \rangle \, \\
    & \llangle af,\mu_t^{M,\ka} \rrangle= \itr af(x)\,\mu_t^{M,\ka} (dxda)
    %=\mE(A_t^{M,\ka} f(X_t^{M,\ka}))
    =\langle f,\rho_t^{M,\ka}  \rangle. 
\end{align*}
From the above we obtain \eqref{limiteeps}. The rest of the proof is analogous to the proof of Proposition \ref{prop limit rho N}. Then, from the uniqueness of the classical solution of \eqref{PDErhoMKE} and of \eqref{PDErhoMK}, the last statement of the proposition follows.
\end{proof}
The proof of  Proposition \ref{prop_e_before_M} hinges on the preliminary estimates contained in Lemma \ref{stime} below, so we state this lemma first, and then prove it at the end of this section. In reading Lemma \ref{stime} note that equations \eqref{PDEzetaMKE} and \eqref{zetaPDEMK} have the same structure, which is the structure of equation \eqref{36 mu} below. That is, using the notation adopted in \eqref{36 mu}, we can write the solution $\zeta_t^{\ep, M,\ka}$ of \eqref{PDEzetaMKE} as $\zeta_t^{\ep, M,\ka} = \zeta_t^{(\mu^{\ep, M,\ka})}$, while for  $\zeta_t^{M,\ka}$ solution of \eqref{zetaPDEMK} we have $\zeta_t^{M,\ka} = \zeta_t^{(\mu^{M,\ka})}$.  The crucial estimate in this section is the lower bound \eqref{lower bound zeta mu} below which allows, in the proof of Proposition \ref{prop_e_before_M}, to obtain the lower bound \eqref{lower bound zeta eps}, which is uniform in $\e>0$.
\begin{lemma}\label{stime}
    Let $\mu$ be a fixed arbitrary path in $\cP_2(\T\times \R)$, i.e.  $\mu=\{\mu_t\}_{t \in [0,T]} \in C([0,T];\cP_2(\T\times \R))$ and let $\zeta_t^{(\mu)}$ be the solution to the following PDE 
    \begin{equation} \label{36 mu}
    \begin{dcases}
        & \pa_t\zeta_t^{(\mu)}=\pa_{xx}\zeta_t^{(\mu)}+\pa_x\left [ (V^{'}+\Gamma_M(x,\mu_t))\zeta_t^{(\mu)}\right ] \\
        &\zeta^{(\mu)}|_{t=0}=\zeta_0 \in H^2(\T;\R),
    \end{dcases}
    \end{equation}
    where $\Gamma_M$ has been defined in \eqref{def_GammaM}.
    Then  the following estimates hold: 
    \begin{itemize}\label{kuos}
        \item[(i)] For every $x \in \T$ and $t \in [0,T]$, 
        \begin{equation}\label{lower bound zeta mu}
            \zeta_t^{(\mu)}(x) \geq e^{-t \,\,\mathcal{A}(M)}\zeta_0(x) \, , 
        \end{equation} 
        where $\mathcal{A}(M)$ is defined as 
        \begin{equation}\label{esponente stima lower bound}
            \mathcal{A}(M):= |V^{'}  |_{L^{\infty}(\T;\R)}^2+M^2 | F^{'} |_{L^{\infty}(\T;\R)}^2+ |V^{''}|_{L^{\infty}(\T;\R)}+M|F^{''}|_{L^{\infty}(\T;\R)} \,. 
        \end{equation}
        \item [(ii)] For any $t \in [0,T]$, 
        \begin{equation}\label{bound H2-norm zeta mu}
            \left |\zeta_t^{(\mu)} \right |_{H^2(\T;\R)}^2\lesssim_{M,t} \cH(|\zeta_0|_{H^2(\T;\R)},M,t)\, ,
        \end{equation}
        where $\cH: \R_{+} \times \R_{+} \times [0,T] \to \R_+$ is a continuous (in the product norm) locally bounded function such that $\cH \xrightarrow[M \to \infty]{} +\infty$ (for each fixed value of the first and third argument of the function).
        \item[(iii)] For any $t \in [0,T]$ and any $\mu,\nu \in C([0,T];\cP_2(\T\times \R))$, 
        \begin{equation} \label{bound differenxe H1 norm zeta mu zeta nu}
            \left |\zeta_t^{(\mu)}-\zeta_t^{(\nu)} \right |_{H^1(\T;\R)}^2\leq \cG(M,t) \int_0^t \cW_2^2(\mu_s,\nu_s)\,ds \, ,
        \end{equation}
        where $\cG: \R_+ \times [0,T] \to \R_+$ is a continuous function (in the product norm) such that $\cG(M,t) \xrightarrow[M \to +\infty]{}+\infty$, for each fixed $t \in [0,T]$.
    \end{itemize}
\end{lemma}
\begin{note}
\label{note:unifinepsilon}
 Let us observe that the lower bound \eqref{lower bound zeta mu} does not depend on $\mu$. This is due to the fact that in the proof the only property of $\mu$ we use is the fact that it is a probability measure for each $t$ fixed, i.e. $\itr \mu_t(dxda) = 1$. In particular, when we use \eqref{lower bound zeta mu} to lower bound $\zeta_t^{\e,M, \ka}$, this will result in a lower bound which is uniform in $\ep>0$, as $\itr \mu_t^{\ep, M, \ka} (dx da) = 1$ for every $\ep, M, \ka$ and $t$, see \eqref{lower bound zeta eps} and the proof of Proposition \ref{prop_e_before_M}. 
\end{note}
\begin{proof}[Proof of Proposition \ref{prop_e_before_M}]
We proceed in the same fashion as in the proof of Proposition \ref{prop_limitN};  to avoid notation overload, and since we focus on the limit $\e \downarrow 0$, we omit the dependence on $M,\ka$ of the quantities used in the proof so that, in this proof only,  $\left ( X_t^{\e}, A_t^{\e},\mu_t^{\e},\zeta_t^{\e} \right )$ and $\left ( X_t,A_t,\mu_t,\zeta_t\right )$ denote the solution to system \eqref{sistemaMKE} and \eqref{sistemaMK}, respectively.  Moreover, throughout the proof $C$ will be a generic constant, changing from line to line,  depending continuously on the parameters in its argument, and increasing as a function of such parameters (e.g. $C(t, M)$ will be a constant depending continuously on $t, M$ and increasing in both arguments).  With this premise,  we start by  estimating the difference $\left | X_t^{\e} - X_t \right |^2$:   
\begin{align*}
            \mE \sup \limits_{s \in [0,t]}\left | X_{s}^{\e}-X_{s} \right |^2  & \leq 2 \int_0^t \mE \sup \limits_{r \in [0,s]}\left | X_{r}^{\e}-X_{r} \right |^2 \,ds + 2 \int_0^t \mE \sup \limits_{r \in [0,s]} \left | \Gamma_M(X_r^{\e},\mu_r^{\e})-\Gamma_M(X_r,\mu_r)\right |^2\,ds\\
            &\leq^{\eqref{ineq gammma}} C(M) \left [ \int_0^t \mE \sup \limits_{r \in [0,s]} \left | X_{r}^{\e}-X_{r} \right |^2 \,ds  +\int_0^t \sup \limits_{r \in [0,s]} \cW_2^2(\mu_{r}^{\e},\mu_{r}) \,ds \right ] \, , 
\end{align*}
where we emphasize that $C(M)$ is independent of $\e>0$.  
Hence, from Gronwall's lemma we obtain
\begin{align}\label{XKEM estimates}
    & \E\left[\sup \limits_{s \in [0,t]}|X_{s}^{\e}-X_{s}|^2 \right]\,\leq C(t,M) \,\,\int_0^t \sup_{r \in [0,s]}\cW^2_2(\mu^{\e}_{r},\mu_{r})\,ds.
\end{align}
Moreover, using \eqref{lower bound zeta mu} applied to $\zeta_t^{\e}$ and $\zeta_t$, respectively,  (see comments before the statement of Lemma \ref{stime})  we deduce the following lower bounds 
\begin{align}\label{lower bound zeta eps}
    \zeta_t^{\e}(x) \geq \Bar{C}(M,t,\eta),\qquad \mbox{for every } x \in \T,\,\,\,t \in [0,T],\\
    \zeta_t(x) \geq \hat{C}(M,t,\eta),\qquad \mbox{for every } x \in \T,\,\,\,t \in [0,T], \label{lower bound zeta senza eps}
\end{align}
where $\eta=\min \limits_{x \in \T} \zeta_0>0$,  $\bar C, \hat C >0$ for each fixed $M, t, \eta$ and we emphasize that $ \Bar{C}$ does not depend on $\e>0$ i.e. the lower bound for $\zeta_t^{\e}$ holds uniformly in $\e>0$ (see Note \ref{note:unifinepsilon}). 
From this and by recalling that $\itt \Phi_{\e}(x)\,dx=1$ we deduce 
\begin{equation}\label{lower bound phi star zeta}
 \left ( \Phi_{\e}*\zeta_t^{\e} \right )(x) \geq \Bar{C}(M,t,\eta),\qquad\text{for every $x \in \T$ and $t \in [0,T]$,}    
\end{equation}
where the above lower bound holds uniformly in $\e>0$. Hence, 
\begin{align}\notag
    & \left|\frac{q(X_t^{\e})}{\Phi_{\e} * \zeta_t^{\e}(X_t^{\e})}\,-\,\frac{q(X_t)}{\zeta_t(X_t)}\right|^2\,= \, \left | \frac{q(X_t^{\e})\,\zeta_t(X_t)-q(X_t)\,\Phi_{\e} * \zeta_t^{\e}(X_t^{\e})}{\left( \Phi_{\e} * \zeta_t^{\e}(X_t^{\e})\right)\,\zeta_t(X_t)} \right |^2 \\\notag
    & \leq^{\eqref{lower bound zeta senza eps},\eqref{lower bound phi star zeta}} C(t,M,\eta)\left | q(X_t^{\e})\,\zeta_t(X_t)-q(X_t)\,\Phi_{\e} * \zeta_t^{\e}(X_t^{\e}) \right |^2 \\ \notag
    & \leq C(t,M,\eta) \left [ \,\,\left | \left (q(X_t^{\e})-q(X_t) \right )\,\zeta_t(X_t) \right |^2 +\left | q(X_t)\,\left (\zeta_t(X_t)-\Phi_{\e} * \zeta_t^{\e}(X_t^{\e}) \right) \right |^2 \right ] \\ \notag
    & \leq C(t,M,\eta) \big [ \,\, |\zeta_t|_{L^{\infty}(\T;\R)}^2  \left |X_t^{\e}-X_t \right |^2+ \left | \zeta_t(X_t)-\zeta_t(X_t^{\e}) \right |^2\\ \notag
    &+ \left | \zeta_t(X_t^{\e})-\Phi_{\e} * \zeta_{t}(X_t^{\e})\right |^2 + \left | \Phi_{\e} * \zeta_{t}(X_t^{\e})-\Phi_{\e} * \zeta_t^{\e}(X_t^{\e})\right |^2 \big ]\\ \notag
    & \leq C(t,M,\eta)  \big [\,\, |\zeta_t|_{L^{\infty}(\T;\R)}^2  \left |X_t^{\e}-X_t \right |^2+ \left |\pa_x \zeta_t \right |_{L^{\infty}(\T;\R)}^2 \left |X_t-X_t^{\e} \right |^2 \\ \label{ultimate estimate}
    & + \left | \zeta_t-\Phi_{\e} * \zeta_t    \right |_{L^{\infty}(\T;\R)}^2+ \left | \zeta_{t}^{\e}-\zeta_t \right |_{L^{\infty}(\T;\R)}^2 \big ],
\end{align}
where the above chain of inequalities holds $\mP$-a.s. for $t \in [0,T]$ and $C(t,M,\eta)$ is a deterministic constant (depending on $\Bar{C}$ and $\hat{C}$).
Once this is in place, from the Sobolev embedding theorem we know that $H^1(\T;\R)$ compactly embeds into $C(\T;\R)$ i.e. there exists a constant $c>0$ such that 
\begin{equation*}
    |f|_{L^{\infty}(\T;\R)} \leq c |f|_{H^{1}(\T;\R)},\,\text{ for all $f \in H^{1}(\T;\R)$} \,.
\end{equation*}
From the above, by using $(ii)$ and $(iii)$ in Lemma \ref{stime} we obtain estimates for $|\zeta_t|_{L^{\infty}(\T;\R)}^2$, $|\pa_x\zeta_t|_{L^{\infty}(\T;\R)}^2$ and $|\zeta_t^{\e}-\zeta_t|_{L^{\infty}(\T;\R)}^2$, respectively. Indeed, from $(ii)$ in Lemma \ref{stime} applied to $\zeta_t$ we obtain that 
\begin{equation}\label{stima pax zeta senzaeps}
    \sup \limits_{t \in [0,T]} \left [ |\zeta_t|_{L^{\infty}(\T;\R)}^2+|\pa_x \zeta_t|_{L^{\infty}(\T;\R)}^2 \right ] \leq \cH\left (|\zeta_0|_{H^2(\T;\R)},M,t\right )
\end{equation}
while, from $(iii)$ in Lemma \ref{stime} applied to $\zeta_t^{\e}$ and $\zeta_t$ we deduce \begin{equation}\label{stima differenza zeta e zeta senza eps}
    |\zeta_t^{\e}-\zeta_t |_{L^{\infty}(\T;\R)}^2 \leq \cG(M,t)\int_0^t \cW_2^2(\mu_s^{\e},\mu_s)\,ds \,.
\end{equation}
From \eqref{ultimate estimate}, \eqref{stima pax zeta senzaeps} and \eqref{stima differenza zeta e zeta senza eps} we deduce that the following holds for every $t \in [0,T]$, $\mP$-a.s.
\begin{align}\label{KEM difference estimate}
    & \left|\frac{q(X_t^{\e})}{\Phi_{\e} * \zeta_t^{\e}(X_t^{\e})}\,-\,\frac{q(X_t)}{\zeta_t(X_t)}\right|^2 \leq C(t,M,\eta) \left [ \left |X_t^{\e}-X_t \right |^2+\int_0^t \cW_2^2(\mu_s^{\e},\mu_s)\,ds+\left | \zeta_t-\Phi_{\e} * \zeta_t  \right |_{L^{\infty}(\T;\R)}^2 \right ].
\end{align}
From Jensen's inequality we then have 
\begin{align}
     \E\left[ \sup \limits_{s \in [0,t]}|A^{\e}_{s}-A_{s}|^2 \right] & \leq C(t) \,\,\mE \left [ \int_0^t \left|\frac{q(X_s^{\e})}{\Phi_{\e} * \zeta_s^{\e}(X_s^{\e})}\,-\,\frac{q(X_s)}{\zeta_s(X_s)}\right|^2 \left |\pa_sY_{s}^{\ka} \right |^2\,ds \right ]\nonumber \\
    & \leq^{\eqref{KEM difference estimate}} C(t,M,\eta,\ka) \left [ \mE \int_0^t |X_s^{\e}-X_s|^2 \,ds +\int_0^t \cW_2^2(\mu_s^{\e},\mu_s)\,ds+\int_0^t \left | \zeta_t-\Phi_{\e} * \zeta_t  \right |_{L^{\infty}(\T;\R)}^2 \,ds \right ]\nonumber \\
   & \leq^{\eqref{XKEM estimates}} C(t,M,\eta,\ka) \left [ \int_0^t \sup \limits_{r \in [0,s]} \cW^2_2(\mu_{r}^{\e},\mu_{r})ds+ \int_0^t \sup \limits_{r \in [0,s]}\left|\zeta_r - \Phi_{\e} * \zeta_r \right |_{L^{\infty}(\T;\R)}^2\,ds \right ].\label{AKEM estimates}
\end{align}
Hence, from the definition of 2-Wasserstein metric on $C_t$, $\Wt$ (see Section \ref{Notation}),   we get  
\begin{align*}
     \Wt^2(\mu_{\cdot}^{\e},\mu_{\cdot}) & \leq \mE \left ( \sup \limits_{s \in [0,t]}|X^{\e}_{s}-X_{s}|^2+\sup \limits_{s \in [0,t]}|A^{\e}_{s}-A_{s}|^2\right ) \\
    & \leq^{\eqref{XKEM estimates}, \eqref{AKEM estimates}} C(t,M,\eta,\ka)\, \left [\,\,\,\int_0^t \sup \limits_{r \in [0,s]}\cW^2_2(\mu^{\e}_r,\mu_r)\,ds\, + \, \int_0^t \sup \limits_{r \in [0,s]} \left |\zeta_r - \Phi_{\e} * \zeta_r \right |_{L^{\infty} 
    (\T;\R)}^2 \,ds \right ]\,\\
    & \leq^{\eqref{W2 leq W2t}} C(t,M,\eta,\ka) \left [\,\,\,\int_0^t \Ws^2(\mu_{\cdot}^{\e},\mu_{\cdot})\,ds+\int_0^t \sup \limits_{r \in [0,s]} \left |\zeta_r - \Phi_{\e} * \zeta_r \right |_{L^{\infty}
    (\T;\R)}^2 \,ds \right ].
\end{align*}
Therefore, from Gronwall's inequality,
\begin{equation*}
    \Wt^2(\mu_{\cdot}^{\e},\mu_{\cdot}) \leq C(t,M,\eta,\ka) \,\int_0^t \sup \limits_{r \in [0,s]} \left |\zeta_r - \Phi_{\e} * \zeta_r \right |_{L^{\infty}
    (\T;\R)}^2 \,ds
\end{equation*}
By letting $\e \downarrow 0$ and by recalling 
$$\sup \limits_{r \in [0,T]} \left |\zeta_r - \Phi_{\e} * \zeta_r \right |_{L^{\infty}
    (\T;\R)} \xrightarrow[\e \to 0]{}0$$ 
    (see $(i)$ of \cite[Theorem 6 in Appendix C]{evans2022partial}) we can conclude and obtain the desired result thanks to the dominated convergence theorem (which can be applied in view of the bounds we obtained above on the $L^{\infty}$ norm of $\zeta_t$ and recalling that $\int_{\T} \Phi_{\ep} (x) dx = 1$ for every $\ep>0$). 
\end{proof}
\begin{proof}[Proof of Lemma \ref{stime}] We show here how to prove point $i)$, which is the crucial one,  the  proofs of $(ii)$ and $(iii)$ can be found in Appendix \ref{proofs}. For $i)$, 
 differentiating under the integral, we obtain 
\begin{align}
\left |\pa_x^n\Gamma_M(x,\mu_t) \right | \leq^{\eqref{def_GammaM}} M|F^{(n+1)}|_{L^{\infty}(\T;\R)}\itr \mu_t(dyda)=M|F^{(n+1)}|_{L^{\infty}(\T;\R)},\qquad n \in \N,  \label{upper bound bouno}
\end{align}
where we have used the fact that $\mu_t$ is a probability measure for each $t$ fixed. Using the above estimate, the lower bound \eqref{lower bound zeta mu} can be readily deduced from Theorem \ref{lower bound w}.
\end{proof}
\begin{note}\label{note uniform bound in eps and M}
Before moving on to the next section we want to clarify why we have opted for an  upper bound of the type \eqref{upper bound bouno}, which produces the lower bound \eqref{lower bound zeta mu}. First of all, note that at this stage in the proof we are only aiming at a lower bound on $\zeta_t^{\e,M,\ka}$ which is uniform in $\e$, we do not need a lower bound uniform in $M$, as $M$ is fixed. And indeed the lower bound \eqref{lower bound zeta mu} is uniform in $\e$ but not in $M$. One may wonder whether it is possible, with easy tricks, to obtain a lower abound which is uniform in both parameters in one go (so that both limits could be taken together). We have not been able to do so and to explain why this is the case it is useful to point out the following: one may try to upper bound the quantity on the LHS of \eqref{upper bound bouno} differently, for example as follows
\begin{equation*}
    \left | \pa_x^n\Gamma_M(x,\mu_t) \right |\leq \itr |\chi_M(a)F^{(n+1)}(x-y)|\mu_t(dyda)\leq |F^{(n+1)}|_{L^{\infty}(\T;\R)}\itr |a|\mu_t(dyda) \,.
\end{equation*}
By doing so $M$ does not appear explicitly,  but now the quantity on the RHS of the above would have to  be controlled in the proofs with  $\mu_t=\mu_t^{\e, M,\ka}$. That is, we would have to control a term of the form
\begin{equation}\label{momento primo - general}
    \itr |\chi_M(a)F^{(n+1)}(x-y)|\mu_t^{\e, M,\ka}(dyda)\leq |F^{(n+1)}|_{L^{\infty}(\T;\R)}\itr |a|\mu_t^{\e, M,\ka}(dyda) \,, 
\end{equation}
and the RHS of \eqref{momento primo - general} does depend on $M$. 
Because of the (complicated) equation for $\mu_t^{\e, M,\ka}$  we have not been able to control the above uniformly in $\e$ and $M$ simultaneously, so we take the approach of producing estimates which are uniform in ``one parameter at a time", the relevant one for the step at hand. The need to control a term of the form \eqref{momento primo - general} will arise again in the next step, when letting  $M\rightarrow \infty $, see \eqref{super}. 
\end{note}

%%%%%%%%Section 6

\section{The Limit $M\to \infty$}\label{limit M to infty}
In this section we prove Proposition \ref{prop limit rho M}, i.e. the convergence of $\rho^{M,\ka}_t$ to $\rho_t^{\ka}$, as $M \to \infty$.  Analogously to previous sections, to do so,  we show that for any given $\ka \in \N$ the law $\mu_t^{M,\ka}$ of the pair $(X_t^{M,\ka},A_t^{M,\ka})$, solution to \eqref{sistemaMK},  converges to $\mu_t^{\ka}$, law of the pair $(X_t^{\ka},A_t^{\ka})$ in \eqref{sistemaK},  as $M \to \infty$. This is the result of Proposition \ref{thm limit M} which implies Proposition \ref{prop limit rho M}. Therefore, we focus on proving Proposition \ref{thm limit M}.

\begin{defn}\label{defK}
 Let $X_0,A_0,\mu_0$ and $\zeta_0$ be as in Definition \ref{def_MV} and Assumption \ref{assunzioni sui dati iniziali}. The 4-tuple $t \to (X_t^{\ka},A_t^{\ka},\mu_t^{\ka}, \zeta_t^{\ka}) \in C([0,T];\T\times \R \times \cP_2(\T \times \R) \times C^{\infty}(\T;\R))$ is a solution to \eqref{sistemaK} if the following conditions hold:
\begin{itemize}
 \item[(i)] $(X_t^{\ka},A_t^{\ka})$ is $\{\cF_t\}_{t \in [0,T]}$-adapted;
  \item[(ii)] The stochastic process $(X_t^{\ka},A_t^{\ka})$ satisfies \eqref{sistemaK}, for every $t\in[0,T]$, $\mP$-a.s., where $\mu_t^{\ka}=\cL(X_t^{\ka},A_t^{\ka})$  and $\zeta_t^{\ka}$ is the unique classic solution to \eqref{PDEzetaK}. %and $\rho_t^{\ka}$ is the unique weak (and, a posteriori classical) solution to \eqref{PDErhoK}.
\end{itemize}
\end{defn}
\begin{lemma}\label{well-posedness kappa}
    For any given $\ka \in \N$, system \eqref{sistemaK} admits a unique solution in the sense of Definition \ref{defK}. Furthermore, for any given $\ka \in \N$ the function $\zeta_t^{\ka}$ is strictly positive, i.e. $\zeta_t^{\ka}(x)>0$ for every $x \in \T$, $t \in [0,T]$ and the pair $(X_t^{\ka},A_t^{\ka})$ belongs to the space $L^2\left (\Omega; C([0,T];\T \times \R) \right )$.    
%Furthermore, the pair $(X_t^{\ka},A_t^{\ka})$ belongs to $L^2\left (\Omega; C([0,T];\T \times \R) \right )$
\end{lemma}
We refer the reader to Appendix \ref{well-posedness} for the proof of this lemma. As mentioned  in Section \ref{section: main results},  for some of the proofs of this section it will be convenient  to work with the dynamics  \eqref{sistemaK}-\eqref{PDEzetaK}  in its equivalent form \eqref{sistemaK-rho}. Hence we clarify  in Note \ref{equivalenza tra i due sistemi rhok} why such equivalence holds. 
\begin{note}\label{equivalenza tra i due sistemi rhok}
Let us first consider a solution to \eqref{sistemaK} and then define $\rho_t^{\ka}$ as in \eqref{rhoK}; one can see that the measure $\rho_t^{\ka} \in \cM(\T)$ is a weak measure-valued solution to \eqref{PDErhoK}. A posteriori, from the smoothing properties of \eqref{PDErhoK} - see e.g. Lemma \ref{lemma: well-posedness of rhoPDE} - $\rho_t^{\ka}$ turns out to be a classical solution as well i.e. $\rho_t^{\ka} \in C^{\infty}(\T;\R)$.\footnote{Note that \eqref{PDErhoK} is non-linear, while \eqref{PDErhoMKE} and \eqref{PDErhoMK} are linear; so proving that a measure-valued solution to \eqref{PDErhoK} with smooth initial datum is a smooth function is a bit more delicate.} Moreover, \eqref{PDErhoK} is a closed equation in $\rho_t^{\ka}$ (this was not true for \eqref{PDErhoMKE} nor for \eqref{PDErhoMK}). Since, in view of \eqref{gamma-rho}, the PDE \eqref{PDEzetaK} solved by $ \zeta_t^{\ka} \in C^{\infty}(\T;\R)$ can be rewritten as 
\begin{equation}\label{PDEzetaK+rho}
    \begin{dcases}
        & \pa_t\zeta_t^{\ka}=\pa_{xx}\zeta_t^{\ka}+\pa_x \left [ (V^{'}+F^{'}*\rho_t^{\ka})\zeta_t^{\ka}\right ], \\
        & \zeta^{\ka}|_{t=0}=\zeta_0,
    \end{dcases}
\end{equation}
by putting together \eqref{sistemaK}, \eqref{gamma-rho}, \eqref{PDErhoK} and \eqref{PDEzetaK+rho} the formulation in terms of \eqref{sistemaK-rho} readily follows.

To prove the converse: let us then consider system \eqref{sistemaK-rho} and define $\mu_t^{\ka}:=\cL(X_t^{\ka},A_t^{\ka})$ where $(X_t^{\ka},A_t^{\ka})$ is the pair in \eqref{sistemaK-rho}. It is easy to see that $\mu_t^{\ka}\in \cP_2(\T\times\R)$.  If we now set $\tilde {\rho}_t^{\ka}(dx):=\int_{\R} a\mu_t^{\ka}(dxda)$ then $\tilde {\rho}_t^{\ka}(dx) \in \cM(\T)$ is a weak measure-valued solution to the PDE \eqref{PDErhoK} with initial datum $\tilde {\rho}_0(x)=\int_{\R}a\mu_0(x,a)da=\rho_0(x) \in C^{\infty}(\T;\R)$. From the smoothing properties of \eqref{PDErhoK} (see again Lemma \ref{lemma: well-posedness of rhoPDE}) we obtain that $\tilde{\rho}_t^{\ka}$ is a classical solution and, therefore, from the uniqueness it thus coincides with $\rho_t^{\ka}$. 
By now using \eqref{gamma-rho} in the opposite direction,  that is, from right to left, it is easy to go from system \eqref{sistemaK-rho} to system \eqref{sistemaK}.
\end{note}

%, we recall that from the definition of $\zeta_t^{\ka}$ and $\rho_t^{\ka}$ (see the PDE \eqref{PDEzetaK} and \eqref{rhoK}, respectively), system \eqref{sistemaK} can be rewritten entirely in terms of $\zeta_t^{\ka}$ and $\rho_t^{\ka}$ as shown in system \eqref{sistemaK-rho}.
%\begin{equation}\label{sistemaKrho}
%    \begin{dcases}
%        & dX_t^{\ka} = -\left [ V^{'}(X_t^{\ka})+F^{'}*\rho_t^{\ka}(X_t^{\ka}) \right ]dt+\sqrt{2}d\beta_t \\
 %       &  dA_t^{\ka}=\frac{q(X_t^{\ka})}{\zeta_t^{\ka}(X_t^{\ka})}d\mY_t^{\ka} \\
 %       & \pa_t\zeta_t^{\ka}=\pa_{xx}\zeta_t^{\ka}+\pa_x \left [ (V^{'}+F^{'}*\rho_t^{\ka})\zeta_t^{\ka}\right ] \\
  %      & \pa_t\rho_t^{\ka}=\pa_{xx}\rho_t^{\ka}+\pa_x \left [ (V^{'}+F^{'}*\rho_t^{\ka})\rho_t^{\ka}\right ] +q\pa_tY_t^{\ka}\\
   %     & X^{\ka}|_{t=0}=X_0,\,\,A^{\ka}|_{t=0}=A_0,\,\,\zeta^{\ka}|_{t=0}=\zeta_0,\,\,\rho^{\ka}|_{t=0}=\rho_0. 
    %\end{dcases}
%\end{equation}
%This allows us, to some extent, to "remove" an explicit dependence on the joint law $\mu_t^{\ka}$ from system \eqref{sistemaK-rho}. 

%; formulation \eqref{sistemaK-rho} motivates the definition below. 

%\begin{proof}
%The proof is completely analogous to the proof of Lemma \ref{wellpossistemaMK} so we don't repeat it.
%\end{proof}
\begin{prop}\label{thm limit M}
Let $\ka \in \N$ be fixed and let $X_0,A_0,\mu_0$ and $\zeta_0$ be as in Definition \ref{defK}. Let us consider the stochastic processes $ (X_t^{M,\ka},A_t^{M,\ka},\mu_t^{M,\ka},\zeta_t^{M,\ka})$ and $(X_t^{\ka},A_t^{\ka},\mu_t^{\ka},\zeta_t^{\ka})$ solutions to \eqref{sistemaMK} and \eqref{sistemaK}, respectively (with the same initial conditions). Then, for every $\ka \in \N$ fixed, the following limit holds:
\begin{equation*}
    \lim \limits_{M \to \infty}\WT(\mu_{\cdot}^{M,\ka},\mu_{\cdot}^{\ka})=0.
\end{equation*}
As a consequence, 
$$
\sup_{t \in [0,T]} \left[\left\vert X_t^{M, \ka} - X_t^{\ka}\right\vert^2+ 
\left\vert A_t^{M, \ka} - A_t^{\ka}\right\vert^2\right] \stackrel{M \rightarrow \infty}{\longrightarrow} 0, \quad \mathbb P - a.s. 
$$
\end{prop}
Proposition \ref{prop limit rho M} follows as a straightforward corollary of the above result.
\begin{proof}[Proof of Proposition \ref{prop limit rho M}]
The proof can be done in the exact same way of the proof of Proposition \ref{prop limit rho eps} so we don't repeat it.
%From \eqref{wasserstein limit M} it follows that  
%\begin{equation}\label{limite muM}
%    \lim \limits_{\e \downarrow 0} \,\, \sup \limits_{t \in [0,T]} \left | \llangle \Psi, \mu_t^{M,\ka} \rrangle- \llangle \Psi, \mu_t^{\ka} \rrangle \right |=0,\,\,\,\,\mP-a.s.
%\end{equation}
%for any given $\Psi \in C^{\infty}(\T \times \R;\R)$ which grows at most quadratically with respect to the variable $a \in \R$. In particular, if we set $\Psi(x,a)=af(x)$, $x \in \T$, $a \in \R$, $f \in C^{\infty}(\T;\R)$ we have that
%\begin{align*}
    %& \llangle af,\mu_t^{M,\ka} \rrangle= \itr af(x)\,\mu_t^{M,\ka} (dxda)=\mE \left (A_t^{M,\ka}f(X_t^{M,\ka}) \right )=\langle f,\rho_t^{M,\ka} \rangle,\,\,t \in 
    %[0,T],\,\,\,\\
    %& \llangle af, \mu_t^{\ka} \rrangle = \itr af(x)\,\mu_t^{\ka} (dxda) = \mE \left (A_t^{\ka} f(X_t^{\ka}) \right ) = \langle f,\rho_t^{\ka} \rangle ,\,\,t \in 
    %[0,T],\,\,\, 
%\end{align*}
%and obtain
%\begin{equation*}
 %   \lim_{\e \downarrow 0} \,\,\sup \limits_{t \in [0,T]} \left | \langle f , \rho_t^{M,\ka} \rangle - \langle f, \rho_t^{\ka} \rangle \right |=0,\,\,
%\end{equation*}
%for every $f\in C^\infty(\T;\R)$. The PDE solved by $\{\rho_t^{\ka}\}_{t \in [0,T]}$ i.e. \eqref{PDErhoMK} can be obtained in the same way we deduced PDE \eqref{PDEmuMKE} for $\{\rho_t^{\e,M\ka}\}_{t \in [0,T]}$ in the proof of  Proposition \ref{prop limit rho N}. This gives Proposition \ref{prop limit rho M}.
\end{proof}
In this section and the next we will make use of the heat semigroup $\{e^{t\pa_{xx}}\}_{t\geq 0}$ which we recall is the family of linear operators from $L^2(\T;\R)$ to $L^2(\T;\R)$ defined as 
\begin{equation}\label{heat semigroup}
e^{t\paxx}f=G_t^{\text{per}}*f,\qquad f \in L^2(\T;\R),t \geq 0,    
\end{equation}
where $G_t^{\text{per}}$ (see \cite{berglund2022introduction}) denotes the fundamental solution of the heat equation on $\T$.
Furthermore, by letting  $\{e_z\}_{z \in \Z}$ be the basis of $L^2(\T;\R)$ defined as 
\begin{equation}\label{fourier}
e_z(x)=
    \begin{dcases}
       \frac{1}{\sqrt{\pi}}\sin(zx),\qquad z>0,\\
       \frac{1}{\sqrt{2\pi}},\qquad z=0,\\
       \frac{1}{\sqrt{\pi}}\cos(zx),\qquad z<0,
    \end{dcases}
\end{equation}
one can readily see that \eqref{heat semigroup} acts on each element of \eqref{fourier} as follows
\begin{equation}\label{heat}
     e^{t\pa_{xx}}e_z(x)=e^{-t|z|^2}e_z(x),\qquad t \geq 0,\,x \in \T.
\end{equation}

Lastly, we use the notation $\qq^{(\cU)}_t$ to denote the following convolution
\begin{equation}\label{notationfor stochastic convolution}
  \qq_t^{(\cU)}(x) = \int_0^t e^{(t-s)\pa_{xx}} d\cU_s(x),\qquad t \geq 0,\,x \in \T \,, 
\end{equation}
where $\cU_t$ is the external forcing term appearing in \eqref{generic PDE+forcing}. Furthermore, if $\cU_t(x)=q(x)Z_t$ with $q$ as in Assumption \ref{assumption 2.1} and $Z_{\cdot}$ a time-dependent path (we emphasize, dependent on time only, not on space as well), either piece-wise $C^1$ or $\gamma$-H\"older continuous for some $\gamma \in (0,1)$; to denote the convolution \eqref{notationfor stochastic convolution} in this case we use the notation $\qq^{(Z)}_t$ instead (i.e. the spatial coefficient $q(x)$ in the superscript is omitted). For instance, in Lemma \ref{well-posedness PDErhoK} below the term $\qq_t^{(Y^{\ka})}$ is defined as $\qq_t^{(Y^{\ka})}(x)=\int_0^t e^{(t-s)\paxx}q(x)\,dY^{\ka}_s$.  

Let us also point out that in \eqref{notationfor stochastic convolution} the integral $\qq_t^{(\cU)}$ is well-defined as a stochastic It\^o integral when $\cU_t=W_t$ is a trace-class Wiener noise, while, when $\cU_t(x)=q(x)Z_t$ it is defined in the Riemann sense if $Z_t$ is piece-wise $C^1$ and in the Young sense if $Z_t$ is H\"older continuous.

%Lastly, if $Z_{\cdot}$ is a time-dependent path (we emphasize, dependent on time only, not on space as well),  either piece-wise $C^1$ or $\gamma$-H\"older continuous  for some $\gamma \in (0,1)$, we use the notation $\qq^{(Z)}_t$ to denote the following convolution
%\begin{equation}\label{notationfor stochastic convolution}
%\qq^{(Z)}_t(x):= \int_0^t e^{(t-s)\pa_{xx}} q(x) dZ_s,\qquad t \geq 0,\,x \in \T \, ;
%\end{equation}
%similarly, the stochastic convolution (i.e. the convolution against Weiner noise) is denoted
%\begin{equation}\label{stochconv}
%    \qq^{(W)}_t(x):= \int_0^t e^{(t-s)\pa_{xx}} dW(s,x),\qquad t \geq 0,\,x \in \T \, .
%\end{equation}
%In \eqref{notationfor stochastic convolution}  the integral is well-defined in the Riemann sense when $Z$ is piece-wise $C^1$ and in the Young sense if $Z$ is H\"older continuous. 

We now state a series of lemmata, which will be needed to prove Proposition \ref{thm limit M}. We then prove Proposition \ref{thm limit M} and postpone the proof of the lemmata to the end of the section. 

%%%%%%%%%%%%%%%%TORNA 1

\begin{lemma}\label{well-posedness PDErhoK}
Let $\rho_0 \in C^{\infty}(\T;\R)$;  then (for each  $\ka \in \N$) the PDE \eqref{PDErhoK}    admits a unique time-continuous and spatially-smooth solution,  i.e. $\rho_t^{\ka} \in C^{0,\infty}([0,T] \times \T;\R)$.\footnote{The space $C^{0,\infty}([0,T] \times \T;\R)$ denotes the class of real-valued functions continuous w.r.t the variable $t \in [0,T]$ and smooth w.r.t the variable $x \in \T$.}
Furthermore, the following estimate on the $L^2(\T;\R)$-norm of $\rho_t^{\ka}$ and $\pa_x\rho_t^{\ka}$ holds:
\begin{equation}\label{stime rhoka}
        \left |\rho_t^{\ka} \right |_{L^2(\T;\R)}^2+\int_0^t \left |\pa_x\rho_s^{\ka} \right |_{L^2(\T;\R)}^2\,ds \leq \mathcal{J}\left (\left |\rho_0 \right |_{L^2(\T;\R)}^2,\left |\qq_{\cdot}^{(Y^{\ka})} \right |_{C([0,T];H^1(\T;\R))} \right )
\end{equation}
where $\mathcal{J}:\R_+ \times \R_+ \to \R_+$ is a locally bounded continuous function. 
%and $\mY_q^{\ka}$ is defined as 
%\begin{equation}\label{smooth convolution}
%    \mY_q^{\ka}(t,x):=\int_0^t e^{(t-s)\pa_{xx}}q(x)d\mY_s^{\ka},\qquad t\in [0,T],\, x \in \T.
%\end{equation}

\end{lemma}
\begin{lemma}\label{lemma first moment}
Let the stochastic process $(X_t^{M,\ka}, A_t^{M,\ka}, \mu_t^{M,\ka},\zeta_t^{M,\ka})$ be the  solution to \eqref{sistemaMK} with initial condition $X_0,A_0,\mu_0$ and $\zeta_0$ as in Definition \ref{defK}. Then for every $M \in \N$ and $t \in [0,T]$ the following estimate for the first moment (in the variable $a \in \R$) of $\mu_t^{M,\ka}$ holds:
   \begin{equation}\label{super}
       \itr |a|\mu_t^{M,\ka}(dxda) \leq\itr |a|\,\mu_0(dxda)+\int_0^t \, ds|\pa_s\mY_s^{\ka}| \int_{\T}|q(x)|\,dx.
   \end{equation}
\end{lemma}
Lemma \ref{cor zeta} below (which hinges on Lemma \ref{lemma first moment}) provides lower and upper bounds on $\zeta_t^{M,\ka}$ which are, crucially, independent of $M \in \N$. Next, Lemma \ref{stime 2} states lower and upper bounds on $\zeta_t^{\ka}$. Moreover in $(iv)$ and $(v)$ of Lemma \ref{stime 2} we prove that \eqref{ineq gammma} and \eqref{bound differenxe H1 norm zeta mu zeta nu}, when applied with $\mu_{\cdot}=\mu_{\cdot}^{M,\ka}$ and $\nu_{\cdot}=\mu_{\cdot}^{\ka}$, hold with constants which are independent of $M$. In other words, \eqref{ineq gammma 2} and \eqref{zetamuk meno zetak} are specific instances of \eqref{ineq gammma} and \eqref{bound differenxe H1 norm zeta mu zeta nu}, the difference being that here we need to show that the constants on the RHS of \eqref{ineq gammma 2} and \eqref{zetamuk meno zetak} do not depend on $M$.

\begin{lemma}\label{cor zeta}
    Let $(X_t^{M,\ka}, A_t^{M,\ka}, \mu_t^{M,\ka},\zeta_t^{M,\ka})$ be the solution to \eqref{sistemaMK} with initial condition $(X_0,A_0,\mu_0,\zeta_0)$ as in Definition \ref{defK}. Then $\zeta_t^{M,\ka}$ is point-wise bounded from below, uniformly in $M \in \N$ and, moreover, the $H^2(\T;\R)$-norm of $\zeta_t^{M,\ka}$ is bounded from above, uniformly in $M \in \N$. That is, 
    \begin{align}\label{first lower}
            & \zeta_t^{M,\ka}(x) \geq C_1, \\ \label{second upper}
            & |\zeta_t^{M,\ka}|_{H^2(\T;\R)}^2 \leq C_2, 
    \end{align}
where the constants $C_1,C_2>0$ are independent of $M \in \N$ and $C_1$ is independent of $x \in \T$ (but both depend continuously on $t$).
\end{lemma}
\begin{lemma}\label{stime 2}
Let $T>0$ be fixed and let $(X_t^{M,\ka},A_t^{M,\ka},\mu_t^{M,\ka},\zeta_t^{M,\ka})$ and $(X_t^{\ka},A_t^{\ka},\mu_t^{\ka}, \zeta_t^{\ka})$ be the solutions (up to time $T$) to \eqref{sistemaMK} and \eqref{sistemaK}, respectively, with the same initial conditions $(X_0,A_0)$, $\mu_0$ and $\zeta_0$ as in Definition \ref{defK}. Let us also consider the weighted distribution $\rho_t^{\ka}$ of $\mu_t^{\ka}$ in \eqref{rhoK}, whose density $\rho_t^{\ka}$ solves the PDE \eqref{PDErhoK}.
Then the following estimates hold:
    \begin{itemize} \label{kuos kuos}
      \item[(i)]  For every $x \in \T$ and $t \in [0,T]$ the following point-wise lower bound for $\zeta_t^{\ka}$ holds
      \begin{equation*} \label{zetakappa}
          \zeta_t^{\ka}(x) \geq e^{-t\cV(\ka,\rho^{\ka})}\zeta_0(x),
      \end{equation*}
      where $\cV$ is defined as
      \begin{align*}
          \cV(\ka,\rho^{\ka}) & := |V^{'}|_{L^{\infty}}^2 +|F^{'}|_{L^{\infty}}^2\sup \limits_{r \in [0,T]}|\rho_r^{\ka}|_{L^2(\T;\R)}^2+|V^{''}|_{L^{\infty}} +|F^{''}|_{L^{\infty}}\sup \limits_{r \in [0,T]}|\rho_r^{\ka}|_{L^2(\T;\R)}.
      \end{align*}
        \item [(ii)] The following estimate for the $H^2(\T;\R)$-norm of $\zeta_t^{\ka}$ holds for every $t \in [0,T]$:
        \begin{equation*}
            |\zeta_t^{\ka}|_{H^2(\T;\R)}^2\leq \cN\left (|\zeta_0|_{H^2(\T;\R)},\sup \limits_{s \in [0,T] }|\rho_s^{\ka}|_{L^2(\T;\R)},t \right ),
        \end{equation*}
        where $\cN:\R_+ \times \R_+ \times [0,T] \to \R_+$ is a continuous (in the product norm) function.
        \item [(iii)] 
        There exists a constant $\bar{K}=\bar{K}(n)>0$ (dependent on $\ka \in \N$) such that for every $n \in \N$ 
        \begin{equation}\label{moments}
            \sup \limits_{t \in [0,T]} \,\,\,\itr |a|^n \mu_t^{\ka}(dxda) \leq K(n).
        \end{equation}
        \item [(iv)] For any $x,y \in \T$, $t \in [0,T]$ and $M \in \N$ we have
\begin{align}\label{ineq gammma 2}
        & \left|\pa_x^n\Gamma_M(x,\mu_t^{M,\ka})-\pa_x^n\Gamma_M(y,\mu_t^{\ka})\right|^2\,\leq C(n)(\,|x-y|^2\,+\, \cW^2_2(\mu_t^{M,\ka},\,\mu_t^{\ka})),
\end{align}
where $C(n):=\max \left \{ \bar K(1),1 \right \}|F^{(n+1)}|_{L^{\infty}(\T;\R)}$ with $\bar K(1)$ the constant in \eqref{moments} for $n=1$ so that $C(n)$ is independent of $M \in \N$.
        \item [(v)] There exists a constant $\hat{K}=\hat{K}(t,\ka)>0$ independent of $M \in \N$ such that the inequality below holds for every $t \in [0,T]$:
         \begin{equation}\label{zetamuk meno zetak}
             |\zeta_t^{\ka}-\zeta_t^{M,\ka}|_{H^1(\T;\R)}^2 \leq \hat{K}(t,\ka) \left [ \int_0^t \cW_2^2(\mu_s^{M,\ka},\mu_s^{\ka})\,ds + \sup \limits_{s \in [0,T]}\,\int_{\T \times \{|a|>M\}} |a|^2\mu_s^{\ka}(dxda) \right ].
         \end{equation}
    \end{itemize}
\end{lemma}
\begin{proof}[Proof of Proposition \ref{thm limit M}]
We proceed in the same fashion of the proof of Proposition \ref{prop_e_before_M}. So some steps are explained more quickly here and we give more details on the part of computation that differs from the proof of Proposition \ref{prop_e_before_M}. In what follows we denote by $C(t)$ a constant depending continuously on $t$ which may change from line to line. Estimating $\mE \,|X_t^{M,\ka}-X_t^{\ka}|^2$ gives us 
%we make use of the representation of system \eqref{sistemaK-rho} in terms of system \eqref{sistemaK} and, moreover,
\begin{align*}
\mE \sup \limits_{s \in [0,t]} \left | X_{s}^{M,\ka}-X_{s}^{\ka} \right |^2  &\leq C(t) \left [\int_0^t \mE \sup \limits_{r \in [0,s]} \left | X_{r}^{M,\ka}-X_{r}^{\ka} \right |^2 \,ds \right. \\
& \left. +\int_0^t \mE \sup \limits_{r \in [0,s]}\left |\Gamma_M(X_r^{M,\ka},\mu_r^{M,\ka})-\Gamma(X_r^{\ka},\mu_r^{\ka}) \right |^2 \,ds \right ],
\end{align*}
so from Gronwall's lemma we deduce
\begin{align*}
\mE \sup \limits_{s \in [0,t]}\left | X_{s}^{M,\ka}-X_{s}^{\ka} \right |^2  \leq C(t) \int_0^t \mE \sup \limits_{r \in [0,s]}\left |\Gamma_M(X_r^{M,\ka},\mu_r^{M,\ka})-\Gamma(X_r^{\ka},\mu_r^{\ka}) \right |^2 \,ds.
\end{align*}
Let us turn our attention to the RHS of the above inequality:
\begin{align}\notag
|\Gamma_M(X_r^{M,\ka},\mu_r^{M,\ka})-\Gamma(X_r^{\ka},\mu_r^{\ka})|^2  & \leq C \left [ |\Gamma_M(X_r^{M,\ka},\mu_r^{M,\ka})-\Gamma_M(X_r^{M,\ka},\mu_r^{\ka})|^2 \right. \\ \notag 
& \left. +|\Gamma_M(X_r^{M,\ka},\mu_r^{\ka})-\Gamma_M(X_r^{\ka},\mu_r^{\ka})|^2  \right. \\ \notag
& \left. +|\Gamma_M(X_r^{\ka},\mu_r^{\ka})-\Gamma(X_r^{\ka},\mu_r^{\ka})|^2 \right ] \\ \notag 
& \leq^{\eqref{ineq gammma 2}} C \left [ \cW_2^2(\mu_r^{M,\ka},\mu_r^{\ka})+|X_r^{M,\ka}-X_r^{\ka}|^2 \right. \\ \notag
& \left. +|\Gamma_M(X_r^{\ka},\mu_r^{\ka})-\Gamma(X_r^{\ka},\mu_r^{\ka})|^2 \right ]\\ \label{Jensen's inequality}
& \leq C \left [\cW_2^2(\mu_r^{M,\ka},\mu_r^{\ka})+|X_r^{M,\ka}-X_r^{\ka}|^2 \right. \\ \notag
&\left. +\int_{\T \times \{|a|>M\}}|a|^2\mu_r^{\ka}(dxda) \right ],\qquad \mP-a.s.
\end{align}
To obtain the third addend in \eqref{Jensen's inequality} we have applied Jensen's inequality to the term  $|\Gamma_M(X_r^{\ka},\mu_r^{\ka})-\Gamma(X_r^{\ka},\mu_r^{\ka})|^2$ combined with the fact that the normalization constant $\cZ_t^{M,\ka}$ of the measure $\1_{\{\T \times \{|a|>M\}}\mu_t^{\ka}$ is bounded from above by 1 uniformly in $M,\ka \in \N$ and $t>0$. Namely, if we let $\cZ_t^{M,\ka}:=\int_{\T \times \{|a|>M\}} \mu_t^{\ka}(dxda)$ and notice that $\cZ_t^{M,\ka}\leq 1$, uniformly in $M,\ka \in \N$ and $t>0$ then we obtain 
\begin{align*}
    \left |\Gamma_M(X_r^{\ka},\mu_r^{\ka})-\Gamma(X_r^{\ka},\mu_r^{\ka}) \right |^2 & =\left(\cZ_t^{M,\ka} \right )^2 \left ( \,\,\, \int_{\T \times \{|a|>M\}} aF^{'}(X_t^{\ka}-y)\frac{\mu_t^{\ka}(dyda)}{\cZ_t^{M,\ka}}\right )^2 \\
    &\leq \left(\cZ_t^{M,\ka} \right )^2 \int_{\T \times \{|a|>M\}} |a|^2 \left |F^{'}(X_t^{\ka}-y) \right |^2 \frac{\mu_t^{\ka}(dyda)}{\cZ_t^{M,\ka}}\\
    & \leq |F^{'}|_{L^{\infty}(\T;\R)}^2\int_{\T \times \{|a|>M\}} |a|^2 \mu_t^{\ka}(dyda).
\end{align*}
This gives the third addend in \eqref{Jensen's inequality}.
Hence, we get
\begin{align*}
    \mE \sup \limits_{s \in [0,t]}\left | X_{s}^{M,\ka}-X_{s}^{\ka} \right |^2  & \leq C(t) \Bigg [ \int_0^t \mE \sup \limits_{r \in [0,s]} \left | X_{r}^{M,\ka}-X_{r}^{\ka} \right |^2 \,ds  +\int_0^t \sup \limits_{r \in [0,s]}W_2^2(\mu_s^{M,\ka},\mu_s^{\ka}) \,ds \\
    & + \sup \limits_{s \in [0,t]}\,\,\,\int_{\T \times \{|a|>M\}}|a|^2\mu_s^{\ka}(dxda) \Bigg ],
\end{align*}
and again from Gronwall's lemma we conclude 
\begin{equation}\label{XMkappa - Xkappa}
    \mE \sup \limits_{s \in [0,t]}\left | X_{s}^{M,\ka}-X_{s}^{\ka} \right |^2  \leq C(t) \left [ \int_0^t \sup \limits_{r \in [0,s]}W_2^2(\mu_s^{M,\ka},\mu_s^{\ka}) \,ds + \sup \limits_{s \in [0,t]}\,\,\,\int_{\T \times \{|a|>M\}}|a|^2\mu_s^{\ka}(dxda) \right ].
\end{equation}
As for the difference $\mE \left | A_t^{M,\ka}-A_t^{\ka} \right |^2$ we can proceed similarly to what we did in \eqref{KEM difference estimate} (we recall that $\eta=\min \limits_{x \in \T}\zeta_0$ and $C(t,\ka,\eta)$ is a constant, depending continuously on $t,\ka$ and $\eta$, which may change from line to line):
\begin{align*}
    & \left|\frac{q(X_t^{M,\ka})}{\zeta_t^{M,\ka}(X_t^{M,\ka})}\,-\,\frac{q(X_t^{\ka})}{\zeta_t^{\ka}
    (X_t)}\right|^2\,=\,\left | \frac{q(X_t^{M,\ka})\,\zeta_t^{\ka}(X_t^{\ka})-q(X_t^{\ka}) \,\zeta_t^{M,\ka}(X_t^{M,\ka})}{\zeta_t^{M,\ka}(X_t^{M,\ka})\,\zeta_t^{\ka}(X_t^{\ka})} \right |^2 \\
    %& \leq^{\eqref{first lower},\eqref{zetakappa}}C(t,\ka,\eta) \left | q(X_t^{M,\ka})\,\zeta_t^{\ka}(X_t^{\ka})-q(X_t^{\ka})\,\zeta_t^{M,\ka}(X_t^{M,\ka}) \right |^2 \\
    %& \leq C(t,\ka,\eta)\left [ \left | \left (q(X_t^{M,\ka})-q(X_t^{\ka}) \right )\,\zeta_t^{\ka}(X_t^{\ka}) \right |^2 +\left | q(X_t^{\ka})\,\left (\zeta_t^{\ka}(X_t^{\ka})-\zeta_t^{M,\ka}(X_t^{M,\ka}) \right) \right |^2 \right ]\\
    %& \leq C(t,\eta,\ka) \left [ |\zeta_t^{\ka}|_{L^{\infty}(\T;\R)}^2  \left |X_t^{M,\ka}-X_t^{\ka} \right |^2+ \left | \zeta_t^{\ka}(X_t^{\ka})-\zeta_t^{\ka}(X_t^{M,\ka}) \right |^2+ \left | \zeta_t^{\ka}(X_t^{M,\ka})-\zeta_{t}^{M,\ka}(X_t^{M,\ka})\right |^2 \right ] \\
    & \leq C(t,\ka,\eta) \left [ |\zeta_t^{\ka}|_{L^{\infty}(\T;\R)}^2  \left | X_t^{M,\ka}-X_t^{\ka} \right |^2+ \left |\pa_x \zeta_t^{\ka} \right |_{L^{\infty}(\T;\R)}^2 \left |X_t^{\ka}-X_t^{M,\ka} \right |^2 + \left |\zeta_{t}^{M,\ka}-\zeta_t^{\ka} \right |_{L^{\infty}(\T;\R)}^2 \right ],
\end{align*}
By means of $(ii)$ combined with estimate \eqref{stime rhoka} for $\rho_t^{\ka}$ and $(v)$ in Lemma \ref{stime 2} we then deduce
\begin{align*}
    \left|\frac{q(X_t^{M,\ka})}{\zeta_t^{M,\ka}(X_t^{M,\ka})}\,-\,\frac{q(X_t^{\ka})}{\zeta_t^{\ka}(X_t^{\ka})}\right|^2 & \leq  C(t,\ka,\eta) \left [ \left |X_t^{M,\ka}-X_t^{\ka} \right |^2+ \int_0^t \cW_2^2(\mu_s^{M,\ka},\mu_s^{\ka})\,ds \right. \\
    & \left. +\sup \limits_{s \in [0,t]} \,\,\, \int_{\T \times \{|a|>M\}}|a|^2\mu_s^{\ka}(dxda)\,ds \right ].
\end{align*}
Hence, similarly to \eqref{XNAN}, by using the above along with \eqref{XMkappa - Xkappa} we obtain
\begin{align}\label{AMkappa-Akappa}
   %  \E\left[\sup \limits_{s \in [0,t]}|X_{s}^{M,\ka}-X_{s}^{\ka}|^2 \right]\,& \leq C(t,\ka,\eta) \Bigg[\,\,\int_0^t \sup \limits_{r \in [0,s]} \cW^2_2(\mu^{M,\ka}_{r},\mu_{r}^{\ka})\,ds\\
%    &  + \sup \limits_{s \in [0,t]}\,\,\,\int_{\T \times \{|a|>M\}}|a|^2\mu_s(dxda) \Bigg ],\\
    \E\left[ \sup \limits_{s \in [0,t]}|A^{M,\ka}_{s}-A_{s}^{\ka}|^2 \right]\,& \leq C(t,\ka,\eta) \left [\,\int_0^t \sup \limits_{r \in [0,s]} \cW^2_2(\mu_{r}^{M,\ka},\mu_{r}^{\ka})|\pa_s\mY^{\ka}|^2\,ds \right. \\
    & \left. + \sup \limits_{s \in [0,t]}\,\,\,\int_{\T \times \{|a|>M\}}|a|^2\mu_s^{\ka}(dxda)\sup \limits_{t \in [0,T]}|\pa_t\mY^{\ka}|^2 \right ].\notag \,\,
\end{align}
From \eqref{XMkappa - Xkappa}, \eqref{AMkappa-Akappa}, the definition \eqref{2-Wasserstein} of 2-Wasserstein distance $\Wt$ with $\cX=C_t$ and \eqref{W2 leq W2t} we deduce
\begin{align*}
    \Wt^2(\mu^{M,\ka},\mu^{\ka}) \leq C(t,\ka,\eta) \left [ \int_0^t \Ws^2(\mu^{M,\ka},\mu^{\ka})\,ds+ \sup \limits_{s \in [0,t]}\,\,\,\int_{\T \times \{|a|>M\}}|a|^2\mu_s^{\ka}(dxda) \right ],
\end{align*}
and, from Gronwall's inequality, 
\begin{equation*}
    \Wt^2(\mu^{M,\ka},\mu^{\ka}) \leq C(t,\ka,\eta) \sup \limits_{s \in [0,t]}\,\,\,\int_{\T \times \{|a|>M\}}|a|^2\mu_s^{\ka}(dxda).
\end{equation*}
By letting $M \to \infty$ and using $(iii)$ in Lemma \ref{stime 2} we obtain desired result.    
\end{proof}

\begin{proof}[Proof of Lemma \ref{well-posedness PDErhoK}]
The well-posedness in mild sense of \eqref{PDErhoK} follows from Lemma \ref{lemma: well-posedness of rhoPDE}. It remains to show that $\rho_t^{\ka}$ is actually a smooth function w.r.t. the spatial variable $x \in \T$. To achieve this, we note that the coefficients in the PDE \eqref{PDErhoK} are smooth, so that the improved regularity in the spatial variable $x \in \T$ can be done by following the steps of \cite[Theorem 4.2]{chazelle2017well} so we don't repeat them. We just highlight that the proof of \cite[Theorem 4.2]{chazelle2017well}  goes through provided one can control the  $L^1(\T;\R)$-norm of the solution itself. In the setting of \cite{chazelle2017well}, due to the gradient structure of PDE at hand, the $L^1(\T;\R)$-norm estimate on the solution was obtained by using a non-negative initial datum and by the positivity-preserving property of the evolution; this allowed the authors of \cite{chazelle2017well} to conclude that $|\rho_t|_{L^1(\T;\R)}=|\rho_0|_{L^1(\T;\R)}$ for $t \geq 0$ (see \cite[Corollary 2.2]{chazelle2017well}). In our case, since our PDE is perturbed by the additive term $q(x)\pa_tY_t^{\ka}$, there is no such thing as mass preservation and preservation of positivity. So, in order to use the scheme of proof of \cite[Theorem 4.2]{chazelle2017well},    we need an estimate on the $L^1(\T;\R)$-norm of $\rho_t^{\ka}$. This can be obtained following the steps of \cite[Proposition B.1]{angeli2023well}, so we don't repeat the whole argument and just conclude that  $|\rho_t^{\ka}|_{L^1(\T;\R)}$ satisfies,  for $t \geq 0$, the following estimate:  
\begin{equation*}
 |\rho_t^{\ka}|_{L^1(\T;\R)} \leq |\rho_0|_{L^1(\T;\R)}+|q|_{L^1(\T;\R)}\int_0^t |\pa_sY_s^{\ka}|\,ds   \,.
\end{equation*}

The estimate \eqref{stime rhoka} can be obtained by following the steps of the proof of \cite[Proposition B.1]{angeli2023well}. 
\end{proof}
%Before proving the limit $M\to \infty$ we need the following lemma which allows us to bound from below the function $\{\zeta_t^{M,\ka}\}_{t \in [0,T]}$ uniformly in $M \in \N$.

\begin{note}\label{note importante}
We recall that %that $\zeta_t^{M,\ka}$ is defined as the (unique) classic solution to the PDE \eqref{zetaPDEMK}. 
from Proposition \ref{prop:law satisfies PDE} we know that the pair $(\mu_t^{M,\ka},\zeta_t^{M,\ka})$ is the unique solution to the system of PDEs \eqref{PDEmuMK}-\eqref{zetaPDEMK}. In particular, we know that $\mu_t^{M,\ka}$ is a weak measure-valued solution to \eqref{PDEmuMK} while $\zeta_t^{M,\ka}$, due to the smoothing properties of \eqref{zetaPDEMK} (see e.g. \cite{friedman2008partial} and \cite[Chapter 7]{evans2022partial}), is the unique classical solution to \eqref{zetaPDEMK}. 

If we define the measure $\tilde{\zeta}_t^{M,\ka}(dx):=\int_{\R} \mu_t^{M,\ka}(dxda)$, a straightforward calculation shows that such a measure is a weak measure-valued solution to \eqref{zetaPDEMK} with initial datum $\zeta_0=\int_{\R} \mu_0(x,a)da$. From \ref{def:zeta0} in Assumption \ref{assunzioni sui dati iniziali} we know that $\int_{\R}\mu_0(x,a)da=\zeta_0(x) \in C^{\infty}(\T;\R)$. Hence, again from the smoothing properties of \eqref{zetaPDEMK} one can see that $\tilde{\zeta}_t^{M,\ka}$ is actually a classical solution to \eqref{zetaPDEMK} and it thus coincides with $\zeta_t^{M,\ka}$. 

As a result of the above, we have obtained that $\zeta_t^{M,\ka}$ can be viewed as a the $\T$-marginal of $\mu_t^{M,\ka}$. Hence, $\zeta_t^{M,\ka}=\cL(X_t^{M,\ka})$ and, therefore,
\begin{align}\label{booooooh}
    \left \langle f, \zeta_t^{M,\ka} \right \rangle=\mE\left (f\left (X_t^{M,\ka}    \right )\right),\qquad\text{for any $f \in C^{\infty}(\T;\R)$}.
\end{align}
%Indeed, if we defined $\langle f, \tilde{\zeta}_t^{M,\ka} \rangle:=\mE(f(X_t^{M,\ka}))$ then from It\^o's formula we would obtain that $\Tilde{\zeta}_t^{M,\ka}$ is a weak (which turns out to be a classical one) solution to PDE \eqref{PDEzetaMKE}. 
%As a consequence of the uniqueness of solution to \eqref{PDEzetaMKE} we obtain $\zeta_t^{M,\ka}=\Tilde{\zeta}_t^{M,\ka}$, $t \in [0,T]$. Hence, a fortiori \eqref{booooooh} follows.
The expression \eqref{booooooh} can of course be written in terms of the measure $\mu_t^{M,\ka}$, %Indeed,  
%\begin{align*}
%            \left \langle f, \zeta_t^{M,\ka} \right \rangle=\mE \left (f \left (X_t^{M,\ka}\right )\right)=\mE \left(\1_{\R}\left(A_t^{M,\ka}\right )f   \left (X_t^{M,\ka}\right)\right)=\llangle \1_{\R}f,    \mu_t^{M,\ka} \rrangle, 
%\end{align*}
that is, 
\begin{equation*}%\label{importantissimo}
    \int_{\T}f(x)\zeta_t^{M,\ka}(x)\,dx=\itr f(x)\mu_t^{M,\ka}(dxda),\qquad \text{for any $f \in C^{\infty}(\T;\R)$}.
\end{equation*}
In particular, since $\zeta_t^{M,\ka} \in C^{\infty}(\T;\R)$ is a deterministic trajectory and $\zeta_t^{M,\ka}>0$, for every $x \in \T$ and $t \in [0,T]$ (from Lemma \ref{wellpossistemaMK}), this implies that $\frac{1}{\zeta_t^{M,\ka}}$ can be taken as a test function in the above equality. Hence, for any given $f \in C^{\infty}(\T;\R)$ we have  
\begin{equation}\label{fondamentale}
    \int_{\T}f(x)\,dx=\itr f(x)\frac{1}{\zeta_t^{M,\ka}(x)}\,\mu_t^{M,\ka}(dxda).
\end{equation}
The equality \eqref{fondamentale} is crucial to the proof of Lemma \ref{lemma first moment} and to our overall strategy.

Had we not fixed the noise but considered the particle system \eqref{particelle} with a Brownian motion $w_t$ in place of $Y_t^{\ka}$, the joint distribution, say $\mu_t^M$, which would appear at this point in the proof, would be stochastic, hence its marginal $\zeta_t^M$ would be stochastic as well, and adapted to the filtration $\Bar{\cF}_t$. Therefore, in this context, equality \eqref{fondamentale} would hold with $\mu_t^M$ and $\zeta_t^M$ in place of $\mu_t^{M,\ka}$ and $\zeta_t^{M,\ka}$ for every realisation $\bar \omega$ belonging to the set defined as $\bar{\Omega}_{\zeta^M}:=\left \{ \bar{\omega} \in \Omega\,|\, \zeta_t^{M}(\bar{\omega})>0,\,\text{for all $t \in [0,T]$ and $x \in \T$}\right \} \in \Bar{\cF}_T$. This produces measurability issues at various stages of the proof. This is another step where working deterministically helps.
\end{note}

\begin{proof}[Proof of Lemma \ref{lemma first moment}]
    %A straightforward application of It\^o's formula shows that $\mu_t^{M,\ka}$ is a weak solution (in the PDE sense) to the following PDE:
    %\begin{equation}\label{118}
    %    \begin{dcases}
     %&\pa_t\mu_t^{M,\ka}=\pa_{xx}\mu_t^{M,\ka}+\pa_x[(V^{'}+\Gamma_M(x,\mu_t^{M,\ka}))\mu_t^{M,\ka}]-\frac{q(x)}{\zeta_t^{M,\ka}(x)}\pa_a \mu_t^{M,\ka} \pa_t\mY^{\kappa}, \\
     %& \mu^{M,\ka}|_{t=0}=\mu_0.
     %   \end{dcases}
    %\end{equation}
    By recalling that $\mu_t^{M,\ka}$ is the weak solution to the PDE \eqref{PDEmuMK} we obtain that for any $\Psi \in C_{c}^{\infty}(\T \times \R;\R)$ and $t \in [0,T]$ the following holds:
    \begin{align}\label{weak sol mu}
    \llangle \Psi,\mu_t^{M,\ka} \rrangle & = \llangle \Psi,\mu_0 \rrangle + \int_0^t \llangle \pa_{xx}\Psi+(V^{'}+\Gamma_M(x,\mu_s^{M,\ka})\pa_x\Psi,\mu_s^{M,\ka} \rrangle\,ds\\ \label{weak solu mu 2}
        &+\int_0^t \llangle \frac{q}{\zeta_s^{M,\ka}}\pa_a\Psi,\mu_s^{M,\ka} \rrangle d\mY_s^{\ka}.
    \end{align}
    Moreover, since we know from Lemma \ref{wellpossistemaMK} that $t \to \mu_t^{M,\ka}\in C([0,T];\cP_2(\T \times \R))$ then it is easy to deduce that the above equality can be extended 
    to any $\Psi \in C^{\infty}(\T \times \R;\R)$ which grows at most quadratically in $a \in \R$ along with its derivatives. To obtain \eqref{super}, it suffices to take as test function $\Psi(x,a)=|a|$. Unfortunately, such test function is not differentiable in $a \in \R$. This minor issue can be overcome by considering $\Psi_{\delta}(x,a)=\sqrt{a^2+\delta}$,\,\,\, $x \in \T,\,\,a \in \R$ and $\delta>0$, doing all the calculations with $\Psi_{\delta}$ in place of $\Psi$ and then set $\delta \downarrow 0$. For the sake of clarity, in what follows, we just work with $\Psi(x,a)=|a|$ (knowing that all the calculations should be done for $\Psi_{\delta}$ and we don't show explicitly the step $\delta \downarrow 0$). From \eqref{weak sol mu} and \eqref{weak solu mu 2} with $\Psi(x,a)=|a|$ we get  
    \begin{align*}
     & \int_{\T \times \R} |a|\mu_t^{M,\ka}(dxda)= \itr |a|\,\mu_0(dxda)+\int_0^t \,d\mY_s^{\ka} \itr \frac{a}{|a|}\frac{q(x)}{\zeta_s^{M,\ka}(x)}\mu_s^{M,\ka}(dxda).
    \end{align*}
    However, since $dY_s^{\ka}=\pa_sY_s^{\ka}ds$ ($Y_t^{\ka}$ is a piece-wise $C^1$ in time $t$) and by using the fact that $\mu_t^{M,\ka}$ is a probability measure we obtain \begin{align*}
        \itr |a|\mu_t^{M,\ka}(dxda) & \leq \itr |a|\,\mu_0(dxda)+\int_0^t \, ds|\pa_s\mY_s^{\ka}| \itr \frac{a}{|a|}\frac{q(x)}{\zeta_s^{M,\ka}(x)}\mu_s^{M,\ka}(dxda) \\
        & \leq^{\eqref{fondamentale}}\itr |a|\,\mu_0(dxda)+\int_0^t \, ds|\pa_s\mY_s^{\ka}| \int_{\T}|q(x)|\,dx,
    \end{align*}
   which concludes the proof.
\end{proof}

\begin{proof}[Proof of Lemma \ref{cor zeta}]
Recalling that $\zeta_t^{M,\ka}$ solves the PDE \eqref{PDEzetaMKE}, from \eqref{super} and Theorem \ref{lower bound w} we deduce 
\begin{equation*}%\label{54}
    \zeta_t^{M,\ka}(x) \geq e^{-t\cE(M)}\zeta_0(x),
\end{equation*}
where $\cE$ is defined as 
\begin{align*}
    \cE(M) & :=|V^{'}|_{L^{\infty}(\T;\R)}^2+\left | \Gamma_M(\cdot ,\mu_t^{M,\ka})\right |_{L^{\infty}(\T;\R)}^2+|V^{''}|_{L^{\infty}(\T;\R)}^2+\left | \Gamma_M(\cdot ,\mu_t^{M,\ka})\right |_{L^{\infty}(\T;\R)}^2.
\end{align*}
From the above, an upper bound for $\cE$ (and consequently a lower bound for $\zeta_t^{M,\ka}$) independent of $M \in \N$ can be easily derived. Indeed, since
\begin{align*}
\cE(M) & \leq |V^{'}|_{L^{\infty}(\T;\R)}^2+ |F^{'}|_{L^{\infty}(\T;\R)}^2 \left ( \,\,\, \itr |a|\mu_t^{M,\ka} \right )^2+ |V^{''}|_{L^{\infty}(\T;\R)} + |F^{''}|_{L^{\infty}(\T;\R)} \itr |a|\mu_t^{M,\ka},
\end{align*}
and recalling \eqref{super}, the lower bound \eqref{first lower} easily follows.
As for \eqref{second upper} we proceed with same method adopted to obtain $(ii)$ in Lemma \ref{stime} so we don't redo the whole proof, we only point out the main difference. In that proof we used upper bound \eqref{upper bound bouno} on $\pa_x^n\Gamma_M$. That bound is not sufficient here, as it would depend on $M$.
% but with a major difference, if we were to apply the strategy of that case we should estimate $\pa_x^n\Gamma_M(x,\mu_t^{M,\ka})$ from above in the following way:
%    $$\left | \pa_x^n\Gamma_M(x,\mu_t^{M,\ka}) \right | \leq M|F^{(n+1)}|_{L^{\infty}(\T;\R)},\,\, n \in \N$$
 %   which is not enough since with this approach we would obtain an upper bound on $|\zeta_t|_{H^2(\T;\R)}$ dependent on $M \in \N$. 
 To overcome this issue, we use upper bound \eqref{super} instead, to estimate $\pa_x^n\Gamma_M(x,\mu_t^{M,\ka})$:
    \begin{equation*}
        \left |\pa_x^n\Gamma_M(x,\mu_t^{M,\ka}) \right |\leq |F^{(n+1)}|_{L^{\infty}(\T;\R)} \itr |a|\,\mu_t^{M,\ka}(dxda)\leq^{\eqref{super}} C, \label{uniform in M}
    \end{equation*}
    with $C$ independent of $M$ (but dependent on $T$ and $\ka \in \N$).
    %which we know it to hold uniformly in $M \in \N$ due to \eqref{super} with the same procedure adopted to obtain $(ii)$ in Lemma \ref{stime} and by replacing estimate \eqref{upper bound bouno} upper bound \eqref{second upper} readily follows.
\end{proof}
\begin{proof} [Proof of Lemma \ref{stime 2}]
The lower bound $(i)$ and the estimate $(ii)$ can be proven similarly to $(i)$ and $(ii)$ in Lemma \ref{stime}, using the formulation \eqref{sistemaK-rho} of \eqref{sistemaK}, so we don't repeat the proof. 
%with $F^{'}*\rho_t^{\ka}$ in place of $\Gamma_M(x,\mu_t)$
The upper bound $(iii)$ follows from the equation of $A_t^{\ka}$ in \eqref{sistemaK-rho} and from $(i)$. Indeed, from Jensen's inequality we have 
\begin{align*}
    \itr |a|^n \mu_t^{\ka}(dxda)=\mE(|A_t^{\ka}|^n) \leq t^{n-1} \int_0^t \frac{|q(X_s^{\ka})|^n}{|\zeta_s^{\ka}(X_s^{\ka})|^n}|\pa_s\mY^{\ka}|^n\,ds,
\end{align*}
which gives $(iii)$. As for $(iv)$, we proceed in the same way we have proven \eqref{ineq gammma} in Lemma \ref{gamma beta}.
Hence, from the triangle inequality and the Lipschitz continuity of $F^{(n+1)}$ we obtain  
 \begin{align*}
     & \left | \pa_x^n\Gamma_M(x,\mu_t^{M,\ka}) - \pa_x^n\Gamma_M(y,\mu_t^{\ka}) \right | \\
     &\leq \left | \pa_x^n\Gamma_M(x,\mu_t^{M,\ka}) -\pa_x^n \Gamma_M(x,\mu_t^{\ka}) \right |+\left | 
     \pa_x^n\Gamma_M(x,\mu_t^{\ka}) - \pa_x^n\Gamma_M(y,\mu_t^{\ka}) \right | \\
     & \leq |F^{(n+1)}|_{L^{\infty}(\T;\R)}\cW_1(\mu_t^{M,\ka},\mu_t^{\ka})+|F^{(n+1)}|_{L^{\infty}
     (\T;\R)}|x-y|\itr \chi_M(a)\,\mu_t^{\ka}(dxda)\\
     & \leq |F^{(n+1)}|_{L^{\infty}(\T;\R)}\cW_1(\mu_t^{M,\ka},\mu_t^{\ka})+|F^{(n+1)}|_{L^{\infty}
     (\T;\R)}|x-y|\itr |a|\,\mu_t^{\ka}(dxda) \\
     & \leq^{\eqref{W1 leq W2}} |F^{(n+1)}|_{L^{\infty}(\T;\R)}\max \left \{ \,\,\itr |a|\,\mu_t^{\ka}(dxda),1 \right \} \left (|x-
     y|+\cW_2(\mu_t^{M,\ka},\mu_t^{\ka}) \right ).
 \end{align*} 
From \eqref{moments} and the above estimate, $(iv)$ follows. 
As for $(v)$ we refer the reader to Appendix \ref{proofs} for its proof. 
\end{proof}

\section{The limit $\ka \to \infty$}\label{limit kappa}
The purpose of this section is to prove Proposition \ref{prop limit rho kappa}. %This will be obtained as a consequence of Proposition \ref{prop rho kappa to 0} below in which we prove that $\rho_t^{\ka}$, solution to the PDE \eqref{PDErhoK}, converges in $L^2(\T;\R)$ to $\rho_t$, solution the PDE \eqref{rough PDE}, as $\ka \to \infty$. Before stating Proposition \ref{prop rho kappa to 0}, 
We start with some preliminary observations.
%The above integral can be viewed in two different ways  (choose one or the other depending on which one is more convenient): in the first case we fix $x \in \T$ and see $\mY_f(t,x)$ as a real-valued young integral; alternatively, we can see (due to Theorem \ref{young integral} with $E=L^2(\T;\R)$) $\mY_f$ as a time-continuous $L^2(\T;\R)$-valued function. It is clear that this two notions coincide as long as we can compute $\mY_f$ point-wise in $x \in \T$. 
%Then the following regularity result holds.
\begin{lemma}\label{6.2}
Let $f \in C^{\infty}([0,T]\times\T;\R)$ and $Z \in C^{\gamma}([0,T];\R)$ with $\gamma \in (0,1)$; let us define the Young integral $\cI_f(t,x)$ as
\begin{equation*}
    \cI_f(t,x)=\int_0^tf_s(x)\,dZ_s,\qquad t \in [0,T],\,x \in \T.
\end{equation*}
%for any given $f \in C^{\infty}([0,T]\times\T;\R)$ and $Z \in C^{\gamma}([0,T];\R)$ with $\gamma \in (0,1)$
Then $\cI_f(t,x) \in C([0,T];H^1(\T;\R))$ and its weak derivative is given by $\cI_{\pa_xf}$. That is, 
\begin{equation}\label{weak derivative}
    \frac{\pa}{\pa x}\cI_f(t,x)=\int_0^t \pa_x f_s(x)\,dZ_s,\,\,\,\,x \in \T,\,\,\,\,t \in [0,T].
\end{equation}
and, moreover, the following estimate holds
\begin{align}\label{kuos 20}
    \sup \limits_{t \in [0,T]}|\cI_f|_{H^1(\T;\R)} & \lesssim_{\gamma}\max \{T,1\}\big( |f|_{L^{\infty}([0,T] \times \T;\R)}+|\pa_tf|_{L^{\infty}([0,T] \times \T;\R)}\\ \notag
    & +|\pa_xf|_{L^{\infty}([0,T] \times \T;\R)}+|\pa_x\pa_tf|_{L^{\infty}([0,T] \times \T;\R)}\big ) |Z_{\cdot}|_{\frac{1}{\gamma},T,\R},
\end{align}
where $|\cdot|_{\frac{1}{\gamma},T,\R}$ is the semi-norm introduced in \eqref{definizione p-variation} with $E=\R$.
\end{lemma}
The proof of Lemma \ref{6.2} can be found in Appendix \ref{young}.
%From the above lemma the following well-posedness result can be proven as in \cite[Section 4]{angeli2023well}.
If we now use Lemma \ref{6.2} with $Z_{t}=Y_{t}$ (we recall that $ Y_t \in C^{\alpha}$ for some $\alpha \in (0,1)$) and $f_t(x)=e^{t\pa_{xx}}q(x)$, then from \eqref{kuos 20} we obtain that $\qq_t^{(Y)} \in H^1(\T;\R)$ for any given $t \geq 0$. Furthermore, from the semi-group property of the heat semi-group \eqref{heat} we get $t \to \qq_t^{(Y)} \in C([0,T];H^1(\T;\R))$.\\
%As a result of that, the following well-posedness result for the PDE \eqref{rough PDE} holds.
%\begin{thm}(cfr. \cite[Section 4]{angeli2023well})
%    The PDE \eqref{rough PDE} admits a unique time-continuous $L^2(\T;\R)$-valued mild solution.
%\end{thm}
%As a corollary of estimate \eqref{kuos 20}, we obtain that the convolution $\mY_q$ belongs to $C([0,T];H^1(\T;\R))$. 
Similarly, applying Lemma \ref{6.2} to the family of piece-wise $C^1$ functions $\{Y_{\cdot}^{\ka}\}_{\ka \in \N}$ approximating $Y_t$ in $C^{\gamma}$ (with $\gamma< \alpha$) we obtain that $\left \{ \qq_{\cdot}^{(Y^{\ka})} \right \}_{\ka \in \N}$ is bounded in $C([0,T];H^1(\T;\R))$, uniformly with respect to $\ka \in \N$. Indeed, from \eqref{kuos 20} applied with $Z_t=Y_t^{\ka}$ and $\gamma < \alpha$ we obtain 
\begin{align}\label{uniform bound on difference}
    \sup \limits_{t \in [0,T]}\left |\qq_t^{(Y^{\ka})} \right |_{H^1(\T;\R)} & \lesssim_{\gamma}\max\{T,1\}  |q|_{H^3(\T;\R)}|\mY_{\cdot}^{\ka}|_{\frac{1}{\gamma},T,\R},\qquad\text{for every $\ka \in \N$},
\end{align}
and from \eqref{kuos 20} applied to $Z_t=Y_t-Y_t^{\ka}$ and $\gamma < \alpha$ we obtain 
  \begin{equation}\label{uniform }
      \sup \limits_{t \in [0,T]}\left |\qq_t^{(Y)}-\qq_t^{(Y^{\ka})} \right |_{H^1(\T;\R)}\lesssim_{\gamma}\max\{T,1\}  |q|_{H^3(\T;\R)}|\mY_{\cdot}-\mY_{\cdot}^{\ka}|_{\frac{1}{\gamma},T,\R},\qquad\text{for every $\ka \in \N$}.
  \end{equation}
  Since we have assumed that $Y_{\cdot}^{\ka} \xrightarrow[]{\kappa \to \infty} Y_{\cdot}$ in $C^{\gamma}([0,T];\R)$, $\gamma <\alpha$, from \eqref{uniform bound on difference} it follows that there exists $\bar{\cC}>0$ independent of $\ka \in \N$ such that
$\sup \limits_{\ka \in \N} \left |Y_{\cdot}^{\ka}\right |_{\frac{1}{\gamma},T,\R} \leq \bar{\cC}$ and, consequently, 
\begin{equation}\label{uniform bound}
            \sup \limits_{\ka \in \N}\, \sup \limits_{t \in [0,T]}\left |\qq_{t}^{(Y^{\ka})} \right |_{H^1(\T;\R)}\lesssim_{\gamma}\bar{\cC} \max\{T,1\} |q|_{H^3(\T;\R)}.
\end{equation}
With this observation in place, since the function $\mathcal J$ in \eqref{stime rhoka} is locally bounded, from \eqref{uniform bound} we have 
\begin{equation}\label{uniform rho}
   \sup \limits_{t \in [0,T]} \left [|\rho_t^{\ka}|_{L^2(\T;\R)}^2+\int_0^T |\pa_x \rho_s^{\ka}|_{L^2(\T;\R)}^2 \,ds \right ] \leq \sup \limits_{\ka \in \N} \mathcal{J}\left (|\rho_0|_{L^2(\T;\R)}^2,\left |\qq_{\cdot}^{(Y^{\ka})} \right |_{C([0,T];H^1(\T;\R))} \right ) \leq \hat{\cC},
\end{equation}
for some $\hat{\cC}>0$ independent of $\ka \in \N$.
We can now prove Proposition \ref{prop limit rho kappa}.
\begin{proof}[Proof of Proposition \ref{prop limit rho kappa}]
Let $f \in C^{\infty}(\T;\R)$ then from H\"older's inequality we have
\begin{align*}
    \sup \limits_{t \in [0,T]} \left | \langle f,\rho_t^{\ka} \rangle_{L^2(\T;\R)}-\langle f,\rho_t \rangle_{L^2(\T;\R)} \right | \leq |f|_{L^2(\T;\R)} \sup \limits_{t \in [0,T]} \left| \rho_t^{\ka} - \rho_t \right |_{L^2(\T;\R)}.
\end{align*}
Hence, if we prove that the following limit holds
\begin{equation}\label{lim rho_K }
    \lim \limits_{\ka \to \infty}\sup \limits_{t \in [0,T]}|\rho_t^{\ka}-\rho_t|_{L^2(\T;\R)}=0,
\end{equation}
then the proof of Proposition \ref{prop limit rho kappa} is concluded.

        To prove \eqref{lim rho_K }, let us first recall the following standard heat kernel estimate (the proof of which can be found in \cite[Lemma 4.1]{angeli2023well}): for every $f \in C([0,T];L^2(\T;\R))$ we have 
        \begin{equation}\label{heat kernel estimate section 7}
      \left | \int_0^t e^{(t-s)\pa_{xx}} \pa_x f_sds \right |_{L^2(\T;\R)} \leq \hat C \int_0^t (t-s)^{-\frac{1}{2}}|f_s|_{L^2(\T;\R)}\,ds,\qquad t \in [0,T],
        \end{equation}
         where $\hat C>0$ is a constant independent of $f$.
		We now recall (see Lemma \ref{lemma: well-posedness of rhoPDE}) that the mild and weak solution of \eqref{PDErhoK} coincide, and the same is true for equation \eqref{rough PDE}. In this proof, we use the mild formulations of both equations (but we could have done it using the weak formulation as well). From the mild formulation of \eqref{rough PDE} and \eqref{PDErhoK} and the estimate \eqref{heat kernel estimate section 7} we have  (in what follows $C$ denotes a generic constant independent of $\ka \in \N$ but possibly dependent on $T$, which may change from line to line)
		\begin{align*}			
			| \rho_t-\rho_t^{\ka}|_{L^2(\T;\R)} & \leq  \sup \limits_{t \in [0,T]} \left |\qq_t^{(Y)}-\qq_t^{(Y^{\ka})} \right |_{L^2(\T;\R)}+C\int_0^t (t-s)^{-\frac{1}{2}} |V^{'}(\rho_s-\rho_s^{\ka} )|_{L^2(\T;\R)}\,ds\\
			& + C\int \limits _0^t (t-s)^{-\frac{1}{2}} | (F^{'} \ast \rho_s) \rho_s-(F^{'} \ast \rho_s^{\ka}) \rho_s^{\ka}|_{L^2(\T;\R)}\,ds \\
			& \leq  \sup \limits_{t \in [0,T]} \left |\qq_t^{(Y)}-\qq_t^{(Y^{\ka})} \right |_{L^2(\T;\R)}+C\int  \limits_0^t (t-s)^{-\frac{1}{2}} |\rho_s-\rho_s^{\ka} |_{L^2(\T;\R)}\,ds\\
        & + C\int_0^t (t-s)^{-\frac{1}{2}} | (F^{'} \ast (\rho_s-\rho_s^{\ka})) \rho_s|_{L^2(\T;\R)}\,ds\\
        & +C \int_0^t (t-s)^{-\frac{1}{2}}|(F^{'} \ast \rho_s^{\ka}) (\rho_s-\rho_s^{\ka})|_{L^2(\T;\R)}\,ds, \\			
		\end{align*}
		where we recall that $\qq_t^{(Y)}$ and $\qq_t^{\left ( Y^{\ka}\right )}$ are defined by \eqref{notationfor stochastic convolution} with $Z_t=Y_t$, $t\geq 0$, and $Z_t=Y_t^{\ka}$, $t \geq 0$, respectively.
		Since $\rho_t$ is the solution to \eqref{rough PDE} (which has the same structure of the PDE \eqref{PDErhoK} solved by $\rho_t^{\ka}$) with the same strategy adopted for obtaining \eqref{stime rhoka} one can see that the exact estimate \eqref{stime rhoka} holds for $\rho_t$ as well so that there exists $\tilde C=\tilde C(T)>0$ for which $\sup \limits_{t \in [0,T]} |\rho_t|_{L^2(\T;\R)}\leq \tilde C$. %from Weierstrass theorem for continuous functions we deduce that there exists $\tilde {C}_T < +\infty$ such that $\sup \limits_{t \in [0,T]} |\rho_t|_{L^2(\T;\R)} \leq \tilde C$. 
        Hence, we obtain  
        \begin{align*}
			\left | \left ( F^{'} \ast \rho_s \right )\left(\rho_s- \rho_s^{\ka}\right)\right |_{L^2(\T;\R)}
			\leq C &|\rho_s|_{L^2(\T;\R)}|\rho_s-\rho_s^{\ka}|_{L^2(\T;\R)} \leq C\cdot \tilde C|\rho_s-\rho_s^{\ka}|_{L^2(\T;\R)},
		\end{align*}
		and, similarly, using \eqref{uniform rho}, we have 
		\begin{align*}
			\left | \left ( F^{'} \ast (\rho_s-\rho_s^{\ka}) \right )\rho_s^{\ka} \right |_{L^2(\T;\R)}\,\leq\,C \cdot \hat C|\rho_s-\rho_s^{\ka} |_{L^2(\T;\R)},
		\end{align*}
		for all $s \in [0,T]$.
		From the above, we then have 
		\begin{equation}
			\begin{split}
				| \rho_t - \rho_t^{\ka} |_{L^2(\T;\R)}  & \leq\, C\left [ \sup \limits_{t \in [0,T]} \left |\qq_t^{(Y)}-\qq_t^{(Y^{\ka})} \right |_{L^2(\T;\R)} + (\hat C+ \tilde C)\int_0^t (t-s)^{-\frac{1}{2}}|\rho_s-\rho_s^{\ka}|_{L^2(\T;\R)}\,ds \right ], \notag
			\end{split}
		\end{equation}
		so that using Generalised Gronwall's lemma inequality (see e.g \cite[Lemma 4.3]{angeli2023well}) finally gives
\begin{equation*}
    \sup \limits_{t \in [0,T]}|\rho_t-\rho_t^{\ka}|_{L^2(\T;\R)} \leq \bar{C} \sup \limits_{t \in [0,T]} \left |\qq_t^{(Y)}-\qq_t^{(Y^{\ka})} \right |_{L^2(\T;\R)}
\end{equation*} 
  where $\bar{C}>0$ depends on $|V^{'}|_{L^{\infty}(\T;\R)}$, $|F^{'}|_{L^{\infty}(\T;\R)}$, $|q|_{H^3(\T;\R)}$, $\hat C$ and $\tilde C$. 
  %Once this is in place, let us note that by applying estimate \eqref{uniform bound on difference} to $Z_t=Y_t-Y_t^{\ka}$ and $\gamma < \alpha$ i.e. to the (Young) integral 
  %\begin{equation*}
  %    \int_0^t e^{(t-s)\pa_{xx}}q(x)d(\mY_s-\mY_s^{\ka}),
  %\end{equation*}
  %we obtain 
  %\begin{equation*}
   %   \sup \limits_{t \in [0,T]}\left |\mY_q-\mY_q^{\ka} \right |\lesssim_{\gamma}\max\{T,1\}  |q|_{H^3(\T;\R)}|\mY-\mY^{\ka}|_{\frac{1}{\gamma},T,\R}.
  %\end{equation*}
  %Hence, we deduce 
  %From \eqre
  %\begin{equation*}
  %    |\rho_t-\rho_t^{\ka}|_{L^2(\T;\R)} \lesssim_{\gamma} \max\{T,1\}|q|_{H^{3}(\T;\R)}\tilde{C}_{T} |\mY-\mY^{\ka}|_{\frac{1}{\gamma},T,\R},\,\,\,t \in [0,T].
  %\end{equation*}
  Now using \eqref{uniform } and by recalling that $|\mY-\mY^{\ka}|_{\frac{1}{\gamma},T,\R} \xrightarrow[]{\kappa \to \infty}0$ the assertion readily follows.
\end{proof}
\section{Proof of Theorem \ref{theorem infinite dimensional approximation}}\label{infinite noise bm}
The purpose of this section is to prove Theorem \ref{theorem infinite dimensional approximation}. This is obtained by first proving Proposition \ref{prop limit m to infty} and then Corollary \ref{corollary propositions combined}. We start by considering the following class of PDEs:
\begin{equation}\label{rough PDE infinite}
    \begin{dcases}
        & \pa_t \rho^{\infty}_t=\pa_{xx}\rho_t^{\infty}+\pa_x[(V^{'}+F*\rho^{\infty}_t)\rho^{\infty}_t]+\sum \limits_{z \in \Z} \lambda_z e_z\pa_t Y_t^z,\qquad t \in (0,T],\\
        & \rho^{\infty}|_{t=0}=\rho_0,
    \end{dcases}
\end{equation}
where $\{e_z\}_{z \in \Z}$ denotes the basis on $L^2(\T;\R)$ defined in \eqref{fourier}, $\{\lambda_z\}_{z \in \Z}$ is as in \eqref{infinite dimensional noise}, lastly, $\{Y_{\cdot}^z\}_{z \in \Z}$ is a family of $\alpha$-H\"older time-continuous paths, $\alpha \in (0,1)$.
%In this context, the setting of the space-time white noise is recovered and can be treated in a deterministic way i.e. the following class of SPDEs can be treated in a deterministic way, using \eqref{infinite dimensional noise} we rewrite \eqref{SPDEintro} as:
Analogously to what we have done in previous sections, this will allow to treat deterministically equation \eqref{SPDEintro}, which, by using \eqref{infinite dimensional noise}, can be rewritten as
\begin{equation}\label{SPDE space-time white noise}
    \begin{dcases}
        & \pa_t v_t=\pa_{xx}v_t+\pa_x[(V^{'}+F*v_t)v_t]+ \sum \limits_{z \in \Z} \lambda_z e_z\pa_tw_t^z,\,\,\,\, t \in (0,T],\\
        & v|_{t=0}=v_0
    \end{dcases}
\end{equation}
where $\{w_{\cdot}^z\}_{z \in \Z}$ is as in \eqref{infinite dimensional noise}. The well-posedness of the SPDE \eqref{SPDEintro} is stated  in Lemma \ref{lemma: well-posedness of rhoPDE}  and it is a consequence of the lemma below, which is a straightforward re-adaptation to our context, of \cite[Theorem 2.13]{dapra04}.
\begin{lemma}\cite[Theorem 2.13]{dapra04}\label{lemma differentiability stochastic convolution}
Let the eigenvalues $\{\lambda_z\}_{z \in \Z}$  in \eqref{infinite dimensional noise} satisfy the following condition: 
        \begin{equation}\label{eigenvalues constraint}
                    \text{$\exists \,\delta \in \left (0,\frac{1}{2} \right )$ such that } \sum \limits_{z \in \Z} \lambda_z^2|z|^{4\delta} <+\infty;
        \end{equation}
then the (stochastic) convolution $t \to \qq_t^{(W)}$ belongs to $C \left ([0,T];H^1(\T;\R) \right )$,  $\hat \mP$-a.s. and 
\begin{equation*}
    \hat \mE \left ( \sup \limits_{t \in [0,T]} \left |\qq_t^{(W)} \right |_{H^1(\T;\R)}^2\right ) < +\infty.
\end{equation*}
\end{lemma}
%Following \cite[Section 4]{angeli2023well} to prove well-posedness of \eqref{SPDE space-time white noise} one needs the stochastic convolution 
%\begin{equation*}
% \sum \limits_{z \in \Z} \lambda_z e_z(x) \int_0^t e^{-(t-s)z^2}\,dw_s^z,\qquad t \in [0,T],
%\end{equation*}
%to be in $H^1(\T;\R)$, $\mP$-a.s. Following \cite[Hypothesis 2.6 and Hypothesis 2.10, (v)]{dapra04}, this is ensured by the condition: 
%\begin{equation}\label{eigenvalues constraint}
%    \text{$\exists \delta \in (0,\frac{1}{2})$ such that } \sum \limits_{z \in \Z} \lambda_z^2 |z|^{4\delta} <+\infty.
%\end{equation}
%Indeed, from the Burkholder-Davis-Gundy inequality for scalar martingales and the fact that the basis $\{e_z\}_{z \in \Z}$ is uniformly bounded in $L^{\infty}(\T;\R)$ ($|e_z(x)| \leq \frac{1}{\sqrt{\pi}}$, for every $x \in \T,\,\,z \in \Z)$ one can see that 
%\begin{align*}
%   & \mE \,\,\,\sum \limits_{z \in \Z} z^2\lambda_z^2 \sup \limits_{t \in [0,T]} \left |\int_0^t e^{-(t-s)z^2}\,d\beta_s^z \right |^2  \\
 %  & \lesssim_T \sum \limits_{z \in \Z} z^2\lambda_z^2 \int_0^T e^{-2(T-s)z^2}\,ds \lesssim_T \sum \limits_{z \in \Z} \lambda_z^2<+\infty .
%\end{align*}
Let us also point out that from \eqref{eigenvalues constraint} and by following the steps of the proof of \cite[Theorem 2.13]{dapra04} one also obtains  
\begin{align}\label{remainder goes to zero BM}
  & \lim \limits_{L \to \infty} \,\, \sup \limits_{t \in [0,T]} \,\, \sum \limits_{|z| \geq L} z^2\lambda_z^2 \left |\int_0^t e^{-(t-s)z^2}\,dw_s^z \right |^2=0,\qquad \hat \mP-a.s. 
\end{align}
We refer the reader to Appendix \ref{proofs} for the proof of \eqref{remainder goes to zero BM}. 
Since our aim is to study \eqref{SPDE space-time white noise} through deterministic systems of type \eqref{rough PDE infinite}, from Lemma \ref{lemma differentiability stochastic convolution} and \eqref{remainder goes to zero BM} we impose the following conditions on the H\"older time-continuous trajectories $Y_t^{z}$ in \eqref{rough PDE infinite}:
\begin{itemize}
    \item The convolution $\mathcal{Y}_{\cdot}$ defined as 
    \begin{equation}\label{deterministic convolution}
        \mathcal{Y}_t(x):= \sum \limits_{z \in \Z} \lambda_z e_z(x) \int_0^t e^{-(t-s)z^2}\,d\mY_s^z,\qquad t \in [0,T],\,\,\,\, x \in \T,
    \end{equation}
    is a time-continuous $H^1(\T;\R)$-valued process. That is, $\mathcal {\cY}_{\cdot} \in C([0,T];H^1(\T;\R))$.
    \item The following condition holds 
    \begin{equation}\label{vincolo remainder}
        \lim \limits_{L \to \infty} \,\, \sup \limits_{t \in [0,T]} \,\, \sum \limits_{|z| \geq L} z^2\lambda_z^2 \left |\int_0^t e^{-(t-s)z^2}\,dY_s^z \right |^2=0.
    \end{equation}
\end{itemize}
Condition \eqref{vincolo remainder} is a technical condition needed to prove the limit $m \to \infty$ and it becomes clear in the proof of Proposition \ref{prop limit m to infty} why we have imposed it on $\{Y_{\cdot}^z\}_{z \in \Z}$. \footnote{Let us point out that Lemma \ref{lemma differentiability stochastic convolution} and \eqref{remainder goes to zero BM} show that set given by the family of trajectories $\{Y_{\cdot}^z\}_{z \in \Z}$ for which the convolution  $\cY_t$ defined as in \eqref{deterministic convolution} belongs to $C([0,T];H^1(\T;\R))$ and that satisfies condition \eqref{vincolo remainder} is non-empty. Note moreover that since $\cY_{\cdot}$ is well-defined in $C([0,T];H^1(\T;\R))$, condition \eqref{vincolo remainder}, without the supremum on $[0,T]$, is automatically true. Hence, \eqref{vincolo remainder} is an extra condition because of the presence of the supremum.} 
The well-posedness of the PDE \eqref{rough PDE infinite} can be found in Lemma \ref{lemma: well-posedness of rhoPDE} and it follows from the lemma below.
\begin{lemma}\label{7.4}
The convolution $\cY_{\cdot}$ defined in \eqref{deterministic convolution} belongs to $C([0,T];H^1(\T;\R))$ and its weak derivative is given by
  \begin{equation}\label{weak derivative infinite}
      \pa_x\cY=\sum \limits_{z \in \Z} |z|\lambda_z \int_0^t e^{-(t-s)z^2}e_{-z}(x)\,d\mY_s^z,\,\,\,\, t \in [0,T],\,\,\,\, dx-a.e.
  \end{equation}
\end{lemma}
\begin{proof}
    We just have to show \eqref{weak derivative infinite}. Let $\varphi \in C^{\infty}(\T;\R)$ be a test function then from Fubini-Tonelli's theorem (which can be applied since $\cY \in C([0,T];H^1(\T;\R))$) we have 
    \begin{align*}
        &\langle \cY, \pa_x\varphi\rangle_{L^2(\T;\R)}=\sum \limits_{z \in \Z} \lambda_z\int_0^t e^{-(t-s)z^2}\,d\mY_s^z\langle e_z,\pa_x\varphi \rangle_{L^2(\T;\R)}\\
        &=-\sum \limits_{z \in \Z} |z|\lambda_z\int_0^t e^{-(t-s)z^2}\,d\mY_s^z\langle e_{-z},\varphi \rangle_{L^2(\T;\R)}=-\left \langle \,\,\,\,\sum \limits_{z \in \Z} |z|\lambda_ze_{-z} \int_0^t e^{-(t-s)z^2}\,d\mY_s^z , \varphi \right \rangle_{L^2(\T;\R)},
    \end{align*}
    which implies the statement.
\end{proof}

\noindent
 
To prove Theorem \ref{theorem infinite dimensional approximation} we begin with observing the following fact: we already know that Theorem \ref{main_thm2} holds; therefore, to conclude it suffices to extend Theorem \ref{main_thm2} to a sum of infinitely many Brownian motions. To start, we know that Proposition \ref{prop linear combination Intro} allows us to extend Theorem \ref{main_thm2} to a sum of finitely many Brownian motions. Hence, to conclude only the limit $m\to \infty$ remains to be shown. With this direction in mind, we first prove a preliminary result, Lemma \ref{8.6}, which provides us a uniform upper bound in $m \in \N$ on the $L^2(\T;\R)$-norm of $\rho_t^m$ so that the limit $m\to \infty$ in the deterministic setting can be taken care of, this is done in Proposition \ref{prop limit m to infty}. After that, we state and prove Corollary \ref{corollary propositions combined} and at last demonstrate Theorem \ref{theorem infinite dimensional approximation}.
\begin{lemma}\label{8.6}
Let $\rho_t^{m}$ be the mild solution to \eqref{rough approx}; then the following holds:
\begin{equation*}%\label{uniform rho omega}
    \sup \limits_{m \in \N} \sup \limits_{t \in [0,T]} |\rho_t^{m}|_{L^2(\T;\R)}<+\infty.
\end{equation*}
\end{lemma}
\begin{proof}
   %To conclude, we are left to show that $\{\rho_t^{m}\}_{M \in \N}$ converges to $\rho_t$ in $C([0,T];L^2(\T;\R))$ as $m \to +\infty$. Before doing so, we need a uniform upper bound (uniform in $M \in \N$) for the $L^2(\T;\R)$-norm of $\{\rho_t^{m}\}_{t \in [0,T]}$. This can be obtained in the same way as \eqref{uniform rho}. Indeed, 
   From \cite[Section 4]{angeli2023well} we know that 
    \begin{equation*}
    \sup \limits_{t \in [0,T]}|\rho_t^{m}|_{L^2(\T;\R)}^2+\int_0^t |\pa_x\rho_s^{m}|_{L^2(\T;\R)}^2 \,ds \leq \mathcal{J} \left (|\rho_0|_{L^2(\T;\R)}^2,\sup \limits_{t \in [0,T]}|\cY^{m}|_{H^1(\T;\R)} \right ),
\end{equation*}
where $\mathcal{J}$ is the same function of estimate \eqref{stime rhoka} and $\cY_t^{m}$ is defined as
\begin{equation}\label{deterministic convolution truncated}
\cY_t^m(x):=\sum \limits_{|z| \leq m} \lambda_z e_z(x) \int_0^t e^{-(t-s)|z|^2}\,dY_s^z,\qquad t \in [0,T],\,x \in \T,\,m \in \N.    
\end{equation} 
Furthermore, since $\cY_{\cdot} \in C([0,T];H^1(\T;\R))$ we obtain 
$$ \sup \limits_{m \in \N} \sup \limits_{t \in [0,T]} |\cY_t^{m}|_{H^1(\T;\R)}^2 < +\infty $$ 
from which the statement readily follows.
\end{proof}
\begin{prop}\label{prop limit m to infty}
    Let $ \rho_t^{m} $ be the solution to the PDE \eqref{rough approx} and let $\rho_t^{\infty}$ be the solution to the PDE \eqref{rough PDE infinite}; then the following limit holds:
    \begin{equation*}
        \lim \limits_{m \to \infty} \sup \limits_{t \in [0,T]} |\rho_t^{m}-\rho_t^{\infty}|_{L^2(\T;\R)}=0.
    \end{equation*}
\end{prop}
\begin{proof}
    Recall that $\rho_t^{m}$ and the $\rho_t^{\infty}$ satisfy 
		\begin{align*}
			\rho^{m}_t\,&=e^{t\pa_{xx}}\rho_0+\int_0^t e^{(t-s)\pa_{xx}}\pa_x[(V^{'}+F^{'}*\rho_s^{m})\rho_s^{m}]\,ds+\cY_t^{m},\\
			\rho_t^{\infty}\,&=e^{t\pa_{xx}}\rho_0+\int_0^t e^{(t-s)\pa_{xx}}\pa_x[(V^{'}+F^{'}*\rho_s^{\infty})\rho_s^{\infty}]\,ds+\cY_t,
		\end{align*}
		where $\cY_t^{m}$ and $\cY_t$ are defined in \eqref{deterministic convolution truncated} and \eqref{deterministic convolution}, respectively.
        By following the steps that allowed us to prove the limit \eqref{lim rho_K } and by using Lemma \ref{8.6} we obtain 
	\begin{equation*}
	   \sup \limits_{t \in [0,T]} |\rho_t^{\infty}-\rho_t^{m}|_{L^2(\T;\R)} \leq \tilde{c}_{T} \sup \limits_{t \in [0,T]} |\cY_t-\cY_t^{m} |_{L^2(\T;\R)},
	\end{equation*}	
 where $\tilde{c}_T$ a suitable constant depending on $T$.
   %for all $t \in [0,T]$, where $\tilde{c}_{T}>0$ depends on $|V^{'}|_{L^{\infty}(\T;\R)}$, $|F^{'}|_{L^{\infty}(\T;\R)}$ and $M_T$. 
   Hence, to conclude it suffices to prove \linebreak $\sup \limits_{t \in [0,T]}|\cY_t-\cY_t^{m} |_{L^2(\T;\R)} \xrightarrow[]{m \to \infty} 0$. This follows from Parseval's identity and \eqref{vincolo remainder}. The proof is thus concluded.
   %we know that 
%\begin{align*}
%      \sup \limits_{t \in [0,T]}|\cY_t-\cY_t^{m}|_{L^2(\T;\R)}^2=\sup \limits_{t \in [0,T]} \sum \limits_{|z| > m} \lambda_z^2 \left |\int_0^t e^{-(t-s)z^2} \,d\mY_s^z\right |^2 \xrightarrow[]{m \to \infty}0,
%\end{align*}
%from which the assertion readily follows.
\end{proof}
\begin{cor}\label{corollary propositions combined}
   %If we let $X_t^{N,\e,M,\ka,m}$ and $A_t^{N,\e,M,\ka,m}$ be as in \eqref{infinite dimensional IPS}.
%Let the eigenvalues $\{\lambda_z\}_{z \in \Z}$ satisfy condition \eqref{eigenvalues constraint}, 
Let $\rho_t^{N,\e,M,\ka,m}$ be the weighted empirical measure defined in \eqref{weighted_empirical_infinite_dimensional} and let $\rho_t$ be the solution to \eqref{rough PDE infinite}.
Then for any given $f \in C^{\infty}(\T;\R)$ the following limit holds 
\begin{equation*}
    \lim \limits_{m \to \infty} \lim_{\ka \to \infty} \lim_{M\to\infty} \lim_{\e\downarrow 0} \lim_{N\to\infty} \,\,\sup \limits_{t \in [0,T]}\,\,\left | \langle f,\rho_t^{N, \e, M, \kappa, m} \rangle \,-\, \langle f, \rho_t^{\infty} \rangle \right |\,=\,0,\,\,\,\,\mP-a.s.
\end{equation*}
\end{cor}
\begin{proof}
The proof follows by combining Proposition \ref{prop linear combination Intro} and Proposition \ref{prop limit m to infty}. 
%and by recalling that \ref{uniform rho omega} implies that 
\end{proof}
\begin{proof}[Proof of Theorem \ref{theorem infinite dimensional approximation}]
Let $\hat{\omega} \in \left (\Omega, \hat{\cF}, \{ \hat{\cF}_t \}_{t \geq 0}, \hat{\mP}\right)$ and let be $\solv_t(\hat{\omega})$ the corresponding realization of the solution to the SPDE \eqref{SPDEintro} and let us set $\hat{Y}_t^z$ appearing in \eqref{rough PDE infinite} to be equal to the realization of the Brownian motion $w_t^z(\hat{\omega})$ appearing in \eqref{SPDE space-time white noise} for every $z \in \Z$. 
%\begin{equation}\label{forcing = bm}
%    \hat{Y}_t^z=w_t^z(\hat{\omega}), \qquad t \geq 0,\,\, z \in \Z.
%\end{equation}    
Then, the claim follows by applying Corollary \ref{corollary propositions combined} to \eqref{rough PDE infinite} with the above choice $\{\hat{Y}_{\cdot}^z\}_{z \in \Z}$ of H\"older paths.     
\end{proof}

\appendix
\section{Well-Posedeness results}\label{well-posedness}

%%%%%%%%%%%%%%%%%%%torna 2

With the notation set in \eqref{notationfor stochastic convolution}, and recalling the definition of weak and mild solution given in Section \ref{section: main results}, see e.g. \eqref{mild solution} and \eqref{def: weak solution}, the following holds. 
\begin{lemma}\label{lemma: well-posedness of rhoPDE}
    Let $\mathcal U_{\cdot}$ (as in \eqref{generic PDE+forcing}) be such that $t \mapsto \qq^{(\mathcal U)}_t$  belongs to $C([0,T]; H^1(\T; \R))$ and let $\nu_0 \in C^{\infty}(\T;\R)$. Then equation \eqref{generic PDE+forcing} admits a unique mild solution. In particular this implies that equations \eqref{rough PDE}, \eqref{pde+bm},  \eqref{SPDEintro}, \eqref{PDErhoK} and \eqref{rough PDE infinite} are all well-posed in $C([0,T];L^2(\T;\R))$ in mild sense, and, moreover, such mild solutions coincide with weak solutions.

    In addition, the (S)PDE \eqref{generic PDE+forcing} enjoys regularising properties. That is, if $\nu_t$ is a measure-valued weak solution to \eqref{generic PDE+forcing} with initial datum $\nu_0 \in C^{\infty}(\T;\R)$ then it 
 admits a density belonging to $L^2(\T;\R)$. To be precise, $\nu_t$ has the same spatial regularity of the convolution $t \to \qq_t^{(\cU)}$ i.e. if $\qq_t^{(\cU)} \in H^{k}(\T;\R)$ ($k \in \N$) then so does the solution $\nu_t$. In particular, weak measure-valued solution of \eqref{PDErhoK} are also classical solutions (therefore belonging to $L^2(\T;\R)$) of such equation.
\end{lemma}
           \begin{proof} The existence and uniqueness of a mild solution for \eqref{SPDEintro} has been proved in \cite[Theorem 2.4]{angeli2023well}. Since \eqref{generic PDE+forcing} is a  more general version of \eqref{SPDEintro}, and \eqref{rough PDE}, \eqref{pde+bm}, \eqref{PDErhoK} and \eqref{rough PDE infinite} are specific instances of \eqref{generic PDE+forcing}, one just needs to prove well-posedness of \eqref{generic PDE+forcing}.  This can be done following the steps of \cite[proof of Theorem 2.4]{angeli2023well}, which goes through as long as the stochastic convolution $\qq_t^{(\mathcal U)}$ in \eqref{notationfor stochastic convolution}  belongs to $C\left ([0,T];H^1(\T;\R)\right)$. 
           This is true when $\mathcal U_t$ is equal to $W(t,x)$ or $\sum \limits_{z \in \Z}\lambda_ze_z\pa_tY_t^{z}$ by Lemma \ref{lemma differentiability stochastic convolution} and Lemma \ref{7.4}, respectively (giving equations \eqref{SPDEintro} and \eqref{rough PDE infinite}, respectively). 
           It is clearly the case when $\mathcal U_t$ is equal to $q(x)Y_t$, $q(x)w_t$ or $q(x)Y_t^{\ka}$, by Lemma \ref{6.2} (giving equations \eqref{rough PDE}, \eqref{pde+bm}, \eqref{PDErhoK}, respectively), since $q$ is smooth and  $Y_t$,  $w_t$ and $Y_t^{\ka}$ are all at least H\"older time-continuous (and space-independent). The fact that the weak and mild solutions coincide has been proven in \cite[Proposition 3.5]{gyongy1998existence}. We just point out that in the context of \cite{gyongy1998existence} the equivalence between weak and mild solution has been shown for SPDEs satisfying Dirichlet boundary conditions (DBC). In our instance, we are dealing with SPDEs having periodic boundary conditions (PBC). This doesn't create any additional issue since the proof of \cite[Proposition 3.5]{gyongy1998existence} can be applied to our context, without making any modifications, once we have replaced $G(t)$ (the heat kernel of the heat equations with DBC) in \cite[Proposition 3.5]{gyongy1998existence} with $e^{t\pa_{xx}}$ (the heat semi-group of the heat equation with PBC).

           To prove the smoothing properties of \eqref{generic PDE+forcing} we first need to recall what a measure-valued weak solution is. A process $\nu_t$ is a weak-measure valued solution to \eqref{generic PDE+forcing} if $t \to \nu_t \in C([0,T];\cM(\T))$ (we endow $\cM(\T)$ with the total variation norm) and for any given $f \in C^{\infty}(\T;\R)$ the following equality is satisfied: ($\bar \mP$ or $\hat \mP$ a.s. if $\cU_t$ is equal to $q(x)w_t$ or $W_t$, respectively)
           \begin{equation}\label{weak sol for measures}
                \langle f,\nu_t \rangle=\langle f,\nu_0 \rangle + \int_0^t \langle \pa_{xx}f,\nu_s \rangle\,ds-\int_0^t \langle \pa_xf,(V^{}+F^{'}*\nu_s)\nu_s \rangle\,ds+ \langle f,\cU_t \rangle .
           \end{equation}
           Following \cite[Proof of Theorem 2]{capasso2019mean} or \cite[Theorem 3.3]{flandoli2017mean} one can see that if $\nu_t$ is a weak-measure valued solution in the sense of \eqref{weak sol for measures} then it is a $H^{-2}$-valued mild solution (we point out that the converse holds as well i.e. a mild solution is a weak solution. However, we don't need this here). Namely, the following equality holds in $H^{-2}(\T;\R)$\footnote{In this proof we are using a broader definition of Sobolev space than the one given in \eqref{sobolev usual}. To be precise, for any $s \in \R$ we define $H^{s}(\T;\R)$ as the space of distributions $\cD$ such that $\sum \limits_{z \in \Z}(1+|z|^2)^s \langle e_z,\cD \rangle^2 < \infty,$ where $\{e_z\}_{z \in \Z}$ is the basis defined in \eqref{fourier}. Let us note that this extension is consistent as $H^s(\T;\R)$ with $s \in\N$ coincides with the Sobolev space given in \eqref{sobolev usual}.}:
           \begin{equation}\label{mild formulation for measures}
               \nu_t=e^{t\pa_{xx}}\nu_0+\int_0^t e^{(t-s)\pa_{xx}}\pa_x \left [(V^{'}+F^{'}*\nu_s) \nu_s \right ]\,ds+\qq_t^{(\cU)},\qquad t \in [0,T].
           \end{equation}
           We also the following regularising property of the heat semi-group (cfr. \cite[Lemma 2.3.3]{berglund2022introduction}) which will be used several times. Namely, let $\cD \in H^{s_1}(\T;\R)$ and let $s_2 >s_1$ then
           \begin{equation}\label{heat kernel regualrising estimate}
               \left | e^{t\pa_{xx}}\cD\right |_{H^{s_2}(\T;\R)} \lesssim_{s_1,s_2}t^{-\frac{s_2-s_1}{2}}\left | \cD\right |_{H^{s_1}(\T;\R)}.\
           \end{equation}
           This can be proven by using the Fourier representation of $e^{t\pa_{xx}}$ combined with the fact that the function $(1+x)^{s_2-s_1}e^{-tx} \lesssim_{s_1,s_2}t^{-(s_2-s_1)}$, for $x \geq 0$. 
           Furthermore, we will also need the following estimate of the $H^{s}(\T;\R)$-norm of $\nu_t$, with $s < -\frac{1}{2}$. Since $\nu_{\cdot} \in C([0,T];\cM(\T))$ then for any given $f \in C^{\infty}(\T;\R)$ we have that 
           \begin{align}\notag
             \left |f\nu_t \right |_{H^{s}(\T;\R)}^2 = \sum \limits_{z \in \Z} (1+|z|^2)^s\langle e_z,f\nu_t \rangle^2 & \leq \sum \limits_{z \in \Z} (1+|z|^2)^s \left [ \itt e_z(x)f(x)\nu_t(dx) \right ]^2 \\ \label{sitma via total variation}
              & \leq \sum \limits_{z \in \Z} (1+|z|^2)^s \frac{1}{\pi}|f|_{L^{\infty}(\T;\R)}^2|\nu_t|_{TV}^2 
           \end{align}
           which is finite as long as $s < -\frac{1}{2}$. 
           Throughout the next calculation we assume that the (stochastic) convolution $\qq_t^{(\cU)} \in H^1(\T; \R)$ (the extension to the instance $\qq_t^{(\cU)} \in H^k(\T;\R)$, $k \in \N$ can be done in the exact same way so we don't repeat it).
           
           Let us now set $s_1=-2$ and $s_2=-1+p$ for a small enough $p>0$; then from the mild formulation of $\nu_t$ we obtain
           \begin{align}\label{sample}
               |\nu_t|_{H^{-1+\e}(\T;\R)} & \lesssim^{\eqref{heat kernel regualrising estimate},\eqref{sitma via total variation}} |\nu_0|_{H^{-1+p}(\T;\R)} +\int_0^t (t-s)^{-1+p} \left ( \left |\nu_s \right |_{TV} +\left |\nu_s\right |_{TV}^2 \right ) \,ds+\left |\qq_t^{(\cU)}\right |_{H^{-2}(\T;\R)}.
           \end{align}
           By combining the above estimate with the fact that $\sup \limits_{t \in [0,T]}|\nu_t|_{TV} < \infty$ (this follows since $\nu_{\cdot} \in C([0,T];\cM(\T))$) and $\qq_{\cdot}^{(\cU)} \in C\left ([0,T];H^1(\T;\R) \right )$ we obtain $\nu_{.} \in L^{\infty}\left ([0,T];H^{-1+p}(\T;\R) \right )$. From equality \eqref{mild formulation for measures} we obtain the continuity in time, i.e. $\nu_{\cdot} \in C([0,T];H^{-1+p}(\T;\R))$.\footnote{This property holds in general in any separable Banach space $E$. Namely, if $E$ is a separable Banach space and $f \in L^{\infty}([0,T];E)$ then its (Bochner) integral $\int_0^t f_s\,ds$, $t \in [0,T]$, (see \cite[Appendix A]{prevot2007concise}) clearly belongs to $C([0,T];E)$.} 

           From this moment onward, one proceeds in a similar fashion. Namely, one now sets $s_2=-1+p$ and $s_2=-2p$ for a small enough $p>0$. Then by applying the same procedure that led to \eqref{sample} we obtain $\nu_{\cdot} \in L^{\infty}([0,T];H^{-2p}(\T;\R))$ which a fortiori implies $\nu_{\cdot} \in C([0,T];H^{-2p}(\T;\R))$. If one repeats the argument with e.g. $s_2=0$ and $s_1=-\frac{1}{2}+p$ for a small enough $p>0$ then obtains $\nu_{\cdot} \in C([0,T];L^2(\T;\R))$ which shows that $\nu_t$ admits a density in $L^2(\T;\R)$. The $H^1(\T;\R)$-regularity of $\nu_t$ can be now obtained easily. This concludes the proof. 
\end{proof}

%%%%%%%%%%%%%%%%%5
%%%%%%%%%%%%%%%%%%%%5torna 3

\begin{proof}[Proof of Proposition \ref{prop:law satisfies PDE}]
Let us begin with proving the well-posedness of the  PDE \eqref{PDEmuMKE}. The existence of a weak solution to the PDE \eqref{PDEmuMKE} follows from the well-posedness of the SDE \eqref{sistemaMKE}. Indeed, by a straightforward application of It\^o's formula to the process $t \to \Psi(X_t^{\e,M,\ka},A_t^{\e,M,\ka})$ where $\Psi$ is a smooth and compactly supported function on $\T \times \R$ (i.e. $\Psi \in C_c^{\infty}(\T \times \R; \R)$) one can readily see that the joint distribution  $\mu_t^{\e,M,\ka}:=\cL(X_t^{\e,M,\ka},A_t^{\e,M,\ka})$ is a weak measure-valued solution to \eqref{PDEmuMKE}. Therefore, the uniqueness of the solution to \eqref{PDEmuMKE} (in the space of probability measures $\cP_2(\T \times \R)$) remains to be shown. This can be done by tracing the steps of \cite[Theorem 5.4]{coghi2019stochastic} (let us note that the result in the scenario of \cite[Theorem 5.4]{coghi2019stochastic} is proven for SPDEs and, therefore, it can be a fortiori applied to PDEs). 

The coefficients of \eqref{PDEmuMKE} do not satisfy \cite[Assumptions 5.1]{coghi2019stochastic}, which are required to prove \cite[Theorem 5.4]{coghi2019stochastic}.
%By looking at \cite[Assumptions 5.1]{coghi2019stochastic} required to prove \cite[Theorem 5.4]{coghi2019stochastic} one may argue that the coefficients of \eqref{PDEmuMKE} do not satisfy those conditions and, therefore, \cite[Theorem 5.4]{coghi2019stochastic} cannot be applied.

However, upon inspection of \cite[Theorem 5.4]{coghi2019stochastic} one can see that the proof does not hinge on the coefficients $a,b,\sigma$ in \cite[Theorem 5.4]{coghi2019stochastic} to satisfy \cite[Assumptions 5.1]{coghi2019stochastic} but rather on %it soon realizes that it is not important the conditions satisfied by the coefficients $a,b$ and $\sigma$ in \cite[Theorem 5.4]{coghi2019stochastic} but rather the conditions satisfied (see \cite[Assumptions 4.1]{coghi2019stochastic}) by 
their linearised versions $\bar{a},\bar b$ and $\bar \sigma$ to satisfy \cite[Assumptions 4.1]{coghi2019stochastic}.
%in \cite[Theorem 5.4]{coghi2019stochastic} of $a,b$ and $\sigma$, respectively.

%In view of that, we obtain that independently from the assumptions on the coefficients $a,b,\sigma$ in \cite[Section 5]{coghi2019stochastic} of the non-linear SPDE
That is, the proof \cite[Theorem 5.4]{coghi2019stochastic} (which takes care of the uniqueness of the solution) works as soon as the linearised coefficients $\bar a, \bar b$ and $\bar \sigma$ in \cite[Theorem 5.4]{coghi2019stochastic} (of $a,b$ and $\sigma$, respectively) satisfy \cite[Assumptions 4.1]{coghi2019stochastic}.

Hence, to obtain the uniqueness of the solution to \eqref{PDEmuMKE} it suffices to identify the linearised coefficients of the PDE \eqref{PDEmuMKE} and verify that \cite[Assumptions 4.1]{coghi2019stochastic} are satisfied.

To avoid notation conflict we denote by $\bar D, \bar B$ and $ \bar \Sigma $ what in \cite[Theorem 5.4]{coghi2019stochastic} is denoted by $\bar a, \bar b$ and $\bar \sigma$. In our case, the linearised versions of the coefficients of \eqref{PDEmuMKE} are given by 
\begin{align*}
    \Bar{D}^{\e,M,\ka}(t,x,a) & :=\begin{bmatrix}
1 & 0 \\
0 & 0 
\end{bmatrix},\qquad \text{$t \in [0,T]$, $(x,a) \in \T \times \R$,}\\
    \bar{B}^{\e,M,\ka}(t,x,a)& :=  \begin{bmatrix}
V^{'}(x)+\Gamma_M(x,\mu_t^{\e,M,\ka}) \\
- \frac{q(x)}{\Xi_{\e}(x,\mu_t^{\e,M,\ka})}\pa_tY_t^{\ka}  
\end{bmatrix},\qquad \text{$t \in [0,T]$, $(x,a) \in \T \times \R$,}\\
\bar \Sigma^{\e,M,\ka} & \equiv 0.
\end{align*}
From \eqref{ineq gammma} and \eqref{ineq beta} applied to $\mu=\nu=\mu_t^{\e,M,\ka}$ one can now readily see that the above coefficients satisfy \cite[Assumptions 4.1]{coghi2019stochastic}, and the uniqueness of the solution to the PDE \eqref{PDEmuMKE} follows. 

In the case of the PDE \eqref{PDEmuMK} and the PDE \eqref{PDEmuK} we still make use of the results obtained in \cite[Section 4 and Section 5]{coghi2019stochastic} but in a different way.

Let us begin with pointing out that contrary to the case of the PDE \eqref{PDEmuMKE} where we are dealing with a closed equation for $\mu_t^{\e,M,\ka}$, in the instance of the PDE \eqref{PDEmuMK} one works with the system of two PDEs \eqref{PDEmuMK}-\eqref{zetaPDEMK} and proceed as follows. In the first place, we consider the iterative system\footnote{In system \eqref{iterated system muMK},  $n \in \N$ denotes the iteration parameter, there is no $\varepsilon$ dependence in such a system; that is, $\mu_t^{n, M, \ka}$  is not to be confused with $\mu_t^{\e,M,\ka}$ as it is not $\mu_t^{\e,M,\ka}$ when $\ep=n$. The sequence $\mu_t^{n, M, \ka}$ is constructed to study \eqref{PDEmuMK}-\eqref{zetaPDEMK}, not \eqref{PDEmuMKE}.}
\begin{equation}\label{iterated system muMK}
    \begin{dcases}
        & \pa_t\mu_t^{n,M,\ka}=\pa_{xx}\mu_t^{n,M,\ka}+\pa_{x}\left [ \left(V^{'}+\Gamma_M(x,\mu_t^{n-1,M,\ka})\right)\mu_t^{n,M,\ka} \right ]-\frac{q}{\zeta_t^{n,M,\ka}}\pa_a\mu_t^{n,M,\ka}\pa_tY_t^{\ka} \\
        & \pa_t\zeta_t^{n,M,\ka}=\pa_{xx}\zeta_t^{n,M,\ka}+\pa_{x}\left [ \left (V^{'}+\Gamma_M(x,\mu_t^{n-1,M,\ka})\right)\zeta_t^{n,M,\ka} \right ]\\
        & \mu^{n,M,\ka}|_{t=0}=\mu_0,\quad \zeta^{n,M,\ka}|_{t=0}=\zeta_0
    \end{dcases}
\end{equation}
with starting condition $\mu_t^{0,M,\ka}=\mu_0$, $t \in [0,T]$. The linear iterative system \eqref{iterated system muMK} is well-posed once the term $\mu_t^{n-1,M,\ka}$ from the previous iteration is given: one takes $\mu_t^{n-1,M,\ka}$ and following the steps in \cite[Theorem 3, Theorem 4, Theorem 5  and Theorem 6 in Chapter 7]{evans2022partial} solves the second PDE in \eqref{iterated system muMK}. This ensures the existence and uniqueness of a classical solution $\zeta_t^{n,M,\ka}$ to the second PDE in \eqref{iterated system muMK}. After that, one takes $\zeta_t^{n,M,\ka}$ and solves the system of SDEs \eqref{sistema iterativo sdes}. The law of the solution (i.e. the pair $(X_t^{n,M,\ka},A_t^{n,M,\ka})$) to \eqref{sistema iterativo sdes} provides a weak measure-valued solution for the first PDE in \eqref{iterated system muMK}. Indeed, if one applies It\^o's formula to the stochastic process $t \to \Psi(X_t^{n,M,\ka},A_t^{n,M,\ka})$ with $\Psi \in C_{c}^{\infty}(\T\times\R;\R)$ and then takes the expectation; the existence of a solution $\mu_t^{n,M,\ka}$ follows. 
As for the uniqueness of such $\mu_t^{n,M,\ka}$, this is a consequence of \cite[Theorem 4.8]{coghi2019stochastic}. To apply this result, from \cite[Theorem 4.8]{coghi2019stochastic} we notice that the coefficients of the PDE solved by $\mu_t^{n,M,\ka}$ in \eqref{iterated system muMK} need to satisfy \cite[Assumptions 4.1]{coghi2019stochastic}. In this instance, the coefficients are given by
\begin{align*}
    \Bar{D}_n^{M,\ka}(t,x,a) & :=\begin{bmatrix}
1 & 0 \\
0 & 0 
\end{bmatrix},\qquad \text{$t \in [0,T]$, $(x,a) \in \T \times \R$, $n \in \N$}\\
    \bar{B}_n^{M,\ka}(t,x,a)& :=  \begin{bmatrix}
V^{'}(x)+\Gamma_M(x,\mu_t^{n-1,M,\ka}) \\
- \frac{q(x)}{\zeta_t^{n,M,\ka}(x)}  \pa_tY_t^{\ka}
\end{bmatrix},\qquad \text{$t \in [0,T]$, $(x,a) \in \T \times \R$, $n \in \N$}\\
\bar \Sigma_n^{M,\ka} & \equiv 0,
\end{align*}
At this stage, if one now takes into account $(i)$ and $(ii)$ in Lemma \ref{stime} applied to $\zeta_t^{n,M,\ka}$ and \eqref{ineq gammma} applied to $\mu=\nu=\mu_t^{n-1,M,\ka}$ it follows that for any given $n \in \N$ the coefficients $\bar D_n^{M,\ka}, \bar B_n^{M,\ka}$ and $\bar \Sigma_n^{M,\ka}$ satisfy \cite[Assumptions 4.1]{coghi2019stochastic} and, therefore, the uniqueness of $\mu_t^{n,M,\ka}$ - weak measure valued solution to the first PDE in \eqref{iterated system muMK}- is obtained from \cite[Theorem 4.8]{coghi2019stochastic}. 

As a byproduct of this we have also obtained that $\mu_t^{n,M,\ka}=\cL(X_t^n,A_t^n)$ where $(X_t^n,A_t^n)$ is the solution to \eqref{sistema iterativo sdes}.

Once this is in place, if we follow the calculations that led to \eqref{contraction W2 mu^n}-\eqref{contraction zeta^n} in the proof of Lemma \ref{wellpossistemaMK} one can see that the same inequalities holds for $\Wt^2(\mu^{n,M,\ka},\mu^{n-1,M,\ka})$ and $\sup \limits_{t \in [0,T]}|\zeta_t^{n,M,\ka}-\zeta_t^{n-1,M,\ka}|_{H^{1}(\T;\R)}^2$. Hence, the solution $(\mu_t^{n,M,\ka},\zeta_t^{n,M,\ka})$ induces a contraction in the product space $\cP_2(C([0,T];\T\times \R)) \times C([0,T];H^1(\T;\R))$; this a fortiori implies that the solution $(\mu_t^{M,\ka},\zeta_t^{M,\ka})$ to \eqref{PDEmuMK}-\eqref{zetaPDEMK} obtained beforehand is unique.

%As a result of that, since the PDE \eqref{zetaPDEMK} solved by $\zeta_t^{M,\ka}$ is well-posed (\cite[cfr. Chapter 7]{evans2022partial}) and its coefficients depend on $\mu_t^{M,\ka}$ (which is uniquely determined); the uniqueness of $\zeta_t^{M,\ka}$ follows.

The existence and uniqueness of a weak-solution $\mu_t^{\ka}$ to the PDE \eqref{PDEmuK} can be carried out in the same way of the PDE \eqref{PDEmuMK} with $\Gamma_M$ being replaced by $\Gamma$ (defined in \eqref{gamma senza M}) so we don't repeat it. The proof is thus concluded. 
%This comes from the fact that while in \eqref{PDEmuMKE} the function $\Xi_{\e}(x,\mu_t^{\e,M,\ka})$ is clearly bounded from below by a positive constant; in the PDE \eqref{PDEmuMK} is the measure $\zea$  

%Indeed, we argue that independently of how one proves the existence of a solution to the the coefficients of the  To be precise, In particular, from the conditions in \cite[Assumptions 5.1.]{coghi2019stochastic} imposed on the coefficients only \cite[$(i)$, $(ii)$ and $(iii)$ in Assumptions 5.1.]{coghi2019stochastic} are needed to prove the uniqueness. The last two conditions i.e. \cite[$(iv)$ and $(v)$ in Assumptions 5.1.]{coghi2019stochastic} are required for the existence of solution to the SPDE in exam. 
\end{proof}

\begin{proof}[Proof of Lemma \ref{well-posedness MKE}]
    In what follows we omit dependence of the quantities at hand on the parameters $\e,M$ and $\ka$. Hence, we write $X_t^n$ and $\mu_t^n$ in place of $X_t^{\e,M,\ka,n}$ and $\mu_t^{\e,M,\ka,n}$. However, we do keep track of such parameters and in what follows $C(t,\e,M,\ka)$ denotes a generic constant, depending continuously on $t,\e,M$ and $\ka$, which may change from line to line.\\
    To prove the well-posedness of system \eqref{sistemaMKE} we apply a standard fixed point type of argument. Let $X_0,A_0,\mu_0$ be as in Definition \ref{def_MV} and let us set the following iterative scheme: 
    \begin{equation*}
        \begin{dcases}
            & X_t^n = X_0 - \int_0^t \left( V^{'}(X_s^n)+\Gamma_M(X_s^n,\mu_s^{n-1}) \right ) \,ds+\sqrt{2}\beta_t\\
            & A_t^n = A_0 + \int_0^t  \Xi_{\e}(X_s^n,\mu_s^{n-1})\,d\mY_s^{\kappa}\\
            & \mu_t^n=\cL(X_t^n,A_t^n),
        \end{dcases}
    \end{equation*}
    where $\mu^n|_{t=0}=\mu_0$ with starting condition $\mu_t^0=\mu_0$, $t \in [0,T]$.
    Using now \eqref{ineq gammma} we get  
    \begin{align*}
        \mE \left( \sup \limits_{s \in [0,t]} |X_s^n-X_s^{n-1}|^2 \right ) & \leq C(t,M) \left [ \int_0^t \mE \left( \sup \limits_{r \in [0,s]} |X_r^n-X_r^{n-1}|^2 \right )ds \right. \\
        & +  \left. \int_0^t  \sup \limits_{r \in [0,s]}\cW_2^2(\mu_r^{n-1},\mu_r^{n-2})\,ds \right  ],
    \end{align*}
    so that Gronwall's lemma implies 
    \begin{align}\label{1}
        & \mE \left( \sup \limits_{s \in [0,t]} |X_s^n-X_s^{n-1}|^2 \right ) \leq C(t,M) \int_0^t  \sup \limits_{r \in [0,s]}\cW_2^2(\mu_r^{n-1},\mu_r^{n-2})\,ds.
    \end{align}
    Similarly, from Jensen's inequality and \eqref{ineq beta} we have  
    \begin{align*}\notag
        \mE \left( \sup \limits_{s \in [0,t]} |A_s^n-A_s^{n-1}|^2 \right )  & \leq C(t,\e) \left [\int_0^t \,  \mE \left( \sup \limits_{r \in [0,s]} \left | X_r^n-X_r^{n-1} \right |^2 \right)|\pa_s\mY_s^{\ka}|^2\,ds  \right. \\ \notag
       & + \left.\int_0^t\sup \limits_{r \in [0,s]}\cW_2^2(\mu_r^{n-1},\mu_r^{n-2})\,|\pa_s\mY_s^{\ka}|^2ds \right ]\\ 
       & \leq C(t,\ka,\e) \left[ \int_0^t \,  \mE \left( \sup \limits_{r \in [0,s]} \left | X_r^n-X_r^{n-1} \right |^2 \right)\,ds + \int_0^t\sup \limits_{r \in [0,s]}\cW_2^2(\mu_r^{n-1},\mu_r^{n-2})ds \right ] 
    \end{align*}
    for $t \in [0,T]$ and $n \in \N$.
    By combining \eqref{1} and the above inequality we obtain 
    \begin{align*}
        & \mE \left( \sup \limits_{s \in [0,t]} |X_s^n-X_s^{n-1}|^2 + \sup \limits_{s \in [0,t]} |A_s^n-A_s^{n-1}|^2 \right ) \leq  C(t,M,\ka,\e) \int_0^t\sup \limits_{r \in [0,s]}\cW_2^2(\mu_r^{n-1},\mu_r^{n-2})ds.
    \end{align*}
    Hence, from the definition 2-Wasserstein with $\cX=C_t$ and inequality \eqref{W2 leq W2t} we deduce 
    \begin{align*}
        \Wt(\mu_{\cdot}^{n},\mu_{\cdot}^{n-1}) \leq C(t,M,\ka,\e) \int_0^t \Ws(\mu_{\cdot}^{n-1},\mu_{\cdot}^{n-2})\,ds 
    \end{align*}
    From this, the statement readily follows.
\end{proof}
\begin{proof}[Proof of Lemma \ref{wellpossistemaMK}]
In what follows we omit dependence of the quantities on the parameters $M$ and $\ka$. Hence, we write $X_t^n$, $\mu_t^n$ and $\zeta_t^{n}$ in place of $X_t^{M,\ka,n}$, $\mu_t^{M,\ka,n}$ and $\zeta_t^{M,\ka,n}$. However, we do keep track of such parameters and in what follows $C(t,M,\ka)$ denotes a generic constant, depending continuously on $t,M$ and $\ka$, which may change from line to line.\\
To prove the well-posedness of system \eqref{sistemaMK} we apply again a standard fixed point type of argument. Let $X_0,A_0,\mu_0$ and $\zeta_0$ be as in Definition \ref{def_MV} and
%$\mu_t^0=\mu_0 \in \cP_2(\T \times \R)$, $t \in [0,T]$, be the initial input. The iterative scheme is the following: 
let us consider the following iterative scheme, 
\begin{equation}\label{sistema iterativo sdes}
    \begin{dcases}
        & X_t^n=X_0-\int_0^t V^{'}(X_s^n)+\Gamma_M(X_s^n,\mu_s^{n-1})\,ds+\sqrt{2}\beta_t,\\
        & A_t^n=A_0+\int_0^t \frac{q(X_s^n)}{\zeta_s^n(X_s^n)}\,d\mY_s^{\ka},\\
        & \pa_t\zeta^n=\pa_{xx}\zeta^n+\pa_x \left [(V^{'}+\Gamma_M(x,\mu_t^{n-1})\zeta^n\right],\\
        & \zeta^n|_{t=0}=\zeta_0,
    \end{dcases}
\end{equation}
where $\mu^n|_{t=0}=\mu_0$ for every $n \in \N$ with starting condition $\mu_t^0=\mu_0$, $t \in [0,T]$.
%In what follows $C(t,M,\ka)$ denotes a constant depending continuously on $t$, $M$ and $\ka$ which may change from line to line.
%Let us now estimate $\left |X_t^n-X_t^{n-1} \right |$ and $\left | A_t^{n}-A_t^{n-1} \right |$. 
From \eqref{ineq gammma} we obtain 
\begin{align*}
    \mE \left( \sup \limits_{s \in [0,t]}|X_s^n-X_s^{n-1}|^2 \!\! \right )\! \leq \! C(t,M) \! \left [ \!\!\!\,\,\int_0^t \mE \left( \sup \limits_{r \in [0,s]}|X_r^n-X_r^{n-1} |^2\right)\!ds+\!\!\!\int_0^t \sup \limits_{r \in [0,s]} \cW_2^2(\mu_r^{n-1},\mu_r^{n-2})ds \right ],
\end{align*}
so that from Gronwall's lemma we then have
\begin{equation*}%\label{estimate X_t^n-X_t^n-1}
    \mE \left( \sup \limits_{s \in [0,t]}|X_s^n-X_s^{n-1}|^2 \right ) \leq C(t,M) \int_0^t \sup \limits_{r \in [0,s]}\cW_2^2(\mu_r^{n-1},\mu_r^{n-2})\,ds ,\qquad t \in [0,T].
\end{equation*}
As for the difference $|A_t^{n}-A_t^{n-1}|$ we follow the steps of \eqref{ultimate estimate} in Proposition \ref{prop_e_before_M}. Indeed, from $(i)$ in Lemma \ref{stime} applied to $\zeta_t^n$ with $\mu_t^{n-1}$ we obtain (recall $\eta:=\min \limits_{\T} \zeta_0(x)>0$)
\begin{equation}\label{zeta^n well-posedness}
    \zeta_t^n(x) \geq \eta \,e^{-t\mathcal{A}(M)},\qquad \text{for every $x \in \T$, $t \in [0,T]$ and $n \in \N$}
\end{equation}
where $\mathcal{A}(M)$ is defined by \eqref{esponente stima lower bound}. Let us point out that from \eqref{zeta^n well-posedness} it follows immediately that for any given $M,\ka \in \N$ the function $\zeta_t^n$ is bounded from below uniformly in $n \in \N$ and, therefore, also the limit of the sequence will enjoy the same property due to the continuity w.r.t. $x \in \T$ of the $\zeta_t^n$'s. 

Once this observation has been made, the rest of the proof is completely analogous to the proof of Lemma \ref{well-posedness MKE} and makes use of calculations similar to those in \eqref{ultimate estimate}, this time using the bounds \eqref{lower bound zeta mu}, \eqref{bound H2-norm zeta mu} and \eqref{bound differenxe H1 norm zeta mu zeta nu}, so we omit it. From this we obtain
%and $(iii)$ in Lemma \ref{stime} applied to $\zeta_t^n$ and $\zeta_t^{n-1}$ with $\mu_t^{n-1}$ and $\mu_t^{n-2}$, respectively, we obtain 
%\begin{align*}
%    & \left|\frac{q(X_t^{n})}{\zeta_t^{n}(X_t^{n})}\,-\,\frac{q(X_t^{n-1})}{\zeta_t^{n-1}
%    (X_t^{n-1})}\right|^2\, =\,\left | \frac{q(X_t^{n})\,\zeta_t^{n-1}(X_t^{n-1})-q(X_t^{n-1}) \,\zeta_t^{n}(X_t^{n})}{\zeta_t^{n}(X_t^{n})\,\zeta_t^{n-1}(X_t^{n-1})} \right |^2 \\
%    & \leq^{\eqref{lower bound zeta mu}} C(t,M,\ka,\eta) \left [ |\zeta_t^{n-1}|_{L^{\infty}(\T;\R)}^2  \left | X_t^{n}-X_t^{n-1} \right |^2 \right. \\
%    &  \left. + \left |\pa_x \zeta_t^{n-1} \right |_{L^{\infty}(\T;\R)}^2 \left |X_t^{n}-X_t^{n-1} \right |^2 + \left |\zeta_{t}^{n-1}-\zeta_t^{n} \right |_{L^{\infty}(\T;\R)}^2 \right ] \\
%    & \leq^{\eqref{bound H2-norm zeta mu},\eqref{bound differenxe H1 norm zeta mu zeta nu}} C(t,M,\ka,\eta) \left [ \left | X_t^{n}-X_t^{n-1} \right |^2 + \int_0^t \cW_2(\mu_s^n,\mu_s^{n-1})\,ds \right ].
%\end{align*}
%Hence,
%\begin{align*}
%     \mE \left( \sup \limits_{s \in [0,t]} |A_t^n-A_t^{n-1}|^2\right ) & \leq C(t,M,\ka,\eta) \left [\,\,\,\int_0^t \mE \left( \sup \limits_{r \in [0,s]}|X_r^n-X_r^{n-1}|^2\right) \,ds \right. \\
%     & \left. +\int_0^t \sup \limits_{r \in [0,s]}\cW_2^2(\mu_r^{n-1},\mu_r^{n-2})\,ds \right ] \\
%\end{align*}
%By combining the above inequality along with \eqref{estimate X_t^n-X_t^n-1}, recalling the definition of $\Wt$ and using \eqref{W2 leq W2t} we obtain 

\begin{align}\label{contraction W2 mu^n}
     \Wt^2(\mu_{\cdot}^{n},\mu_{\cdot}^{n-1}) & \leq C(t,M,\ka) \int_0^t \Ws^2(\mu_{\cdot}^{n-1},\mu_{\cdot}^{n-2})\,ds,  \\ 
        \sup \limits_{s \in [0,t]}|\zeta_s^{n}-\zeta_s^{n-1}|_{H^1(\T;\R)}^2 & \leq C(t,M,\ka) \int_0^t \Ws^2(\mu_{\cdot}^{n-1},\mu_{\cdot}^{n-2})\,ds \label{contraction zeta^n},
\end{align}
and this concludes the proof.
\end{proof}
\begin{proof}[Proof of Lemma \ref{well-posedness kappa}]
We use exactly the same iterative scheme implemented in the proof of Lemma \ref{wellpossistemaMK} but the function $\Gamma_M$ in \eqref{def_GammaM} is replaced by $\Gamma$ defined in \eqref{gamma senza M} and, moreover, also in this proof a lower bound for $\zeta_t^n$ uniform in $n \in \N$ is needed (we recall that $n \in \N$ is the iteration parameter). This is obtained by following the steps and calculations of Lemma \ref{lemma first moment} and $(i)$ of Lemma \ref{stime}.
\end{proof}

\section{Lower bound for the solution of linear parabolic PDEs}
In this section we provide a lower bound for parabolic PDEs of the following form, 
\begin{equation}\label{linear w}
    \begin{dcases}
        & \pa_t w=\pa_{xx}w+\pa_x [f w],\,\,  \quad x \in \T,\,\, t \in (0,T]\\
    & w|_{t=0}=w_0,\,\, \quad  x \in \T \,, 
    \end{dcases}
\end{equation}
for the unknown $w = w_t(x):\R_+ \times \T \rightarrow \R$,  where $f \in C\left([0,T];C^{\infty}(\T;\R)\right)$. We are pretty sure that the bounds we produce here must exists in the literature,  but we were not able to find them in the exact form in which we need them, so we present a proof, following  the strategy of \cite[Theorem 10 in Chapter 7]{evans2022partial}. The main result we need is Theorem \ref{lower bound w}, the other lemmata below are preliminary to that result. 
\begin{lemma}\label{zeta non negativa}
Let $w_t$ be the solution to \eqref{linear w} and $w_0 \geq 0$ then $w_t$ is non-negative.
\end{lemma}
\begin{proof}
The proof can be readily adapted from \cite[Lemma 2.1 and  Corollary 2.2]{chazelle2017well}, so we don't repeat it. 
\end{proof}
\begin{lemma}\label{lemma B2}
Let $w_t$ be the solution to \eqref{linear w} and let $w_0 \in C^{\infty}(\T;\R)$;  
then, for any given $k \in \N$ the solution $w_t$ converges to $w_0$ in $H^k(\T;\R)$ as $t \downarrow 0$, i.e.
\begin{equation*}
    \lim \limits_{t \downarrow 0} |w_t-w_0|_{H^k(\T;\R)}=0.
\end{equation*}
\end{lemma}
\begin{note}\label{continuity}
It is known that for any given $w_0 \in L^2(\T;\R)$ the solution $w_t$ to \eqref{linear w} is smooth for positive times (cfr. \cite[Theorem 5 and Theorem 6 in Chapter 7]{evans2022partial}). With the above lemma and Sobolev's embedding theorem we extended the smoothness up to the initial time $t=0$, in other words, $w \in C^{\infty}([0,T] \times \T)$.
\end{note}
\begin{proof}[Proof of Lemma \ref{lemma B2}]
We proceed via an induction argument and we show that for any given $k \in \N$ the following two properties hold:
\begin{equation}\label{k in N}
    \begin{dcases}
        & |w_t|_{H^k(\T;\R)}^2 \lesssim_{t,f} |w_0|_{H^k(\T;\R)}^2,\,\, t \in [0,T]\\
        & \lim \limits_{t \downarrow 0} |w_t-w_0|_{H^k(\T;\R)}=0 \,. 
    \end{dcases}
\end{equation}
If we set $k=0$ we have to prove that \eqref{k in N} holds with $k=0$.
A straightforward calculation shows that for $t\in (0,T]$
\begin{align*}
    \frac{d}{dt}|w_t|_{L^2(\T;\R)}^2+\frac{1}{2}|\pa_xw_t|_{L^2(\T;\R)}^2 \lesssim \int_{\T} f^2w_t^2 \,dx \lesssim_f |w_t|_{L^2(\T;\R)}^2
\end{align*}
which implies from Gronwall's lemma
\begin{equation*}
   |w_t|_{L^2(\T;\R)}^2 \lesssim_{t,f} |w_0|_{L^2(\T;\R)}^2,\,\, t \in (0,T].
\end{equation*}
To prove the limit as $t \downarrow 0$ a similar calculation is performed. Indeed, by differentiating the function $t  \to |w_t-w_0|_{L^2(\T;\R)}^2$ and after some calculations we obtain, for every $t \in (0,T]$, 
\begin{align*}
 \frac{d}{dt}|w_t-w_0|_{L^2(\T;\R)}^2 & \lesssim \int_{\T} w_t\pa_{xx}w_0\,dx + \int_{\T}fw_t\pa_xw_0\,dx +\int_{\T} f^2w_t^2 \,dx \\
& \lesssim_f |\pa_xw_0|_{H^1(\T;\R)}|w_t|_{L^2(\T;\R)}+|w_t|_{L^2(\T;\R)}^2 \, . 
\end{align*}
Integrating with respect to time yields
\begin{align*}
|w_t-w_0|_{L^2(\T;\R)}^2 \lesssim_f |\pa_xw_0|_{H^1(\T;\R)} \int_0^t |w_s|_{L^2(\T;\R)} \,ds +\int_0^t|w_s|_{L^2(\T;\R)}^2\,ds,\,\,t \in (0,T],
\end{align*}
which implies 
\begin{equation*}
    \limsup \limits_{t \downarrow 0} |w_t-w_0|_{L^2(\T;\R)}^2 \leq 0.
\end{equation*}
Hence, \eqref{k in N} for $k=0$ is thus proven. With the base of the mathematical induction argument proven, we are left to show the induction step. That is, if we assume that \eqref{k in N} holds for $k-1$ then we have to show that \eqref{k in N} holds for $k$ as well. 

The upper bound for $|\pa_x^k w_t|_{L^2(\T;\R)}$ readily follows from Gronwall's lemma once one calculates $\frac{d}{dt}|\pa_x^k w_t|_{L^2(\T;\R)}^2$ and uses the bound \eqref{k in N} for $k-1$.

As for the time-continuity in $L^2(\T;\R)$ of $\pa_x^k w_t$ at $t=0$ this follows once one calculates $\frac{d}{dt}|\pa_x^k w- \pa_x^k w_0|_{L^2(\T;\R)}^2$, integrates from $0$ to $t$ as sets $t \downarrow 0$.
%\begin{align*}
%    \frac{d}{dt}|\pa_x^k w_t|_{L^2(\T;\R)}^2+\frac{1}{2}|\pa_x^{k+1}w_t|_{L^2(\T;\R)}^2 \lesssim \int_{\T} |\pa_x^k(fw_t)|^2 \,dx \lesssim_f |\pa_x^k w|_{L^2(\T;\R)}^2+|w_t|_{H^{k-1}(\T;\R)}^2
%\end{align*}
%which again implies from Gronwall's lemma 
%\begin{align*}
%    |\pa_x^k w_t|_{L^2(\T;\R)}^2 \lesssim_{t,f} |\pa_x^k w_0|_{L^2(\T;\R)}^2+\int_0^t |w_s|_{H^{k-1}(\T;\R)}^2\,ds,\,\, t \in (0,T].
%\end{align*}
%Hence, from the induction hypothesis we deduce that 
%\begin{align*}
%    |w_t|_{H^k(\T;\R)}^2 \lesssim_{t,f} |w_0|_{H^k(\T;\R)}^2,\,\, t \in (0,T].
%\end{align*}
%Once this is in place, to conclude it suffices to show that 
%\begin{equation} \label{Hk limit}
 %   \lim \limits_{t \downarrow 0} |\pa_x^kw_t-\pa_x^k w_0|_{L^2(\T;\R)}^2=0.
%\end{equation}
%Indeed, by differentiating the function $t  \to |\pa_kw_t-\pa_kw_0|_{L^2(\T;\R)}^2$ and after some calculations we obtain 
%\begin{align*}
%& \frac{d}{dt}|\pa_x^k w- \pa_x^k w_0|_{L^2(\T;\R)}^2 \lesssim \int_{\T} \pa_x^{k+2}w_0\,\pa_x^kw\,dx +\int_{\T} \pa_x^{k+1}w_0 \,\pa_x^k(fw)\,dx+ \int_{\T} |\pa_k(fw)|^2 \,dx \\
%& \lesssim_f |\pa_x^{k+1}w_0|_{H^1(\T;\R)}|w|_{H^k(\T;\R)}+|w|_{H^k(\T;\R)}^2,\,\,t \in (0,T].
%\end{align*}
%Integrating with respect to time yields
%\begin{align*}
%|\pa_x^kw_t- \pa_x^k w_0|_{L^2(\T;\R)}^2  \lesssim_f |\pa_x^{k+1}w_0|_{H^1(\T;\R)}\int_0^t |w_s|_{H^k(\T;\R)}^2\,ds +\int_0^t |w_s|_{H^k(\T;\R)}^2\,ds ,\,\,t \in (0,T],
%\end{align*}
%which gives \eqref{Hk limit} after having set $t \downarrow 0$.
\end{proof}
\begin{thm}[Lower bound for $w$]\label{lower bound w}
Let $w$ be the solution to \eqref{linear w} with initial condition $w_0 > 0$ on $\T$ then the following lower bound holds 
\begin{equation}\label{lower bound w2}
    w(t,x) \geq w_0(x)e^{-t\, a(f)} \, , \qquad \mbox{for every } t \in [0,T],\,\, x \in \T,
\end{equation}
where $a(f)$ is defined as 
\begin{align}\label{a(f)}
    a(f)= \frac{1}{2}\sup \limits_{(t,x) \in [0,T] \times \T} |f|^2+ \sup \limits_{(t,x) \in [0,T] \times \T} |\pa_xf|.
\end{align}  
\end{thm}
\begin{proof}
    We follow closely \cite[Theorem 10 in Chapter 7]{evans2022partial}. \\
\text{Step 1)}
From Lemma \ref{zeta non negativa} we know that $ w = w_t(x) \geq 0$ and therefore, we can consider the function $v= v_t(x)$ defined as  $v= \log (w+\delta)$ for any given $\delta \in (0,1)$. \\
\text{Step 2)}[PDE solved by $v$]
By using the following identities 
\begin{align*}%\label{aplha-1}
    \pa_xv=\frac{\pa_xw}{w+\delta},\,\,\,\,\pa_{xx}v+|\pa_xv|^2=\frac{\pa_{xx}w}{w+\delta} \, ,
\end{align*}
it can be readily seen that $v$ is a solution to the following non-linear PDE
\begin{equation}\label{pde v-delta}
    \pa_tv=\pa_{xx}v+|\pa_xv|^2+f\pa_xv+\pa_xfg_{\delta}(x).
\end{equation}
where $g^{\delta}_t(x)=\frac{w_t(x)}{w_t(x)+\delta}$, for every $\delta>0$. \\
\text{Step 3)}[Differential inequality needed to conclude]
Our aim is to obtain a differential inequality for $v$ of the following type 
\begin{equation}\label{alpha}
    \pa_t v \geq -\alpha,\,\, x \in \T,\,\, t \in [0,T]
\end{equation}
for a suitable constant $\alpha>0$ which depends only on the $L^{\infty} \left ([0,T] \times \T; \R \right)$-norm  of $f$ and its derivatives. Indeed, if the above holds, we obtain
\begin{align*}
    v_t(x)-v_0(x) & \geq - \int_0^1 \alpha \,ds =-t \alpha \, . 
\end{align*}
By first taking the exponential on both sides and then letting $\delta \downarrow 0$, the above  gives \eqref{lower bound w2}.\\
\text{Step 4)}[Auxiliary function]
In order to get \eqref{alpha}  we begin with noting that by \eqref{pde v-delta} and Young's inequality, we have   
\begin{align}\label{aux v}
    \pa_t v \geq \pa_{xx}v + \frac{1}{2}|\pa_xv|^2-\frac{1}{2}|f|^2-|\pa_xf|\left |\frac{w}{w+\delta }\right| \geq \pa_{xx}v-a(f),\,\,
\end{align}
where $a(f)$ is defined as in \eqref{a(f)}. To obtain the above inequality we have used the fact that $g^{\delta}_t(x) \leq 1$, (because $w_t(x) \geq 0$). 
Let us now consider the function $\ell:[0,T] \times \T \to \R$ defined as 
\begin{equation*}
    \ell=v+t\,\theta 
\end{equation*}
where $\theta$ is a positive constant such that $\theta > a(f)$. From Note \ref{continuity} we know that $\ell \in C^{1,2}([0,T]\times\T;\R)$. Therefore, if we assume that $(t_0,x_0) \in (0,T] \times \T$ is a minimun point of $\ell$ then the following holds at $(t_0,x_0)$
\begin{equation*}
    \begin{dcases}
        & \pa_x\ell=\pa_x v=0 \\
        & \pa_{xx} \ell=\pa_{xx} v \geq 0 \\
        & 0 \geq \pa_t\ell=\pa_tv+\theta. 
    \end{dcases}
\end{equation*}
By using \eqref{aux v} we obtain 
\begin{align*}
    0 \geq \pa_t \ell=\pa_t v +\theta \geq \pa_{xx}v+ \theta- a(f) \geq \theta -a(f)
\end{align*}
where the above inequality holds at $(t_0,x_0)$. Since we have chosen $\theta >a(f)$ we get a contradiction. Hence, when $\theta > a(f)$ if $(t_0,x_0)$ is a minimum point then it must lie on the line $\{0\}\times \T$ i.e. $(t_0,x_0) \in \{0\}\times \T$. Therefore, we have that 
\begin{align*}
    v+t\,\theta \geq v_0,\quad  \mbox{for every  } (t,x) \in [0,T] \times \T
\end{align*}
which implies (by taking the exponential on both sides) 
\begin{equation*}
w_t(x)+\delta \geq (w_0(x)+\delta)e^{-t\,\theta},\quad \mbox{for every  }  (t,x) \in [0,T] \times \T,\,\,\,\theta > a(f).
\end{equation*}
By letting $\theta \downarrow a(f)$ and $\delta \downarrow 0$ we obtain the desired result.
\end{proof}

\section{The Young integral}\label{young}
In this section we give a quick introduction to the theory of Young integral. For a deeper introduction we refer the reader to \cite[Chapter 1]{lyons2007differential}. Let us begin with defining what a path with finite $p$-variation is.
From this moment onward the couple $(E,|\,\cdot\,|_E)$ will denote a Banach space.
\begin{defn}
    Let $\mY:[0,T]\to E$ be a trajectory and let $p \geq 1$. We say that $\mY$ has finite $p$-variation if the following holds 
    \begin{equation}\label{definizione p-variation}
        |\mY|_{p,T,E}:=\sup \limits_{\cD \subset [0,T]} \left [\sum \limits_{j=0}^{r-1} |\mY_{t_{i+1}}-\mY_{t_i}|_E^p \right]^{\frac{1}{p}}
    \end{equation}
where the supremum is taken over all partitions of the interval $[0,T]$.
\end{defn}
For each $p \geq 1$, we denote by $V_p([0,T], E)$ or simply $V_{p,T}^E$ denote the subset of $C([0,T];E)$ consisting of those paths which have finite $p$-variation. We endow such a space with the following norm. For each $\mY \in V_{p,T}^E$, we set
\begin{equation*}
    |\mY|_{V_{p,T}^E}:=|\mY|_{p,T,E}+\sup \limits_{t \in [0,T]}|\mY_t|_E
\end{equation*}
To avoid any ambiguity, we call $|\mY|_{p,T}$ the $p$-variation of $\mY$ on $[0,T]$ and
$|\mY|_{V_{p,T}^E}$ the $p$-variation norm of $\mY$ on $[0,T]$. Moreover, the following chain of inclusions holds.
\begin{prop}\cite[Proposition 1.7]{lyons2007differential}
For each $p \geq 1$, the set $V_{p,T}^E$ is a linear subspace of $C([0,T], E)$ on which $|\,\cdot\,|_{V_{p,T}^E}$ is a norm. Moreover, the couple $ \left ( V_{p,T}^E, |\,\cdot\,|_{V_{p,T}^E} \right )$ is a Banach space. Finally, if $1 \leq p \leq q$, then the following inclusions hold and are continuous 
\begin{equation*}
    V_{1,T}^E \subset V_{p,T}^E \subset V_{q,T}^E \subset C([0,T];E)
\end{equation*}
\end{prop}
\begin{lemma}\cite[Lemma 1.13]{lyons2007differential}
Let $p$ and $q$ be such that $1 \leq p < q$. Let $\mX$, $\mY$ be two paths of $V_p^T$. Then the following estimate holds
\begin{equation*}
|\mX - \mY|_{V_{q,T}^E} \leq \left ( 2 \sup \limits_{t \in [0,T]} |\mX_t - \mY_t|_E \right )^{\frac{q-p}{q}}|\mX-\mY|_{p,T,E}^{\frac{p}{q}} +\sup \limits_{t \in [0,T]} |\mX_t-\mY_t|_E 
\end{equation*}
\end{lemma}
\begin{thm}\cite[Theorem 1.16 and Remark 1.17]{lyons2007differential}\label{young integral}
Let $p, q \geq 1$ be two real numbers such that $\frac{1}{p}+\frac{1}{q}>1$
Let $>0$ be a positive real number and let us consider two paths
$\mX \in V_{p,T}^E$ and $\mY \in V_{q,T}^{\R}$. Then, for each $t \in [0, T]$, the limit
\begin{equation}\label{81}
    \int_0^t \mX_s\,d\mY_s:=\lim \limits_{|\cD| \to 0,\,\cD \subset [0,t]}\,\, \sum \limits_{i=0}^{r-1}\mX_{t_i}(\mY_{t_{i+1}}-\mY_{t_i})
\end{equation}
exists. As a function of $t$, this limit belongs to $V_{q,T}^E$ and there exists a
constant $C_{p,q}$ which depends only on $p$ and $q$ such that the following inequality holds:
\begin{equation*}
        \left | \int_0^{\cdot} \mX_s d\mY_s \right |_{E,q,T} \lesssim_{p,q} |\mX|_{V_{p,T}^E}|\mY|_{q,T,\R}
\end{equation*}
Moreover, the following upper bound on the $L^{\infty}$-norm holds 
\begin{equation}\label{sup-norm upper bound}
    \sup \limits_{t \in [0,T]} \left |\,\, \int_0^{t} \mX_s \,d\mY_s \,\,\right |_E\, \lesssim_{p,q} |\mX|_{V_{p,T}^E}|\mY|_{q,T,\R}.
\end{equation}
\end{thm}
In particular, since a path is H\"older continuous with exponent $\alpha$, with $0 < \alpha \leq 1$, then it has finite $\frac{1}{\alpha}$-variation we have the following.
\begin{cor}\label{C.5}
If we let $\mY \in C^{\alpha}([0,T];\R)$ and $\mX \in C^{\beta}([0,T];E)$ then the integral in \eqref{81} is well-defined whenever $\alpha+\beta>1$.
\end{cor}
\begin{proof}[Proof of Lemma \ref{6.2}]
Let us start by noting that from Theorem \ref{young integral} and upper bound \eqref{sup-norm upper bound} we obtain that
\begin{equation*}
     \sup \limits_{t \in [0,T]} \left |\,\, \cI_f(t,x) \,\,\right |\, \lesssim_{\gamma} |f_{\cdot}(x)|_{V_{1,T}^{\R}}|Z|_{\frac{1}{\gamma},T,\R}.
\end{equation*}
We aim to obtain an estimate for the $L^{\infty}([0,T] \times \T;\R)$-norm of $\cI_f$. To this end, we need to estimate the $|\,\cdot\,|_{V_{1,T}^{\R}}$-norm of $f$. By recalling the definition of $V_{1,T}^{\R}$ two terms have to be controlled. Namely, $|\,\cdot\,|_{1,T,\R}$ and the supremum norm with respect to time $t$. For the supremum norm we simply have:
\begin{equation*}
    \sup \limits_{t \in [0,T]} |f_t(x)| \leq |f|_{L^{\infty}([0,T] \times \T;\R)},\,\,\,\,x \in \T.
\end{equation*}
As for the $1$-variation norm we proceed as follows. If we let $\cD$ be a partition of $[0,T]$ then from the Lipschitz continuity of $f$ follows that
\begin{align*}
    \sum \limits_{i=0}^{r-1} |f_{t_{i+1}}(x)-f_{t_i}(x)| \leq \sum \limits_{i=0}^{r-1} |\pa_tf|_{L^{\infty}([0,T] \times \T;\R)}|t_{i+1}-t_i| \leq T \,\,\, |\pa_tf|_{L^{\infty}([0,T] \times \T;\R)},\,\,\,x \in \T. 
\end{align*}
Hence by taking the supremum over all partitions of the interval $[0,T]$ we obtain that
\begin{equation*}
    |f(x)|_{V_{1,T}^{\R}} \leq \max \{T,1\} \left ( |f|_{L^{\infty}([0,T] \times \T;\R)}+|\pa_tf|_{L^{\infty}([0,T] \times \T;\R)} \right ),\,\,\,\, x \in \T.
\end{equation*}
which gives the following upper bound for the $L^{\infty}([0,T] \times \T;\R)$-norm of $\cI_f$
\begin{equation}\label{upper bound}
      \left |\,\, \cI_f \,\,\right |_{L^{\infty}([0,T] \times \T;\R)}\, \lesssim_{\gamma} \max \{T,1\}(|f|_{L^{\infty}([0,T] \times \T;\R)}+|\pa_tf|_{L^{\infty}([0,T] \times \T;\R)} )|Z|_{\frac{1}{\gamma},T,\R}.
\end{equation}
Furthermore, by applying \eqref{upper bound} to $\cI_{\pa_xf}$ we obtain 
\begin{equation}\label{upper bound derivative}
    \left |\,\, \cI_{\pa_xf} \,\,\right |_{L^{\infty}([0,T] \times \T;\R)}\, \lesssim_{\gamma} \max\{T,1\}(|\pa_xf|_{L^{\infty}([0,T] \times \T;\R)}+|\pa_x\pa_{t}f|_{L^{\infty}([0,T] \times \T;\R)} )|Z|_{\frac{1}{\gamma},T,\R}
\end{equation}
By combining \eqref{upper bound} and \eqref{upper bound derivative} we obtain \eqref{kuos 20}.
We are left to prove \eqref{weak derivative}. Let us fix a partition $\cD$ of the interval $[0,T]$ and consider the partial sums. Since $f$ is smooth both in space and time we have that for any given test function $\varphi \in C^{\infty}(\T;\R)$ the following holds 
\begin{align*}
    & \left \langle \sum \limits_{i=0}^{r-1} f_{t_i}(Z_{t_{i+1}}-Z_{t_i}),\pa_x\varphi \right \rangle_{L^2(\T;\R)}=\sum \limits_{i=0}^{r-1}(Z_{t_{i+1}}-Z_{t_i}) \left \langle f_{t_i},\pa_x\varphi \right \rangle_{L^2(\T;\R)}\\
    & =-\sum \limits_{i=0}^{r-1}(Z_{t_{i+1}}-Z_{t_i}) \left \langle \pa_x f_{t_i},\varphi \right \rangle_{L^2(\T;\R)}=-\left \langle \sum \limits_{i=0}^{r-1} \pa_xf_{t_i}(Z_{t_{i+1}}-Z_{t_i}),\varphi \right \rangle_{L^2(\T;\R)}.
\end{align*}
By now letting the mesh size $|\cD| \downarrow 0$, by recalling that
\begin{align*}
   & \sum \limits_{i=0}^{r-1} f_{t_i}(x)(Z_{t_{i+1}}-Z_{t_i}) \xrightarrow[|\cD|\downarrow 0]{} \cI_f(t,x),\,\,\,t \in [0,T],\,\,\,\,x \in \T,\\
   & \sum \limits_{i=0}^{r-1} \pa_xf_{t_i}(x)(Z_{t_{i+1}}-Z_{t_i}) \xrightarrow[|\cD| \downarrow 0]{}\cI_{\pa_xf}(t,x),\,\,\, t \in [0,T],\,\,\,\,x \in \T
\end{align*}
then from \eqref{upper bound}, \eqref{upper bound derivative} and Lebesgue's dominated convergence theorem we obtain
\begin{equation*}
\left \langle \cI_f, \pa_x\varphi \right \rangle_{L^2(\T;\R)}=-\left \langle \cI_{\pa_xf},\varphi \right \rangle_{L^2(\T;\R)}.   
\end{equation*}
which gives \eqref{weak derivative} and the proof is thus concluded.
\end{proof}

\section{Proofs from Section \ref{section_limitN} to Section \ref{infinite noise bm}}\label{proofs}
\begin{proof}[Proof of $(ii)$ and $(iii)$ of Lemma \ref{stime}]\,\\
    \textbf{Proof of (ii).}
Let us begin with the $L^2(\T;\R)$-norm of $\zeta_t^{(\mu)}$: 
\begin{align*}
    & \frac{d}{dt}|\zeta_t^{(\mu)}|_{L^2(\T;\R)}^2 =2 \itt \zeta_t^{(\mu)}\pa_t\zeta_t^{(\mu)}\,dx =2\itt \zeta_t^{(\mu)}\left[ \pa_{xx}\zeta_t^{(\mu)}+\pa_x \left ( (V^{'}+\Gamma_M(x,\mu_t))\zeta_t^{(\mu)}\right) \right ]\,dx\\
    & = -2\itt |\pa_x\zeta_t^{(\mu)}|^2\,dx-2\itt V^{'}\zeta_t^{(\mu)}\pa_x\zeta_t^{(\mu)}\,dx-2\itt \Gamma_M(x,\mu_t) \zeta_t^{(\mu)}\pa_x\zeta_t^{(\mu)}\,dx\\
    &= -2\itt |\pa_x\zeta_t^{(\mu)}|^2\,dx-\itt V^{'}\pa_x \left (\zeta_t^{(\mu)} \right )^2 \,dx-\itt \Gamma_M(x,\mu_t)\pa_x \left (\zeta_t^{(\mu)} \right )^2\,dx\\
    &=-2\itt |\pa_x\zeta_t^{(\mu)}|^2\,dx+\itt V^{''}|\zeta_t^{(\mu)}|^2\,dx+\itt |\zeta_t^{(\mu)}|^2 \frac{\pa}{\pa x} \Gamma_M(x,\mu_t)\,dx \\
    & \leq^{\eqref{upper bound bouno}} |V^{''}|_{L^{\infty}(\T;\R)} |\zeta_t^{(\mu)}|_{L^2(\T;\R)}^2+M|F^{''}|_{L^{\infty}(\T;\R)}|\zeta_t^{(\mu)}|_{L^2(\T;\R)}^2
\end{align*}
so that from Gronwall's lemma we deduce
 \begin{equation}\label{stima norma L2 zetamu}
    |\zeta_t^{(\mu)}|_{L^2(\T;\R)}^2 \leq |\zeta_0|_{L^2(\T;\R)}^2e^{(|V^{''}|_{L^{\infty}(\T;\R)}+M|F^{''}|_{L^{\infty}(\T;\R)})t},\,\,t \in [0,T].
 \end{equation}
 After having done so, with a similar calculation one can see that 
\begin{align*}
     \frac{d}{dt}|\pa_x\zeta_t^{(\mu)}|_{L^2(\T;\R)}^2 & \leq \left |\frac{d^4V}{dx^4}\right |_{L^{\infty}(\T;\R)}  |\zeta_t^{(\mu)}|_{L^2(\T;\R)}^2+3|V^{''}|_{L^{\infty}(\T;\R)}|\pa_x\zeta_t^{(\mu)}|_{L^2(\T;\R)}^2\\
    & M\left |\frac{d^4F}{dx^4} \right |_{L^{\infty}(\T;\R)}  |\zeta_t^{(\mu)}|_{L^2(\T;\R)}^2+3M\left |F^{''}\right |_{L^{\infty}(\T;\R)}|\pa_x\zeta_t^{(\mu)}|_{L^2(\T;\R)}^2 \\
    & \leq \left (\left |\frac{d^4V}{dx^4}\right |_{L^{\infty}(\T;\R)} + M\left |\frac{d^4F}{dx^4} \right |_{L^{\infty}(\T;\R)} \right )|\zeta_t^{(\mu)}|_{L^2(\T;\R)}^2\\
    & +3 \left( |V^{''}|_{L^{\infty}(\T;\R)}+M\left |F^{''}\right |_{L^{\infty}(\T;\R)} \right )|\pa_x\zeta_t^{(\mu)}|_{L^2(\T;\R)}^2
\end{align*}
Now, by estimate \eqref{stima norma L2 zetamu} and again Gronwall's lemma, we obtain 
\begin{align}\label{stima buona pax zeta}
    |\pa_x\zeta_t^{(\mu)}|_{L^2(\T;\R)}^2 & \leq |\pa_x \zeta_0|_{L^2(\T;\R)}^2 e^{(3|V^{''}|_{L^{\infty}(\T;\R)}+3M|F^{''}|_{L^{\infty}(\T;\R)})t}\\ \notag
    & +\left (\left |\frac{d^4V}{dx^4}\right |_{L^{\infty}(\T;\R)} + M\left |\frac{d^4F}{dx^4} \right |_{L^{\infty}(\T;\R)} \right )|\zeta_0|_{L^2(\T;\R)}^2e^{(|V^{''}|_{L^{\infty}(\T;\R)}+M|F^{''}|_{L^{\infty}(\T;\R)})t}
\end{align}
Similarly, for the second derivative
\begin{align*}
     & \frac{d}{dt}|\pa_{xx}\zeta_t^{(\mu)}|_{L^2(\T;\R)}^2 \leq |\pa_{xx}(V^{'}\zeta_t^{(\mu)})|_{L^2(\T;\R)}^2+|\pa_{xx}(\Gamma_M(\cdot,\mu_t)\zeta_t^{(\mu)})|_{L^2(\T;\R)}^2 \\
     & \leq 3 (|V^{'}|_{L^{\infty}(\T;\R)}^2+M^2|F^{'}|_{L^{\infty}(\T;\R)}^2)|\pa_{xx} \zeta_t^{(\mu)}|_{L^2(\T;\R)}^2\\
     & + 3 (|V^{''}|_{L^{\infty}(\T;\R)}^2+M^2|F^{''}|_{L^{\infty}(\T;\R)}^2)|\pa_x \zeta_t^{(\mu)}|_{L^2(\T;\R)}^2 \\
     & + 3 (|V^{'''}|_{L^{\infty}(\T;\R)}^2+M^2|F^{'''}|_{L^{\infty}(\T;\R)}^2)|\zeta_t^{(\mu)}|_{L^2(\T;\R)}^2
\end{align*}
Using \eqref{stima norma L2 zetamu}, \eqref{stima buona pax zeta}, the above estimate and by again applying Gronwall's lemma we obtain $(ii)$.\\
\textbf{Proof of (iii).}
From a straightforward calculation we have
\begin{align*}
    & \frac{d}{dt}|\zeta_t^{(\mu)}-\zeta_t^{(\nu)}|_{L^2(\T;\R)}^2 \leq |V^{''}|_{L^{\infty}(\T;\R)}^2|\zeta_t^{(\mu)}-\zeta_t^{(\nu)}|_{L^2(\T;\R)}^2+\int_{\T} \left | \Gamma_M(x,\mu_t)\zeta_t^{(\mu)}-\Gamma_M(x,\nu_t)\zeta_t^{(\nu)}\right |^2\,dx\\
    & \leq |V^{''}|_{L^{\infty}(\T;\R)}^2|\zeta_t^{(\mu)}-\zeta_t^{(\nu)}|_{L^2(\T;\R)}^2 \!\!+2\!\!\int_{\T}\left|\Gamma_M(x,\nu_t)(\zeta_t^{(\nu)}-\zeta_t^{(\mu)})\right |^2\!\!\!+\!\left | (\Gamma_M(x,\mu_t)-\Gamma_M(x,\nu_t))\zeta_t^{(\mu)}\right |^2\!\!dx\\
    & \leq^{\eqref{ineq gammma}}C\cW_2^2(\mu_t,\nu_t)|\zeta_t^{(\mu)}|_{L^2(\T;\R)}^2+(|V^{''}|_{L^{\infty}(\T;\R)}^2+M^2|F^{''}|_{L^{\infty}(\T;\R)}^2 )|\zeta_t^{(\mu)}-\zeta_t^{(\nu)}|_{L^2(\T;\R)}^2\\
    & \leq^{\eqref{stima norma L2 zetamu}} C\cW_2^2(\mu_t,\nu_t)|\zeta_0|_{L^2(\T;\R)}^2e^{(|V^{''}|_{L^{\infty}(\T;\R)}+M|F^{''}|_{L^{\infty}(\T;\R)})t}\\
    & +(|V^{''}|_{L^{\infty}(\T;\R)}^2+M^2|F^{''}|_{L^{\infty}(\T;\R)}^2 )|\zeta_t^{(\mu)}-\zeta_t^{(\nu)}|_{L^2(\T;\R)}^2.
\end{align*}
From Gronwall's lemma we obtain  
\begin{equation}\label{h0}
        |\zeta_t^{(\mu)}-\zeta_t^{(\nu)}|_{L^2(\T;\R)}^2 \leq C|\zeta_0|_{L^2(\T;\R)}^2e^{\left [\left(|V^{''}|_{L^{\infty}(\T;\R)}^2+1\right)+M^2\left(|F^{''}|_{L^{\infty}(\T;\R)}^2+1\right )\right ]t}\int_0^t \cW_2^2(\mu_s,\nu_s)\,ds.
\end{equation}
Similarly,  
\begin{align*}
    & \frac{d}{dt}|\pa_x\zeta_t^{(\mu)}-\pa_x\zeta_t^{(\nu)}|_{L^2(\T;\R)}^2 \leq  \left |\frac{d^4V}{dx^4}\right |_{L^{\infty}(\T;\R)}  |\zeta_t^{(\mu)}-\zeta_t^{(\nu)}|_{L^2(\T;\R)}^2+3|V^{''}|_{L^{\infty}(\T;\R)}|\pa_x\zeta_t^{(\mu)}-\pa_x\zeta_t^{(\nu)}|_{L^2(\T;\R)}^2\\
    & +\int_{\T} \left | (\Gamma_M(x,\mu_t)-\Gamma_M(x,\nu_t))\pa_x\zeta_t^{(\mu)}\right |^2\,dx+\itt \left|\Gamma_M(x,\nu_t)(\pa_x\zeta_t^{(\nu)}-\pa_x\zeta_t^{(\mu)})\right |^2\,dx\\
    & +\int_{\T} \left | (\pa_x\Gamma_M(x,\mu_t)-\pa_x\Gamma_M(x,\nu_t))\zeta_t^{(\mu)}\right |^2\,dx+\itt \left|\pa_x\Gamma_M(x,\nu_t)(\zeta_t^{(\nu)}-\zeta_t^{(\mu)})\right |^2\,dx\\
    & \leq^{\eqref{ineq gammma}} C\cW_2^2(\mu_t,\nu_t)(|\zeta_t^{(\mu)}|_{L^2(\T;\R)}^2+|\pa_x\zeta_t^{(\mu)}|_{L^2(\T;\R)}^2)+\left |\frac{d^4V}{dx^4} \right |_{L^{\infty}(\T;\R)}|\zeta_t^{(\mu)}-\zeta_t^{(\nu)}|_{L^2(\T;\R)}^2\\
    & + M^2|F^{''}|_{L^{\infty}(\T;\R)}^2 |\zeta_t^{(\mu)}-\zeta_t^{(\nu)}|_{L^2(\T;\R)}^2+(3|V^{''}|_{L^{\infty}(\T;\R)}+M^2|F^{'}|_{L^{\infty}(T;\R)}^2)|\pa_x\zeta_t^{(\mu)}-\pa_x\zeta_t^{(\nu)}|_{L^2(\T;\R)}^2 \\
\end{align*}
Once this is place, from \eqref{stima norma L2 zetamu}, \eqref{stima buona pax zeta}, \eqref{h0} and again Gronwall's lemma we obtain $(iii)$.
\end{proof}

\begin{proof}[Proof of $v)$ in Lemma \ref{stime 2}]
A calculation similar to the one in the proof of $(iii)$ in Lemma \ref{stime} gives 
\begin{align*}
    & \frac{d}{dt}|\zeta_t^{M,\ka}-\zeta_t^{\ka}|_{L^2(\T;\R)}^2 \leq |V^{''}|_{L^{\infty}(\T;\R)}|\zeta_t^{M,\ka}-\zeta_t^{\ka}|_{L^2(\T;\R)}^2+\int_{\T} |\Gamma_M(x,\mu_t^{M,\ka})\zeta_t^{M,\ka}-\Gamma(x,\mu_t^{\ka})\zeta_t^{\ka}|^2\,dx \\
    & \leq |V^{''}|_{L^{\infty}(\T;\R)}\,|\zeta_t^{M,\ka} - \zeta_t^{\ka}|_{L^2(\T;\R)}^2 + \sup \limits_{x \in \T} \,\,|\Gamma_M(x,\mu_t^{M,\ka})-\Gamma_M(x,\mu_t^{\ka})|^2|\zeta_t^{M,\ka}|_{L^2(\T;\R)}^2\\
    & +\sup \limits_{x\in \T}\,\,|\Gamma_M(x,\mu_t^{\ka})|^2|\zeta_t^{M,\ka}-\zeta_t^{\ka}|_{L^2(\T;\R)}^2+\sup \limits_{x\in \T}\,\,|\Gamma_M(x,\mu_t^{\ka})-\Gamma(x,\mu_t^{\ka})|^2|\zeta_t^{\ka}|_{L^2(\T;\R)}^2
\end{align*}
By now recalling \eqref{ineq gammma 2}, \eqref{second upper}, moment estimate \eqref{moments} (with $n=2$) and lastly $(ii)$ we obtain
\begin{align*}
    |\zeta_t^{M,\ka}-\zeta_t^{\ka}|_{L^2(\T;\R)}^2 & \leq C(t,\eta) \bigg( |\zeta_t^{M,\ka}-\zeta_t^{\ka}|_{L^2(\T;\R)}^2+\cW_2^2(\mu_t^{M,\ka},\mu_t^{\ka})\\
    & + \sup \limits_{s \in [0,T]}\,\int_{\T \times \{|a|>M\}} |a|^2\mu_s^{\ka}(dxda) \bigg ) \, ,
\end{align*}
where we recall $C(t,\eta)$ is a constant depending on $t \in [0,T]$ and $\eta=\min \zeta_0(x)$ (there is no dependence on $M \in \N$ since the estimates used are uniform in $M \in \N$). From Gronwall's inequality we have 
\begin{equation}\label{kuos4}
        |\zeta_t^{M,\ka}-\zeta_t^{\ka}|_{L^2(\T;\R)}^2 \leq C(t,\eta) \left ( \int_0^t \cW_2^2(\mu_s^{M,\ka},\mu_s^{\ka})\,ds+ \sup \limits_{s \in [0,T]}\,\int_{\T \times \{|a|>M\}} |a|^2\mu_s^{\ka}(dxda) \right ).
\end{equation}    
In a similar way, one can see that 
\begin{equation}\label{kuos5}
    |\pa_x \zeta_t^{M,\ka}-\pa_x \zeta_t^{\ka}|_{L^2(\T;\R)}^2 \leq \Bar{C}(t,\eta) \left (\int_0^t \cW_2^2(\mu_s^{M,\ka},\mu_s^{\ka})\,ds+ \sup \limits_{s \in [0,T]}\,\int_{\T \times \{|a|>M\}} |a|^2\mu_s^{\ka}(dxda) \right ),
\end{equation}
for $t \in [0,T]$. Combining \eqref{kuos4} and \eqref{kuos5} gives $(v)$ and this concludes the proof.     
\end{proof}

\begin{proof}[Proof of limit \eqref{remainder goes to zero BM}]
We are aiming at proving \eqref{remainder goes to zero BM} under the condition \eqref{eigenvalues constraint}.
Let us begin with denoting by $S_N$ the following quantity:
        \begin{equation*}
            S_N:= \sup \limits_{t \in [0,T]} \,\, \sum \limits_{|z| \geq N} |z|^2 \lambda_z^2 \left |\int_0^t e^{-(t-s)|z|^2}dw_s^z \right |^2,\qquad N \in \N.
        \end{equation*}
    Clearly, $\{ S_N \}_{N \in \N}$ is a non-increasing sequence and, therefore, it admits a limit $\mP$-a.s. Let be such a limit be denoted by $\bar S$. Clearly, we have $\bar S \geq 0$, $\mP$-a.s. so that if we showed $\bar S=0$, $\mP$-a.s. then \eqref{remainder goes to zero BM} would easily follow. Hence, we turn our attention to proving that $\bar S=0$, $\mP$-a.s. 
    
    By following the steps of the proof of \cite[Theorem 2.9]{dapra04} (cfr. \cite[Theorem 2.13]{dapra04} as well) along with the Factorization lemma \cite[Lemma 2.7, pp.20-21]{dapra04} if we set $m$ in \cite[Theorem 2.13]{dapra04} such that $m > \frac{1}{2\delta}$ ($\delta$ is the parameter appearing in \eqref{eigenvalues constraint}) we obtain
    \begin{align*}
        \mE |S_N|^m & \leq C_{T,m} \left [ \,\,\, \sum \limits_{|z| \geq N} |z|^2\lambda_z^2 \int_0^T s^{-2\delta} e^{-2s|z|^2}\,ds \right ]^m \\
        & \leq C_{T,m} \left [ \,\,\, \sum \limits_{|z| \geq N} \lambda_z^2 |z|^{4\delta} \int_0^{\infty} s^{-2\delta} e^{-s}\,ds \right ]^m\\
        & \leq C_{T,m,\delta} \left [ \,\,\, \sum \limits_{|z| \geq N} \lambda_z^2 |z|^{4\delta} \right ]^m .
    \end{align*}
    where $C_{T,m,\delta}$ is constant positive depending on $T$, $m$ and $\delta$.

On the one hand, one can see that from the above computation $|S_N|^m \leq |S_0|^m \in L^1(\Omega;\R)$ so that from Lebesgue's dominated convergence theorem we obtain $\mE |S_N|^m \xrightarrow[]{N \to \infty} \mE \,\,\bar {S}^m$. On the other hand, from again the above computation but this time combined with \eqref{eigenvalues constraint} gives $\mE |S_N|^m \xrightarrow[]{N \to \infty} 0$. The uniqueness of the limit yields $\mE \,\, \bar {S}^m =0$ which implies $\bar S=0$, $\mP$-a.s. This concludes the proof.
    %where $Y(T)$ is defined as
    %\begin{equation}
    %    Y(T):=\sum \limits_{|z| \geq N} \lambda_z^2 |z|^2 \left | \int_0^T e^{-(T-s)|z|^2}(T-s)^{-\delta}\,d\beta_s^z \right |^2 
    %\end{equation}
\end{proof}

\section*{Acknowledgments.} L.A. and M.O. have been supported by the Leverhulme grant RPG–2020–09. M.K. acknowledges the support of the EPSRC grant EP/V520044/1.

\bibliographystyle{abbrvnat}
\bibliography{bib}
\end{document}